\documentclass[ejs,preprint]{imsart}

\RequirePackage{amsthm,amsmath}
\RequirePackage[numbers]{natbib}
\RequirePackage[colorlinks,citecolor=blue,urlcolor=blue]{hyperref}
\RequirePackage{graphicx}
\RequirePackage{lscape}
\RequirePackage{scalerel}

\RequirePackage{amssymb}
\RequirePackage{enumerate}
\RequirePackage{enumitem}

\RequirePackage{multirow}
\RequirePackage{makecell}
\RequirePackage{booktabs}
\RequirePackage{float}

\RequirePackage{epstopdf}
\RequirePackage{graphicx}
\RequirePackage{overpic}
\RequirePackage{mathtools}

\RequirePackage{xcolor}

\newcommand{\C}{\mathbb{C}} 
\newcommand{\R}{\mathbb{R}}

\renewcommand{\hat}[1]{\widehat{#1}}

\renewcommand{\P}{\mathbb{P}}
\newcommand{\E}{\mathbb{E}}

\newcommand{\e}{\epsilon}

\newcommand{\cov}{\text{cov}}
\newcommand{\tr}{\textup{tr}}

\newcommand{\ttop}{^{\top}}

\newcommand{\ts}{\textstyle}

\newcommand{\var}{\operatorname{var}}
\newcommand{\x}{\mathbf{x}}
\renewcommand{\u}{\mathbf{u}}

\pubyear{2020}
\volume{0}
\issue{0}
\firstpage{1}
\lastpage{8}

\startlocaldefs
\numberwithin{equation}{section}
\theoremstyle{plain}

\endlocaldefs

\newtheorem{lemma}{Lemma}
\newtheorem{theorem}{Theorem}
\newtheorem{assumption}{Assumption}
\newtheorem{proposition}{Proposition}
\newtheorem{alg}{Algorithm}

\begin{document}

\begin{frontmatter}
\title{A Bootstrap Method for Spectral Statistics in High-Dimensional\\ Elliptical Models}

\runtitle{Bootstrap in High-Dimensional Elliptical Models}

\begin{aug}
\author{\fnms{Siyao} \snm{Wang}\ead[label=e1]{syywang@ucdavis.edu}}
\and
\author{\fnms{Miles E.} \snm{Lopes}\thanksref{t1}\ead[label=e2]{melopes@ucdavis.edu}}

\address{Department of Statistics, University of California, Davis\\
\printead{e1,e2}}



\thankstext{t1}{Supported in part by NSF Grant DMS 1915786}
\runauthor{S.~Wang \& M.~E.~Lopes}

\end{aug}

\begin{abstract}
    Although there is an extensive literature on the eigenvalues of high-dimensional sample covariance matrices, much of it is specialized to independent components (IC) models---in which observations are represented as linear transformations of random vectors with independent entries. By contrast, less is known in the context of elliptical models, which violate the independence structure of IC models and exhibit quite different statistical phenomena. In particular, very little is known about the scope of bootstrap methods for doing inference with spectral statistics in high-dimensional elliptical models. To fill this gap, we show how a bootstrap approach developed previously for IC models can be extended to handle the different properties of elliptical models. 
    Within this setting, our main theoretical result guarantees that the proposed method consistently approximates the distributions of linear spectral statistics,
    which play a fundamental role in multivariate analysis.
    We also provide empirical results showing that the proposed method performs well for a variety of nonlinear spectral statistics.
\end{abstract}

\begin{keyword}[class=MSC]
\kwd[Primary ]{	62H99 }
\kwd{62F40}
\kwd[; secondary ]{62H25}
\end{keyword}

\begin{keyword}
\kwd{high-dimensional statistics}
\kwd{bootstrap}
\kwd{covariance matrices}
\kwd{elliptical distributions}
\end{keyword}
\end{frontmatter}

\section{Introduction}

The analysis of spectral statistics of sample covariance matrices is a major research area within multivariate analysis, random matrix theory, high-dimensional statistics, and related fields~\citep{Bai:Silverstein:2010,yao2015sample,vershynin2018high,potters2020first}. If $\x_1,\dots,\x_n$ are centered i.i.d.~observations in $\R^p$ with a sample covariance matrix denoted by
\begin{equation}
    \hat\Sigma_n = \frac{1}{n}\sum_{i=1}^n \x_i\x_i\ttop,
\end{equation}
then we say that a random variable $T_n$ is a spectral statistic if it has the form $T_n=\psi(\lambda_1(\hat\Sigma_n),\dots,\lambda_p(\hat\Sigma_n))$, where
$\lambda_1(\hat\Sigma_n)\geq\cdots\geq \lambda_p(\hat\Sigma_n)$ are the sorted eigenvalues of $\hat\Sigma_n$, and $\psi$ is a generic real-valued function. Over the past two decades, there has been a tremendous growth of interest in spectral statistics in high-dimensional settings where $p=p(n)$ grows so that $p/n$ converges to a positive constant as $n\to\infty$. Likewise, statistical tools for approximating the distributions of spectral statistics have been applied to high-dimensional data in a broad range of domains, such as electrical engineering, finance, and biology~\citep[][]{Patterson:2006,oxford:handbook,Debbah:Couillet2011,edelman2013random}.

Although the research on spectral statistics has dealt with many different statistical models, two of the most influential ones have been \emph{elliptical} and \emph{independent components} (IC) models. To be specific, we say that the random vector $\x_1$ follows an elliptical model if it can be represented in the form 
\begin{equation}\label{eqn:ellip}
\x_1=\xi_1 \Sigma_n^{1/2} \mathbf{u}_1
\end{equation}
where $\Sigma_n=\E(\x_1\x_1\ttop)$, and \smash{$(\u_1,\xi_1)\in \R^{p}\times [0,\infty)$} is a random vector such that $\u_1$ and $\xi_1$ are independent, with $\u_1$ being uniformly distributed on the unit sphere of $\R^p$. Alternatively, we say that $\x_1$ follows an IC model if 
\begin{equation}
\x_1= \Sigma_n^{1/2}\mathbf{z}_1,
\end{equation}
 where $\mathbf{z}_1\in\R^p$ is a random vector whose entries are centered and independent. 

At first sight, these two models may seem to be very similar, but this outward appearance conceals some crucial differences in modelling capabilities. In particular, it should be stressed that the entries of the random vector $\xi_1\mathbf{u}_1$ in an elliptical model are correlated, which contrasts with the independence of the entries of $\mathbf{z}_1$ in an IC model. Also, since the scalar random variable $\xi_1$ is shared across all entries of $\x_1$ in an elliptical model, this enhances the ability to capture scenarios where the magnitudes of all entries of $\x_1$ move in the same direction simultaneously. This is a key effect in some application domains, such as in finance, where the entries of $\x_1$ correspond to stock prices that can  fall in tandem during a sharp market downturn. Additional background on related merits of elliptical models can be found in~\citep{Karoui:2009,Embrechts:2011,gupta2013elliptically}. More generally, the multivariate analysis literature has placed a longstanding emphasis on the benefits of elliptical models in fitting various types of non-Gaussian data~\citep{fang2018symmetric,anderson1958introduction,hardle2019applied}.

However, looking beyond the points just mentioned,
IC models have played a more dominant role than elliptical models in the literature on spectral statistics in high dimensions.
Consequently, the established body of high-dimensional limit theory is much less complete for elliptical models.
Indeed, the challenge of extending results from IC models to elliptical ones has become a prominent topic of ongoing research, which has led to important advances in the limit theory for spectral statistics~\citep[e.g.][]{li2018structure,hu2019high,hu2019aos,yang2021testing,jun2022tracy,li2022eigenvalues,zhang2022clt}. As a matter of historical context, it also worth bearing in mind that for some spectral statistics, it took many years for such extensions to be established.

From the standpoint of statistical methodology, a corresponding set of gaps exists between elliptical and IC models. These gaps are especially apparent in the current state of bootstrap methods for high-dimensional data. In particular, it is known from~\citep{lopes2019bootstrapping} that a form of parametric bootstrapping can successfully approximate the distributions of spectral statistics in IC models, whereas very little is known for elliptical models. 
Accordingly, our primary goal in the current paper is to resolve this issue by developing a parametric bootstrap method that is both theoretically and empirically effective in high-dimensional elliptical models.

With regard to theory, we will focus on the class of \emph{linear spectral statistics}, which have the form
\begin{equation}\label{eqn:lssdef}
    T_n(f) = \frac{1}{p}\sum_{j=1}^p f(\lambda_j(\hat\Sigma_n)),
\end{equation}
for a suitable real-valued function $f$. Beginning with the pathbreaking works~\citep{jonsson1982some,Bai:Silverstein:2004} that established the earliest versions of the central limit theorem for linear spectral statistics in high dimensions, these statistics have been a perennial focus of research. Their importance is underscored by the fact that they appear frequently throughout multivariate analysis, with some of the most well-known examples being $\frac 1p\tr(\hat\Sigma_n)$, $\frac 1p\tr(\hat\Sigma_n^2)$, and $\frac 1p\log\det(\hat\Sigma_n)$, among various other classical statistics for testing hypotheses~\citep{yao2015sample}. 

Motivated by these considerations, our main theoretical result (Theorem~\ref{thm:main}) shows that the proposed bootstrap method consistently approximates the distributions of linear spectral statistics when the underlying data are elliptical and $p/n$ converges to a positive constant as $n\to\infty$. The proof substantially leverages recent progress on the central limit theorem for linear spectral statistics in elliptical models due to~\citep{hu2019aos}. Also, an intermediate step in the proof (Lemma~\ref{lem:LW}) shows that the well-known eigenvalue estimation method QuEST is consistent in elliptical models---which may be of independent interest, since it seems that QuEST's consistency has not previously been reported outside of IC models~\citep{ledoit2015spectrum}. Moreover, Section~\ref{sec:rank} develops an application of Theorem~\ref{thm:main} where we establish the asymptotic validity of inference procedures related to the \emph{stable rank} parameter  $r_n=\tr(\Sigma_n)^2/\tr(\Sigma_n^2)$ of the population covariance matrix $\Sigma_n$.

To address the empirical performance of the proposed method, Section~\ref{sec:numerical} presents numerical results for a wide variety of model settings and statistics. Most notably, these results show encouraging performance for \emph{both linear and nonlinear spectral statistics}. (We regard any function of $(\lambda_1(\hat\Sigma_n),\dots,\lambda_p(\hat\Sigma_n))$ that is not of the form~\eqref{eqn:lssdef} as a nonlinear spectral statistic.)
To put this point into perspective, it is important to highlight the fact that asymptotic formulas for the distributions of nonlinear spectral statistics are typically developed on a case-by-case basis, and are relatively scarce in comparison to those for linear spectral statistics. Even when such formulas are available, they may be very different from those for linear spectral statistics---as they may require the estimation of different model parameters, or the implementation of different algorithms for numerical evaluation. On the other hand, the bootstrap approach may be considered more user-friendly, since it can be applied to different types of spectral statistics in an automatic and unified way.
Similarly, the bootstrap approach can provide the user with the freedom to easily explore statistics that depend on several linear spectral statistics in a complicated manner (which would otherwise require intricate delta-method calculations), or statistics for which formulas may not be available at all. \\
\\
\noindent\textbf{Notation and terminology.} For a random object $Y$, the expression $\mathcal{L}(Y)$ denotes its distribution, and $\mathcal{L}(Y|X)$ denotes its conditional distribution given the observations $\x_1,\dots,\x_n$. Similarly, we use $|X$ when referring to probabilities, expectations, and variances that are conditional on the observations. The symbols $\xrightarrow{\P}$ and $\Rightarrow$ respectively denote convergence in probability and convergence in distribution. 
For a set $A\subset\R^k$ and a number $\delta>0$, the outer $\delta$-neighborhood of $A$ is defined as $A^{\delta}=\{\mathbf{a}'\in\R^k| \inf_{\mathbf{a}\in A}\|\mathbf{a}'-\mathbf{a}\|_2\leq \delta\}$, where $\|\cdot\|_2$ is the Euclidean norm.
If $\mathbf{v}$ and $\mathbf{w}$ are random vectors in $\R^k$, then the L\'evy-Prohorov metric \label{page:dLP} $d_{\textup{LP}}(\mathcal{L}(\mathbf{v}),\mathcal{L}(\mathbf{w}))$ between their distributions is defined as the infimum over all numbers $\delta>0$ such that the inequality $\P(\mathbf{v}\in A)\leq \P(\mathbf{w}\in A^{\delta})+\delta$ holds for all Borel sets $A\subset\R^k$. If $\nu,\nu_1,\nu_2,\dots$ is a sequence of random probability distributions on $\R^k$, then the expression $\nu_n\overset{\P}{\Rightarrow}\nu$ means that the sequence of scalar random variables  $d_{\textup{LP}}(\nu_n,\nu)$ converges to 0 in probability as $n\to\infty$. For two sequences of non-negative real numbers $\{a_n\}$ and $\{b_n\}$, we write $a_n\lesssim b_n$ if there is a constant $C>0$ not depending on $n$ such that  $a_n\leq Cb_n$ holds for all large $n$. When both of the relations $a_n\lesssim b_n$ and $b_n\lesssim a_n$ hold, we write $a_n\asymp b_n$. The relation $a_n=o(b_n)$ means $a_n/b_n\to 0$ as $n\to\infty$, and the relation $a_n=\mathcal{O}(b_n)$ is equivalent to $a_n\lesssim b_n$. The $k\times k$ identity matrix is denoted as $I_k$, and indicator function for a condition $\cdots$ is denoted as $1\{\cdots\}$. Lastly, we use $\mathbb{C}^+$ to refer to the set of complex numbers with positive imaginary part.

\section{Method}\label{sec:bootstrapmethod}

Conceptually, the proposed method is motivated by the fact that the standard nonparametric bootstrap, based on sampling with replacement, often performs poorly when it is applied naively to high-dimensional data, unless special low-dimensional structure is available. (For additional background, see the papers~\citep{karoui2019,lopescoordinate,lopesyao}, as well as the numerical results presented here at the end of Section~\ref{sec:numlss}.)
This general difficulty can be understood by noting that sampling with replacement implicitly relies on the empirical distribution of the data as a substitute for the true data-generating distribution. In other words, the nonparametric bootstrap attempts to approximate a $p$-dimensional distribution in a fully non-parametric way, which can be challenging for even moderately large values of $p$. For this reason, alternative bootstrap methods that sample from parametric distributions have been advocated to improve upon the nonparametric bootstrap in high dimensions~\citep[e.g.][]{Mammen,lopes2019bootstrapping,Candes}, and this is the viewpoint that we pursue here. 

\subsection{Bootstrap algorithm}

At an algorithmic level, the proposed method is built on top of two estimators. The first is an estimator $\hat\varsigma_n^{\,2}$ for the variance parameter $\varsigma_n^2=\var(\xi_1^2)$. The second is an estimator $\tilde\Sigma_n$ for the diagonal matrix of population eigenvalues $\textup{diag}(\lambda_1(\Sigma_n),\dots,\lambda_p(\Sigma_n))$. Once these two estimators have been assembled, the method generates bootstrap data from an elliptical model parameterized in terms of $\hat\varsigma_n^{\,2}$ and $\tilde\Sigma_n$. More specifically, the $i$th sample of each bootstrap dataset $\x_1^*,\dots,\x_n^*$ is generated to be of the form
$$\x_i^*=\xi_i^*\tilde\Sigma^{1/2}_n\u_i^*,$$
where $\xi_i^*$ is a non-negative random variable satisfying $\E((\xi_i^*)^2|X)=p$ and $\var((\xi_i^*)^2|X)=\hat\varsigma_n^{\,2}$, and $\u_i^*\in\R^p$ is drawn uniformly from the unit sphere, independently of $\xi_i^*$. Then, a single bootstrap sample of a generic spectral statistic $T_n=\psi(\lambda_1(\hat\Sigma_n),\dots,\lambda_p(\hat\Sigma_n))$ is computed as $T_n^*=\psi(\lambda_1(\hat\Sigma_n^*),\dots,\lambda_p(\hat\Sigma_n^*))$, where $\hat\Sigma_n^*$ is the sample covariance matrix of the bootstrap data.

To emphasize the modular role that $\hat\varsigma_n^{\,2}$ and $\tilde\Sigma_n$ play in the bootstrap sampling process, we will provide the details for their construction later in Sections~\ref{sec:varest} and~\ref{sec:specest}. With these points understood, the following algorithm shows that the method is very easy to implement.

\begin{alg}[Bootstrap for spectral statistics]\label{alg:boot}
\normalfont
~\\[-0.2cm]
\hrule
\vspace{0.1cm}
\noindent \textbf{Input:} The number of bootstrap replicates $B$, as well as $\tilde{\Sigma}_n$ and $\hat{\varsigma}_n^2$.\\[0.2cm]
\textbf{For:} $b=1,\dots,B$ \textbf{ do in parallel }\\[-0.4cm]
\begin{itemize}[leftmargin=.5cm]
\item[1.]  Generate independent~random variables $g_1^*,\dots,g_n^*$ from a Gamma distribution with mean $p$ and variance $\hat\varsigma_n^2$, and then put $\xi^*_i=\sqrt{g^*_i}$ for $i=1,\dots,n$.\\[-0.2cm]

\item[2.] Generate independent random vectors $\mathbf{u}^*_1,\ldots,\mathbf{u}^*_n\in\R^p$ from the uniform distribution on the unit sphere.\\[-0.2cm]
\item[3.]  Compute  $\mathbf{x}^*_{i}=\xi^*_{i} \tilde{\Sigma}_n^{1/2} \mathbf{u}^*_{i}$ for $i=1,\dots,n$, and form $\hat\Sigma_n^*=\frac{1}{n}\sum_{i=1}^n \x_i^*(\x_i^*)\ttop$.\\[-0.2cm]
\item[4.] Compute the spectral statistic $T^*_{n,b}= \psi(\lambda_1(\hat\Sigma_n^*),\ldots,\lambda_p(\hat\Sigma_n^*))$.
\end{itemize}
\noindent \textbf{end for}\\[-0.2cm]

\noindent \textbf{Return:} The empirical distribution of $T^*_{n,1}, \ldots, T^*_{n,B}.$
\end{alg}
\vspace{-0.05cm}
\hrule
~\\[-0.5cm]
\\

\noindent\textbf{Remarks.} One basic but valuable feature of the algorithm is that it can be applied with equal ease to both linear and nonlinear spectral statistics. 
To comment on some more technical aspects of the algorithm, the Gamma distribution is used in step 1 because it offers a convenient way to generate non-negative random variables whose means and variances can be matched to any pair of positive numbers. (If the event $\hat\varsigma_n^{\,2}=0$ happens to occur, then the Gamma distribution in step 1 is interpreted as the point mass at $p$, so that $\xi_i^*=\sqrt p$ for all $i=1,\dots,n$.) Nevertheless, the choice of the Gamma distribution for generating $g_1^*,\dots,g_n^*$ is not required. Any other family of distributions on $[0,\infty)$ parameterized by means $\mu$ and variances $\sigma^2$, say $\{G_{\mu,\sigma^2}\}$, will be compatible with our bootstrap consistency result in Theorem~\ref{thm:main} if it satisfies the following two conditions: First, the pair $(\mu,\sigma^2)$ can be set to $(p,v)$ for any integer $p\geq 1$ and real number $v>0$. Second, there exists some fixed $\varepsilon>0$ such that the quantity $\int|\frac{t-p}{\sqrt p}|^{4+\varepsilon}dG_{p,\sigma^2}(t)$ remains bounded when $p$ diverges and $\sigma^2/p$ converges to a finite limit.
\subsection{Variance estimation}\label{sec:varest}

The estimation of the parameter $\varsigma_n^2=\var(\xi_1^2)$ is complicated by the fact that the random variables $\xi_1,\dots,\xi_n$ are not directly observable in an elliptical model. It is possible to overcome this challenge with the following estimating equation, which can be derived from an explicit formula  for $\var(\|\x_1\|_2^2)$ that is given in Lemma~\ref{lem:quadform}, 
\begin{equation}\label{eqn:esteqn}
    \varsigma_n^2 \ = \  p(p+2)\frac{\var(\|\mathbf{x}_1\|_2^2) - 2\tr(\Sigma_n^2)}{\tr(\Sigma_n)^2 + 2\tr(\Sigma_n^2)} +2p.
\end{equation}
Based on this equation, our approach is to separately estimate each of the three moment parameters on the right hand side, denoted as
 \begin{equation*}
     \begin{split}
          \alpha_n &= \tr(\Sigma_n^2)\\[0.2cm]
          \beta_n &= \var(\|\mathbf{x}_1\|_2^2)\\[0.2cm]
          \gamma_n &= \tr(\Sigma_n)^2.
     \end{split}
 \end{equation*}
 These parameters have the advantage that they can be estimated in a more direct manner, due to their simpler relations with the observations and the matrix $\Sigma_n$. Specifically, we use estimates defined according to 
\begin{align*}
    \hat{\alpha}_n &= \tr(\hat{\Sigma}_n^2) - \ts\frac{1}{n} \tr(\hat{\Sigma}_n)^2,\\
    \hat{\beta}_n & =\frac{1}{n-1}\sum_{i=1}^n \Big(\|\mathbf{x}_{i}\|_2^2 - \ts\frac{1}{n}\sum_{i'=1}^n \|\mathbf{x}_{i'}\|_2^2\Big)^2,\\
    \hat{\gamma}_n &= \tr(\hat{\Sigma}_n)^2.
\end{align*}
Substituting these estimates into~\eqref{eqn:esteqn} yields our proposed estimate for $\varsigma_n^2$
\begin{align}\label{eqn:varest}
   \hat{\varsigma}_n^{\,2} = \bigg( p(p+2)\frac{\hat{\beta}_n - 2\hat{\alpha}_n}{\hat{\gamma}_n + 2\hat{\alpha}_n} +2p\bigg)_{\!+}
\end{align}
where $x_+=\max\{x,0\}$ denotes the non-negative part of any real number $x$. The consistency of this estimate will be established in Theorem~\ref{thm:estimators}, which shows $\frac{1}{p}(\hat\varsigma_n^{\,2}-\varsigma_n^2)\to 0$ in probability as $n\to\infty$.

\subsection{Spectrum estimation}\label{sec:specest}
The problem of estimating the eigenvalues of a population covariance matrix has attracted long-term interest in the high-dimensional statistics literature, and many different estimation methods have been proposed ~\citep[e.g.][]{el2008spectrum,mestre2008improved,bai2010estimation, ledoit2015spectrum,kong2017spectrum}. In order to estimate $\lambda_1(\Sigma_n),\dots,\lambda_p(\Sigma_n)$ in our current setting, we modify the method of QuEST~\citep{ledoit2015spectrum}, which has become widely used in recent years. This choice has the benefit of making all aspects of our proposed method easy to implement, because QuEST is supported by a turnkey software package~\citep{Ledoit:numerical}.

We denote the estimates produced by QuEST as $\hat\lambda_{\textup{Q},1} \geq \dots \geq \hat{\lambda}_{\textup{Q},p}$, and we will use modified versions of them defined by
\begin{align}\label{eqn:bound}
    \tilde{\lambda}_{j} = \min(\hat{\lambda}_{\textup{Q},j}, \hat{b}_n),
\end{align}
for $j=1,\dots,p$, where we let $\hat{b}_n = \lambda_1(\hat\Sigma_n)+1$. The modification is done for theoretical reasons, to ensure that $\tilde\lambda_1,\dots,\tilde\lambda_p$ are asymptotically bounded, which follows from Lemma~\ref{lem:largest eigenvalues}. In addition, we define the diagonal $p\times p$ matrix associated with these estimates as  
\begin{align}\label{eqn:Quest}
    \tilde{\Sigma}_n = \textup{diag}(\tilde{\lambda}_{1},\ldots,\tilde{\lambda}_{p}).
\end{align}
Later, in Theorem~\ref{thm:estimators}, we will show that the estimates $\tilde\lambda_1,\dots,\tilde\lambda_p$ are consistent, in the sense that their empirical distribution converges weakly in probability to the correct limit as $n\to\infty$.

\section{Theoretical results}\label{sec:theory}
In this section, we present three theoretical guarantees for the proposed method. Theorem~\ref{thm:estimators} establishes appropriate notions of consistency for each of the estimators $\hat\varsigma_n^{\,2}$ and $\tilde\Sigma_n$. Second, our main result in Theorem~\ref{thm:main} shows that the bootstrap samples generated in Algorithm~\ref{alg:boot} consistently approximate the distributions of linear spectral statistics. Lastly, Theorem~\ref{thm:r} demonstrates the asymptotic validity of bootstrap-based inference procedures involving nonlinear spectral statistics.

\subsection{Setup}\label{sec:setup}
All of our theoretical analysis is framed in terms of a sequence of models indexed by $n$, so that all model parameters are allowed to vary with $n$, except when stated otherwise. The details of our model assumptions are given below in Assumptions~\ref{Data generating model} and~\ref{Regularity of spectrum}.

\begin{assumption}[Data generating model]
\label{Data generating model}
As $n\to\infty$, the dimension $p$ grows so that the ratio $c_n=p/n$ satisfies $c_n\to c$ for some positive constant $c$ different from 1. For each $i=1,\dots,n$, the observation $\x_i\in\R^p$ can be represented as
  \begin{equation}
      \x_i=\xi_i\Sigma_n^{1/2}\u_i,
  \end{equation}
  where $\Sigma_n\in\R^{p\times p}$ is a deterministic non-zero positive semidefinite matrix, and $(\u_1,\xi_1),\dots,$ $(\u_n,\xi_n)$ are i.i.d.~random vectors in $\R^p\times [0,\infty)$ satisfying the following conditions: The vector $\u_1$  is drawn from the uniform distribution on the unit sphere of $\R^p$, and is independent of $\xi_1$. In addition, as $n\to\infty$, the random variable $\xi_1^2$ satisfies $\E(\xi_1^2)=p$, as well as the conditions
 \begin{align} \label{eqn:moment assumption}
   \var\Big(\ts\frac{\xi_1^2-p}{\sqrt p}\Big)  =  \tau +o(1) \ \ \ \ \ \  \text{ and} \ \ \ \ \ \  \E\,\Big|\ts\frac{\xi_1^2-p}{\sqrt p}\Big|^{4+\varepsilon} \, \lesssim \, 1,
  \end{align}
 for some fixed constants $\tau \geq 0$ and $\varepsilon>0$ that do not depend on $n$.
\end{assumption}

\noindent\textbf{Remarks.}  With regard to the limiting value $c$ for the ratio $p/n$, the single case of $c=1$ is excluded so that we may employ certain facts about the QuEST method that were established in~\cite{ledoit2015spectrum}. From a practical standpoint, our proposed method can still be used effectively when $p/n=1$, as shown in Section~\ref{sec:numerical}. To address the moment conditions on $\xi_1^2$, a similar set of conditions was used in~\citep{hu2019aos} to establish a high-dimensional central limit theorem for linear spectral statistics. However, our condition involving a $4+\varepsilon$ moment for $(\xi_1^2-p)/\sqrt p$ replaces a corresponding $2+\varepsilon$ moment condition in that work. The extra bit of integrability is used here to show that the estimators  $\hat\varsigma_n^2$ and $\tilde\Sigma_n$ have suitable asymptotic properties for ensuring bootstrap consistency. 
\\

\noindent\textbf{Examples.} To illustrate that the conditions in~\eqref{eqn:moment assumption} cover a substantial range of situations, it is possible to provide quite a few explicit examples of distributions for $\xi_1^2$ that are conforming:
\begin{enumerate}
    \item 
    Chi-Squared distribution with $p$ degrees of freedom\\[-0.2cm]
    \item Poisson($p$)\\[-0.2cm]
    \item $(1-\tau)$Negative-Binomial$(p,1-\tau)$, \ for any $\tau \in (0,1)$\\[-0.2cm]
    \item Gamma$(p/\tau,1/\tau)$, \ for any $\tau>0$\\[-0.2cm]
    \item Beta-Prime$\left(\frac{p(1+p+\tau)}{\tau},\frac{1+p+2\tau}{\tau}\right)$, \ for any $\tau >0$\\[-0.2cm]
    \item Log-Normal$\left(\log(p)-\frac{1}{2}\log\big(1+\frac{\tau}{p}\big), \log\big(1+\frac{\tau}{p}\big) \right)$, \ for any $\tau > 0$ \\[-0.2cm]

    \item $(p+2\beta)$\,Beta$(p/2,\beta)$, \ for any $\beta>0$
\end{enumerate}
 It is also possible to give a more abstract class of examples that subsumes some of the previous ones as special cases. In detail, the conditions in~\eqref{eqn:moment assumption} will hold for any $\tau\geq0$ if $\xi_1^2= \sum_{j=1}^p z_{1j}^2$ for some independent random variables $z_{11},\dots,z_{1p}$ satisfying 
 \begin{equation}\label{eqn:suff}
 \small
         \ts\frac{1}{p}\sum_{j=1}^p \E(z_{1j}^2)=1, \quad \ts\frac{1}{p}\sum_{j=1}^p\var(z_{1j}^2) \to \tau, \text{\quad and \quad } \displaystyle\max_{1\leq j\leq p}\E|z_{1j}|^{8+2\varepsilon}\,\lesssim\, p^{1+\frac{\varepsilon}{4}}.
 \end{equation}
Further details for checking the validity of the previous examples, as well as explicit parameterizations, are provided in Appendix~\ref{app:examples}.

In addition to Assumption~\ref{Data generating model}, we need one more assumption dealing with the spectrum of the population covariance matrix $\Sigma_n$. To state this assumption, let $H_n$ denote the empirical distribution function associated with $\lambda_1(\Sigma_n),\dots,\lambda_p(\Sigma_n)$, which is defined for any $t\in \R$ according to
\begin{equation}\label{eqn:Hn}
    H_n(t)  \ = \ \frac{1}{p}\displaystyle\sum_{j=1}^p 1\{\lambda_j(\Sigma_n)\leq t\}.
\end{equation}
\begin{assumption}[Spectral structure]
\label{Regularity of spectrum}
There is a limiting spectral distribution $H$ such that as $n\to\infty$,
\begin{equation}
    H_n \ \Rightarrow \ H,
\end{equation}
where the support of $H$ is a finite union of closed intervals, bounded away from $0$ and $\infty$. Furthermore, there is a fixed compact interval in $(0,\infty)$ containing the support of $H_n$ for all large $n$.
\end{assumption}
\noindent\textbf{Remarks.} Variations of these conditions on population eigenvalues are commonly used throughout random matrix theory. This particular set of conditions was used in~\citep{ledoit2015spectrum} to establish theoretical guarantees for the QuEST estimation method in the context of IC models.

\subsection{Consistency of estimators}
Here, we establish the consistency of the estimators $\hat\varsigma_n^{\,2}$ and $\tilde\Sigma_n$, defined in~\eqref{eqn:varest} and~\eqref{eqn:Quest}. The appropriate notion of consistency for $\tilde\Sigma_n$ is stated in terms of its empirical spectral distribution function, which is defined for any $t\in\R$ as
\begin{equation}
    \tilde H_n(t)=\frac 1p\sum_{j=1}^p 1\{\tilde\lambda_j\leq t\}.
\end{equation}

\begin{theorem}
    \label{thm:estimators}
    Under Assumptions~\ref{Data generating model} and~\ref{Regularity of spectrum}, the following limits hold as $n \to \infty$,
    \begin{align}
       \ts\frac{1}{p}(\hat\varsigma_n^2-\varsigma_n^2)\ & \xrightarrow{\P} 0,\label{eqn:tauconsistency}\\[0.2cm]
    \tilde{H}_n &\xRightarrow{\P}H.\label{eqn:Hconsistency}
    \end{align}
\end{theorem}
\noindent\textbf{Remarks.} The limits~\eqref{eqn:tauconsistency} and~\eqref{eqn:Hconsistency} are proved in Appendices~\ref{sec:tauconsistency} and~\ref{sec:Hconsistency} respectively. Although these limits can be stated in a succinct form, quite a few details are involved in their proofs. For instance, the analysis of $\hat\varsigma_n^{\,2}$ is based on extensive calculations with polynomial functions of the quadratic forms $\|\x_i\|_2^2$ and $\x_i\ttop\x_j$ with $i\neq j$, as well as associated mixed moments. The consistency of $\tilde H_n$ is also notable because it requires showing the consistency of QuEST in elliptical models, and it seems that the consistency of QuEST has not previously been reported outside of IC models.

\subsection{Consistency of bootstrap}
To develop our main result on the consistency of the proposed bootstrap method, it is necessary to recall some background facts and introduce several pieces of notation.

Under Assumptions~\ref{Data generating model} and~\ref{Regularity of spectrum}, it is known from~\citep[Theorem 1.1]{Bai:Zhou:2008} that an extended version of the classical Mar\v{c}enko-Pastur Theorem holds for the empirical spectral distribution function $\hat H_n(t)=\frac{1}{p}\sum_{j=1}^p 1\{\lambda_j(\hat{\Sigma}_n)\leq t\}$. Namely, there is a probability distribution $\Psi(H,c)$ on $[0,\infty)$, depending only on $H$ and $c$, such that the weak limit $\hat H_n\Rightarrow \Psi(H,c)$ occurs almost surely. In this statement, we may regard $\Psi(\cdot,\cdot)$ as a map that takes a distribution $H'$ on $[0,\infty)$ and a number $c'>0$ as input, and returns a new distribution $\Psi(H'\!,c')$ on $[0,\infty)$ whose Stieltjes transform $m_{H'\!,c'}(z)=\int \frac{1}{\lambda-z}d(\Psi(H'\!,c'))(\lambda)$ solves the Mar\v{c}enko-Pastur equation~\eqref{eqn:MPdef} below.
 That is, for any $z\in\mathbb{C}^+$, the number $m_{H'\!,c'}=m_{H'\!,c'}(z)$ is the unique solution to the equation 
\begin{equation}\label{eqn:MPdef}
    m_{H'\!,c'}=\int \frac{1}{t(1-c'-c'zm_{H'\!,c'})-z}dH'(t)
\end{equation}
within the set $\{m\in\C\, | \, (c'-1)/z+c'm\in\mathbb{C}^+\}$.

The map $\Psi(\cdot,\cdot)$ is relevant to our purposes here, because it determines a centering parameter that is commonly used in limit theorems for linear spectral statistics.
In detail, if $H_{n,c_n}$ denotes a shorthand for the probability distribution $\Psi(H_n,c_n)$, and if $f$ is a real-valued function defined on the support of $H_{n,c_n}$, then the associated centering parameter is defined as
\begin{equation}\label{eqn:varthetadef}
    \vartheta_n(f) \ = \ \int f(t)dH_{n,c_n}(t).
\end{equation}
Similarly, let $\tilde H_{n,c_n}=\Psi(\tilde H_n,c_n)$ and $\tilde\vartheta_n(f)=\int f(t)d\tilde H_{n,c_n}(t)$. Also, in order to simplify notation for handling the joint distribution of several linear spectral statistics arising from a fixed set of functions $\mathbf{f}=(f_1,\dots,f_k)$, we write $T_n(\mathbf f)=(T_n(f_1),\dots,T_n(f_k))$, and likewise for $T_{n,1}^*(\mathbf{f})$, $\vartheta_n(\mathbf{f})$, and $\tilde\vartheta_n(\mathbf{f})$.

As one more preparatory item, recall from page~\pageref{page:dLP} that $d_{\textup{LP}}$ denotes the L\'evy-Prohorov metric for comparing distributions on $\R^k$. This metric is a standard choice for formulating bootstrap consistency results, as it has the fundamental property of metrizing weak convergence. 
\begin{theorem}\label{thm:main}
Suppose Assumptions~\ref{Data generating model} and~\ref{Regularity of spectrum} hold, and let $\mathbf{f}=(f_{1}, \ldots, f_{k})$ be a fixed set of real-valued functions that are analytic on an open subset of $\R$ containing $[0,\infty)$. Under these conditions, if $T_{n,1}^*(\mathbf{f})$ is generated as in Algorithm~\ref{alg:boot} using the estimators $\hat\varsigma_n^{\,2}$ and $\tilde\Sigma_n$ defined by~\eqref{eqn:varest} and~\eqref{eqn:Quest}, then the following limit holds as $n\to\infty$
\begin{equation}\label{eqn:main}
d_{\mathrm{LP}}\Big(\mathcal{L}\big(p\{T_{n}(\mathbf{f})-\vartheta_{n}(\mathbf{f})\}\big)\, ,\, 
\mathcal{L}\big(p\{T_{n, 1}^{*}(\mathbf{f})-\tilde{\vartheta}_{n}(\mathbf{f})\} \big| X\big)\Big)
\ \xrightarrow{\P}\ 0.
\end{equation}
\end{theorem} 
\noindent\textbf{Remarks.} The proof is given in Appendix~\ref{app:mainproof}, and makes key use of a recently developed central limit theorem for linear spectral statistics due to~\citep{hu2019aos}.  Regarding other aspects of the theorem, there are two points to discuss. First, the assumption that the functions $f_1,\dots,f_k$ are defined on an open set containing $[0,\infty)$ has been made for technical simplicity. In the setting where $p/n\to c>1$, this assumption is minor because $f_1,\dots,f_k$ must be defined at 0 due to the singularity of $\hat\Sigma_n$. Nevertheless, if $p/n\to c\in (0,1)$, then it is possible to show that a corresponding version of the theorem holds for analytic functions that are not defined at 0, and our numerical results confirm that the proposed method can successfully handle such cases.
Second, the quantities $\vartheta_n(\mathbf{f})$ and $\tilde\vartheta_n(\mathbf{f})$ are only introduced so that the distributions appearing in~\eqref{eqn:main} have non-trivial weak limits, which facilitates the proof. Still, the proposed method can be applied without requiring any particular type of centering.

\subsection{Application to inference on stable rank with guarantees}\label{sec:rank}
When dealing with high-dimensional covariance matrices, it is often of interest to have a measure of the number of ``dominant eigenvalues''. One such measure is the \emph{stable rank}, defined as
\begin{equation}\label{eqn:rn}
r_n = \frac{\ \tr(\Sigma_n)^2}{\tr(\Sigma_n^2)},
\end{equation}
which arises naturally in a plethora of situations~\citep[e.g.][]{Bai:1996,Srivastava2005some,Tang:2012,Lopes:2016,Spielman,Lopes:EJS}. Whenever $\Sigma_n$ is non-zero, this parameter satisfies $1\leq r_n\leq \textup{rank}(\Sigma_n)$, and the equality $r_n=p$ holds if and only if $\Sigma_n$ is proportional to the identity matrix.

In this subsection, we illustrate how the proposed bootstrap method can be applied to solve some inference problems involving the parameter $r_n$. Our first example shows how to construct a confidence interval for $r_n$, and our subsequent examples deal with testing procedures related to $r_n$. Later, in Theorem~\ref{thm:r}, we establish the theoretical validity of the methods used in these examples---showing that the confidence interval has asymptotically exact coverage, and that the relevant testing procedures maintain asymptotic control of their levels.

\subsubsection{Confidence interval for stable rank} 

Our confidence interval for $r_n$ is constructed using the estimator 
\begin{align}\label{eqn:rhatdef}
    \hat{r}_n = \frac{\tr(\hat\Sigma_n)^2}{\tr(\hat\Sigma_n^2) - \hat\Delta_n},
\end{align}
where we define
\begin{align*}
    \hat \Delta_n &= \frac{\tr(\hat\Sigma_n^2) }{n}\bigg[\frac{2(n-1)}{n}\frac{p^2+\hat\varsigma_n^2}{p(p+2)} - 1\bigg]\\[0.2cm]
    & \ \ \  + \frac{\tr(\hat\Sigma_n)^2}{n}\bigg[ \frac{n+1}{n}  +\frac{n-1}{n} \frac{{\hat\varsigma}_n^2 - 2p}{p(p+2)} - \frac{2(n-1)}{n^2} \frac{p^2+\hat\varsigma_n^2}{p(p+2)}\bigg],  
\end{align*}
and we set $\hat r_n$ equal to $n$ in the exceptional case that its denominator is 0.
It should be noted that $\hat r_n$ differs from the naive plug-in rule $\tr(\hat\Sigma_n)^2/\tr(\hat\Sigma_n^2)$, since the extra term $\hat\Delta$ in the denominator serves as a bias correction.

To proceed, let  $q_{1-\alpha}$ denote $(1-\alpha)$-quantile of the random variable \smash{$(\hat r_n-r_n)/p$} for any fixed $\alpha\in (0,1)$, and consider the interval $\mathcal{I}_n=[\hat{r}_n-p q_{1-\alpha/2}, \hat{r}_n-p q_{\alpha/2}]$. Whenever the distribution of $(\hat r_n-r_n)/p$ is continuous, this interval satisfies
\begin{equation}
    \P(r_n\in\mathcal{I}_n)=1-\alpha.
\end{equation}
However, the quantiles $q_{\alpha/2}$ and $q_{1-\alpha/2}$ are unknown, and so they must be estimated. This can be accomplished by generating bootstrap samples of the following form in Algorithm~\ref{alg:boot},

$$
 T^*_{n,1} = \frac{1}{p}\bigg(\frac{ \tr(\hat{\Sigma}^*_n)^2}{\tr((\hat{\Sigma}^*_n)^2) -
\hat{\Delta}_{n}^{^*}}
\, - \,
 \frac{ \tr(\tilde\Sigma_n)^2}{\tr(\tilde\Sigma_n^2)}\bigg),
$$
where $\hat\Delta_n^{^*}$ is defined by modifying the previous formula for $\hat\Delta_n$ so that $\hat\Sigma_n$ and $\hat\varsigma_n^2$ are replaced by versions computed with the bootstrap data $\x_1^*,\dots,\x_n^*$. Also, we define $T_{n,1}^*$ to be 0 in the exceptional case of a denominator being equal to 0.
 Letting $\hat q_{\alpha/2}$ and $\hat q_{1-\alpha/2}$ denote the respective $\alpha/2$ and $(1-\alpha/2)$-quantiles of $\mathcal{L}(T_{n,1}^*|X)$, the proposed confidence interval is defined as
 \begin{equation}\label{eqn:Ihatdef}
     \hat{\mathcal{I}}_n \ = \ \Big[ \hat{r}_n-p \hat q_{1-\alpha/2}\, , \, \hat{r}_n-p \hat q_{\alpha/2} \Big].
 \end{equation}
 Below, Theorem~\ref{thm:r} shows that as $n\to\infty$, the coverage probability $\P(r_n\in\hat{\mathcal{I}}_n)$ converges to $1-\alpha$, as desired.

\subsubsection{Hypotheses related to stable rank}\label{sec:testrank}
\noindent\textbf{Screening data for PCA.} Consider a scenario where a collection of different datasets are to be screened for further investigation by principal components analysis (PCA). In this situation, the datasets that should be discarded are the ones that cannot be well summarized by a moderate number of principal components. To put this in more quantitative terms, a dataset may be considered unsuitable for PCA if the stable rank $r_n$ exceeds a certain fraction of the full dimension $p$. That is, if $\frac{r_n}{p}>\e_0$ for some fixed reference value $\e_0\in (0,1)$. On the other hand, if $\frac{r_n}{p}\leq \e_0$, then the dataset may be retained. This leads to considering the hypothesis testing problem
\begin{equation}\label{eqn:H0}
\mathsf{H}_{0,n}: \ts\frac{r_n}{p} \leq \epsilon_0 \text{ \ \ \ \ \ \  vs. \ \ \ \ \ \  } \mathsf{H}_{1,n}: \ts\frac{r_n}{p} > \epsilon_0.
\end{equation}
To develop a testing procedure, we may again consider the $(1-\alpha)$-quantile $q_{1-\alpha}$ of the random variable $(\hat r_n-r_n)/p$. This quantile serves as a conceptual basis for a rejection criterion, because it satisfies the following inequality under the null hypothesis
\begin{equation}
    \P\Big(\ts\frac{\hat{r}_n}{p}-\e_0 > q_{1-\alpha}\Big) \ \leq \ \alpha.
\end{equation}
In other words, if $q_{1-\alpha}$ were known, then a level-$\alpha$ testing procedure would result from using $\frac{1}{p}\hat{r}_n-\epsilon_0> q_{1-\alpha}$ as a rejection criterion. Accordingly, a bootstrap-based version of this procedure rejects the null hypothesis when $\frac{1}{p}\hat{r}_n-\epsilon_0> \hat q_{1-\alpha}$, with the quantile estimate $\hat q_{1-\alpha}$ defined as before.\\
\\
\noindent \textbf{Testing for sphericity.} One more example of a testing problem related to $r_n$ is that of testing for sphericity, 
\small
\begin{equation}\label{eqn:sphericity}
    \mathsf{H}'_{0,n}: \Sigma_n=\sigma^2 I \text{ \ for some \ $\sigma^2>0$}  \text{ \ \ \     vs.  \ \ \     }  \mathsf{H}'_{1,n}: \Sigma_n\neq \sigma^2 I \text{ \  for any \ $\sigma^2>0$}. 
\end{equation}
\normalsize
The connection to the parameter $r_n$ arises from the fact that~\eqref{eqn:sphericity} is equivalent to the problem $\mathsf{H}'_{0,n}: \frac{r_n}{p}=1$ vs. $\mathsf{H}'_{1,n}: \frac{r_n}{p}<1$. This observation was used in the influential paper~\citep{Srivastava2005some} to develop a formula-based sphericity test. By analogy with our discussion of the problem~\eqref{eqn:H0}, this observation can also be used to develop a bootstrap-based sphericity test. However, there is one point of distinction in the current situation, which is that the eigenvalues of $\Sigma_n$ no longer need to be estimated. The reason is that under $\mathsf{H}_{0,n}'$, the scale-invariance of the statistic $(\hat r_n-r_n)/p$ causes it to behave as if $\Sigma_n=I_p$. Consequently, to estimate the quantiles of the null distribution of $(\hat r_n-r_n)/p$, Algorithm~\ref{alg:boot} can be run using $\tilde \Sigma_n=I_p$. To summarize, if we let $\hat q_{\alpha}'$ denote the resulting estimate for the $\alpha$-quantile of the null distribution of $(\hat r_n-r_n)/p$, then the rejection criterion is $\frac{\hat r_n}{p}-1\leq \hat q_{\alpha}'$.

The following result establishes the theoretical validity of the procedures discussed in this subsection.
\begin{theorem}\label{thm:r}
 If Assumptions \ref{Data generating model} and \ref{Regularity of spectrum} hold,  then as $n \to \infty$,
 \begin{equation}\label{eqn:Hlim1}
    \P\big(r_n\in\hat{\mathcal{I}}_n\big) \ \to \ 1-\alpha.
\end{equation}
Furthermore, if $\mathsf{H}_{0,n}$ holds for all large $n$, then 
\begin{equation}\label{eqn:Hlim2}
\P\Big(\ts\frac{\hat{r}_n}{p}-\epsilon_0> \hat{q}_{1-\alpha}\Big)
 \ \leq \  \alpha+o(1),
\end{equation}
or if $\mathsf{H}'_{0,n}$ holds for all large $n$, then 
\begin{equation}\label{eqn:Hlim3}
\P\Big(\ts\frac{\hat{r}_n}{p}-1\leq\hat{q}_{\alpha}'\Big)
 \ \to \ \alpha.
\end{equation}
\end{theorem}
\noindent\textbf{Remarks.} This result provides a notable complement to Theorem~\ref{thm:main}, because it demonstrates that bootstrap consistency can be established in tasks that are based on a nonlinear spectral statistic, namely $\hat r_n-r_n$. The proof of this result is given in Appendix~\ref{app:r}, where it can be seen that the limiting distribution of $\hat r_n-r_n$ has a very complicated dependence on the moments of $H$, due to the correlation between $\tr(\hat\Sigma_n)$ and $\tr(\hat\Sigma_n^2)$. 
In this way, the proof illustrates the utility of the bootstrap, since the bootstrap enables the user to completely bypass such complexity.

\section{Numerical results}\label{sec:numerical}

This section explores the empirical performance of the proposed bootstrap method in three different ways. Sections~\ref{sec:numlss} and~\ref{sec:numnlss} deal with bootstrap approximations for linear and nonlinear spectral statistics, while Section~\ref{sec:numr} looks at procedures for doing inference on the stable rank parameter $r_n$.

\subsection{Parameter settings}
In our simulations, we generated data from elliptical models that were parameterized as follows. The random variable $\xi_1^2$ was generated using four choices of distributions:
\begin{enumerate}[label=(\roman*).]
    \item Chi-Squared distribution with $p$ degrees of freedom \label{item:xi1}\\[-0.2cm]
    \item Beta-Prime$\left(\frac{p(1+p+8)}{8},\frac{1+p+16}{8}\right)$,\\[-0.2cm]
    \item $(p+4)\textup{Beta}(p/2,2)$,\\[-0.2cm]
    \item $\frac{9p}{10}\textup{F}(p,20)$,
\end{enumerate}
where F$(d_1,d_2)$ denotes an F-distribution with $d_1$ and $d_2$ degrees of freedom.
Note that cases (i), (iii), (iv), correspond respectively to a multivariate Gaussian distribution, a multivariate Pearson type II distribution, and a multivariate t-distribution with 20 degrees of freedom. Also, the numerical values appearing in (i)-(iv) were chosen to ensure the normalization condition $\E(\xi_1^2)=p$.

The population covariance matrix $\Sigma_n$ was selected from five options:
\begin{enumerate}[label=\arabic*.]
    \item The eigenvalues of $\Sigma_n$ are $\lambda_1(\Sigma_n)=\cdots=\lambda_5(\Sigma_n)=\frac 43$ and $\lambda_j(\Sigma_n)=1$ for $j\in\{6,\dots,p\}$. The $p\times p$ matrix of eigenvectors of $\Sigma_n$ is generated from the uniform distribution on orthogonal matrices. \label{item:M1}\\[-0.2cm]
    \item The eigenvalues of $\Sigma_n$ are $\lambda_j(\Sigma_n)=\exp(-j/4)$ for $j\in\{ 1,\dots,20\}$, and $\lambda_{20}(\Sigma_n)=\cdots=\lambda_p(\Sigma_n)$.
    The eigenvectors are the same as in case 1.\\[-0.2cm]
    \item The matrix $\Sigma_n$ has entries of the form $(\Sigma_n)_{ij} = (\frac{1}{10})^{|i-j|}+1\{i=j\}$.\\[-0.2cm]
    \item The matrix $\Sigma_n$ has entries of the form $(\Sigma_n)_{ij} = (\frac{1}{10}) 1\{i\neq j\} +1\{i=j\}$.\\[-0.2cm]
    \item The eigenvalues of $\Sigma_n$ are $\lambda_1(\Sigma_n)=5$ and $\lambda_j(\Sigma_n)=1$ for $j\in\{2,\dots,p\}$. The eigenvectors are the same as in case 1.
\end{enumerate}

\subsection{Linear spectral statistics}\label{sec:numlss}
Our experiments for linear spectral statistics were based on the task of using the proposed bootstrap to estimate three parameters of $\mathcal{L}(p(T_{n}(f)-\vartheta_{n}(f)))$: the mean, standard deviation, and 95th percentile. 
We considered two choices for the function $f$, namely  $f(x)=x^2$ and $f(x)=x-\log(x)-1$.
In the first case, we selected the ratio $p/n$ so that $p/n\in\{0.5, 1, 1.5\}$, and in the second case we used $p/n\in\{0.3, 0.5, 0.7\}$.
\\

\noindent\textbf{Design of experiments.} For each possible choice of $(\xi_1,\Sigma_n,p/n)$, we generated 5000 realizations of the dataset $\x_1,\dots,\x_n$, with $n=400$. These datasets allowed us to compute 5000 realizations of the statistic $p(T_{n}(f)-\vartheta_{n}(f))$, and we treated the empirical mean, standard deviation, and 95th percentile of these 5000 realizations as ground truth for our parameters of interest. In Tables~\ref{tab:x2} and~\ref{tab:log}, the ground truth values are reported in the first row of numbers corresponding to each choice of $\Sigma_n$.

With regard to the bootstrap, we ran Algorithm~\ref{alg:boot} on the first 500 datasets corresponding to each parameter setting. Also, we generated $B=250$ bootstrap samples of the form \smash{$p(T_{n,1}^*(f)-\tilde\vartheta_{n}(f))$} during every run. As a result of these runs, we obtained  500 different bootstrap estimates of the mean, standard deviation, and 95th percentile of $\mathcal{L}(p(T_{n}(f)-\vartheta_{n}(f)))$. In Tables~\ref{tab:x2} and~\ref{tab:log}, we report the empirical mean and standard deviation (in parenthesis) of these 500 estimates in the second row of numbers corresponding to each choice of $\Sigma_n$.

One more detail to mention is related to the computation of $\vartheta_n(f)$ and $\tilde \vartheta_n(f)$. For each parameter setting, we approximated $\vartheta_{n}(f)$ as follows. We averaged 30 realizations of 
$\frac{1}{40p}\sum_{j=1}^{40p}f(\lambda_j(\frac{1}{40n}\sum_{i=1}^{40n} \xi_i^2\mathsf{\Sigma}_n^{1/2}\u_i\u_i\ttop\mathsf{\Sigma}_n^{1/2}))$, 
where $\mathsf{\Sigma}_n=I_{40}\otimes \Sigma_n$ is of size $40p\times 40p$,  
each $\u_i$ was drawn from the uniform distribution on the unit sphere of $\R^{40p}$, and each $\xi_i^2$ was generated as in (i)-(iv), but with $40p$ replacing $p$. 
For the bootstrap samples, we computed one realization of the statistic \smash{$\frac{1}{40p}\sum_{j=1}^{40p}f(\lambda_j(\frac{1}{40n}\sum_{i=1}^{40n} (\xi_i^*)^2\tilde{\mathsf{\Sigma}}_n^{1/2}  \u_i^* (\u_i^*)\ttop \tilde{\mathsf{\Sigma}}_n^{1/2}))$} to approximate $\tilde \vartheta_n(f)$ during every run of Algorithm~\ref{alg:boot}, where $\tilde{\mathsf{\Sigma}}_n =I_{40}\otimes \tilde\Sigma_n$ is of size $40p\times 40p$, each $\u_i^*$ was drawn from the uniform distribution on the unit sphere of $\R^{40p}$, and each $(\xi_i^*)^2$ was drawn from a Gamma distribution with mean $40p$ and variance $40\hat\varsigma_n^2$.\\

\noindent\textbf{Comments on results.} It is easiest to explain the format of the tables with an example: The two entries in the upper right corner of Table~\ref{tab:x2} show that in settings (i) and 1 with $p/n=1.5$, the 95th percentile of $\mathcal{L}(p(T_n(f)-\vartheta_n(f)))$ with $f(x)= x^2$ is equal to 16.48, and the bootstrap estimate for the 95th percentile has a mean (standard deviation) of 16.73 (1.48). Table~\ref{tab:log} presents results for $f(x)=\log(x)-x-1$ in the same format.

In most settings, the bootstrap estimates perform well, with their bias and standard deviation being small in proportion to the parameter being estimated.
However, there are some specific parameter settings that require more attention. These settings involve choice $4$ for $\Sigma_n$, which is an equi-correlation matrix, and choice (iv) for $\xi_1$, which induces a multivariate t-distribution with 20 degrees of freedom. Notably, these choices correspond to settings that \smash{violate} Assumptions~\ref{Data generating model} and~\ref{Regularity of spectrum} of our theoretical results. In the case of the equi-correlation matrix, the bootstrap approximations for $f(x)= x^2$ are less accurate in comparison to other choices of $\Sigma_n$, due to increased variance. By contrast, if  $f(x)=\log(x)-x-1$, then the bootstrap approximations have similar accuracy across all choices of $\Sigma_n$ while holding other parameters fixed.

Lastly, in the case of the multivariate t-distribution, the bootstrap is able to accurately estimate the standard deviation of $\mathcal{L}(p(T_n(f)-\vartheta_n(f)))$ for both choices of $f$, but difficulties arise in estimating the mean and 95th percentile. To understand these mixed results, it is important to recognize the standard deviation does not depend on the centering parameter $\vartheta_n(f)$, whereas the mean and 95th percentile do. Also, the choice of $\vartheta_n(f)$ as a centering parameter is based on the CLT for linear spectral statistics established in~\citep{hu2019aos}, and the assumptions underlying that result are violated by the multivariate t-distribution.

\begin{table}[H]
\centering
\scriptsize
\setlength\tabcolsep{1pt}
\caption{Results for $p(T_n(f)-\vartheta_n(f))$ with $f(x)= x^2$ }
\resizebox{\textwidth}{!}{
\begin{tabular}{lllllllllllll}
\hline
                   &                          & \multicolumn{3}{c}{$p/n=0.5$} &  & \multicolumn{3}{c}{$p/n=1$} &  & \multicolumn{3}{c}{$p/n=1.5$}  \\ \cline{3-5} \cline{7-9} \cline{11-13}
$\xi_1$ & $\Sigma_n$ &mean   & sd  & 95th &&  mean   & sd  & 95th &&  mean   & sd  & 95th  \\ \hline
\specialrule{0em}{1pt}{1pt}
                   & \multirow{2}{*}{1}    &  0.51 &3.31 &6.02         &  &         0.85 &6.11 &10.87        &  &         1.44 &9.15 &16.48        \\
                   &                          &          0.48(0.24) &3.27(0.19) &5.85(0.53)         &  &         0.96(0.44) &6.09(0.35) &11.02(1)        &  &         1.47(0.65) &9.25(0.51) &16.73(1.48)        \\
                  \specialrule{0em}{3pt}{3pt}
                   & \multirow{2}{*}{2}       &          0 &0.11 &0.19        &  &         0.01 &0.11 &0.19        &  &         0 &0.11 &0.2        \\ 
                   &                          &          0(0.01) &0.11(0.01) &0.19(0.03)         &  &         0(0.01) &0.11(0.01) &0.19(0.02)        &  &         0(0.01) &0.11(0.01) &0.2(0.03)        \\  
                   \specialrule{0em}{3pt}{3pt}
\multirow{2}{*}{(i)}  & \multirow{2}{*}{3}       &          2.16 &12.88 &23.12         &  &         4.02 &24.12 &44.03        &  &         5.8 &36.13 &65.72       \\
                   &                          &          1.9(0.88) &12.85(0.73) &23.2(2.07)         &  &         3.94(1.75) &24.27(1.35) &43.95(3.82)        &  &         5.8(2.77) &37.08(2.2) &67.09(6.1) \\
                   \specialrule{0em}{3pt}{3pt}
                   & \multirow{2}{*}{4}       &            0.64 &63.35 &110.6          &  &           2.48 &241.3 &420        &  &           14.11 &540 &952.2       \\
                   &                          &            1.53(4.7) &62.9(9.6) &111.2(19.68)          &  &           4.95(17.04) &245(34.43) &433(70.15)         &  &           11.44(36.41) &541.9(78.44) &958.3(161.6)         \\
                   \specialrule{0em}{3pt}{3pt}
                   & \multirow{2}{*}{5}       &              0.55 &5.03 &9.1           &  &             1.03 &7.49 &13.4         &  &             1.57 &10.42 &18.58       \\
                   &                          &              0.55(0.34) &5.1(0.46) &9.15(1.05)          &  &             1.04(0.53) &7.41(0.48) &13.4(1.22)         &  &             1.55(0.74) &10.35(0.65) &18.61(1.7)       \\
                   \specialrule{0em}{1pt}{1pt}
                   \hline
                   \specialrule{0em}{1pt}{1pt}
                   & \multirow{2}{*}{1}       &          3.57 &6.35 &14.09         &  &         6.69 &11.93 &26.9       &  &         10.25 &18.09 &39.92        \\
                   &                          &          3.47(0.59) &6.36(0.45) &14.03(1.29)         &  &         6.83(1) &11.85(0.76) &26.51(2.32)        &  &         10.26(1.53) &17.99(1.11) &39.92(3.25)        \\
                   \specialrule{0em}{3pt}{3pt}
                   & \multirow{2}{*}{2}       &          0.01 &0.12 &0.21         &  &         0 &0.11 &0.2        &  &         0 &0.11 &0.19        \\ 
                   &                          &          0.01(0.01) &0.12(0.01) &0.2(0.03)         &  &         0.01(0.01) &0.11(0.01) &0.2(0.03)        &  &         0(0.01) &0.11(0.01) &0.2(0.03)        \\ 
                   \specialrule{0em}{3pt}{3pt}
\multirow{2}{*}{(ii)} & \multirow{2}{*}{3}       &          14.16 &25.27 &56         &  &               27.59 &47.57 &105.8                   &  &         41.49 &71.66 &159.7        \\
                   &                          &         13.67(2.24) &25.12(1.76) &55.35(5.1)   &  &         26.94(3.72) &47.09(2.97) &104.9(8.33)        &  &         40.39(6.13) &71.59(4.5) &158.6(13.38)        \\
                   \specialrule{0em}{3pt}{3pt}
                   & \multirow{2}{*}{4}       &            2.26 &65.49 &114.9          &  &           9.27 &246.8 &435.7        &  &           23.51 &539.9 &970.1       \\
                   &                          &            4.7(4.65) &66.35(9.67) &119.7(20.26)          &  &           9.81(16.29) &248.2(38.15) &441.7(75.8)         &  &           19.95(37.73) &548.5(83.74) &971.1(172)         \\
                   \specialrule{0em}{3pt}{3pt}
                   & \multirow{2}{*}{5}       &              3.75 &7.75 &16.66           &  &             7.21 &13.01 &28.58         &  &             10.46 &18.87 &42.02        \\
                   &                          &              3.59(0.65) &7.86(0.62) &16.69(1.58)           &  &             7.03(1.05) &12.87(0.84) &28.29(2.43)         &  &             10.4(1.51) &18.83(1.1) &41.45(3.22)        \\
                   \specialrule{0em}{1pt}{1pt} \hline
                   \specialrule{0em}{1pt}{1pt}
                   & \multirow{2}{*}{1}      &          -0.46 &1.13 &1.44         &  &         -0.95 &2.13 &2.61        &  &         -1.47 &3.12 &3.71        \\
                   &                          &          -0.47(0.08) &1.14(0.07) &1.4(0.17)         &  &         -0.96(0.14) &2.15(0.11) &2.58(0.33)        &  &         -1.43(0.21) &3.17(0.15) &3.81(0.45)        \\
                   \specialrule{0em}{3pt}{3pt}
& \multirow{2}{*}{2}       &          0.01 &0.11 &0.19         &  &         0 &0.11 &0.19        &  &         0 &0.11 &0.19        \\ 
                   &                          &          0(0.01) &0.11(0.01) &0.19(0.03)         &  &         0(0.01) &0.11(0.01) &0.19(0.02)        &  &         0(0.01) &0.11(0.01) &0.19(0.03)        \\ 
                   \specialrule{0em}{3pt}{3pt}
\multirow{2}{*}{(iii)} & \multirow{2}{*}{3}       &          -1.93 &4.49 &5.41         &  &         -3.89 &8.66 &10.62        &  &         -5.62 &13 &15.33        \\
                   &                          &          -1.85(0.33) &4.56(0.25) &5.68(0.69)         &  &         -3.78(0.57) &8.69(0.48) &10.57(1.26)        &  &         -5.7(0.9) &12.87(0.68) &15.52(1.88)        \\
                   \specialrule{0em}{3pt}{3pt}
                   & \multirow{2}{*}{4}       &            1.61 &63.23 &110          &  &         6.1 &242 &420.7        &  &           -6.78 &541.8 &944.3      \\
                   &                          &            0.75(4.21) &63.54(9.08) &111.2(18.37)          &  &         3.06(16.56) &242.5(37.18) &423.6(74.62)        &  &           5.52(35.54) &530.5(80.96) &928.1(161.8)         \\ 
                   \specialrule{0em}{3pt}{3pt}
                   & \multirow{2}{*}{5}       &              -0.52 &3.78 &5.98           &  &             -0.9 &4.45 &6.74         &  &             -1.3 &5.35 &7.75        \\
                   &                          &              -0.48(0.25) &3.83(0.51) &6.14(1.02)           &  &             -0.94(0.32) &4.53(0.55) &6.81(1.17)         &  &             -1.4(0.36) &5.38(0.53) &7.67(1.15)        \\
                   \specialrule{0em}{1pt}{1pt}
                   \hline
                   \specialrule{0em}{1pt}{1pt}
                   & \multirow{2}{*}{1}       &            0.67 &12.57 &21.94          &  &           1.19 &33.98 &58.49        &  &          0.77 &64.69 &114.1        \\
                   &                          &          13.02(2.03) &12.28(1.26) &33.54(4.18)          &  &           50.46(7.66) &33.11(3.74) &105.7(13.79)        &  &          111.2(15.79) &62.29(6.76) &215.5(26.71)         \\
                   \specialrule{0em}{3pt}{3pt}
                    & \multirow{2}{*}{2}       &          0.01 &0.13 &0.24          &  &           0.01 &0.14 &0.25        &  &          0 &0.14 &0.23         \\
                   &                          &          0.01(0.01) &0.13(0.02) &0.24(0.04)          &  &           0.02(0.01) &0.14(0.02) &0.25(0.04)        &  &          0.02(0.01) &0.14(0.02) &0.25(0.04)         \\
                   \specialrule{0em}{3pt}{3pt}
\multirow{2}{*}{(iv)} & \multirow{2}{*}{3}       &          0.68 &48.68 &82.62          &  &           1.3 &136.4 &236.7       &  &          15.26 &260 &457.1         \\
                   &                          &         51.21(8.4) &48.62(5.03) &132.7(17.04)          &  &           198.5(28.4) &131.5(13.71) &418.8(51.54)        &  &          442.9(64.24) &250.9(28.24) &861.1(108)         \\
                   \specialrule{0em}{3pt}{3pt}
                   & \multirow{2}{*}{4}       &          2.75 &73.19 &135.4          &  &             17.25 &281.2 &506.3        &  &         27.98 &603.8 &1097         \\
                   &                          &         14.03(5.46) &73.41(12.32) &142.2(25.81)          &  &             54.28(20.28) &277.7(44.26) &538.9(90.87)         &  &          123.8(46.62) &622.4(97.37) &1217(206.5)         \\
                   \specialrule{0em}{3pt}{3pt}
                   & \multirow{2}{*}{5}       &          0.42 &14.02 &23.86          &  &             2.62 &35.23 &62.08        &  &          -0.69 &66.26 &110.7         \\
                   &                          &         13.2(2.08) &13.82(1.41) &36.31(4.4)          &  &             50.71(7.33) &34.45(3.39) &108.3(12.79)         &  &          111.5(15.1) &64.04(6.65) &219(26.3)         \\
                   \specialrule{0em}{1pt}{1pt} \hline
\end{tabular}
}
\label{tab:x2}
\end{table}

\begin{table}[H]
\centering
\footnotesize
\setlength\tabcolsep{3pt}
\caption{Results for $p(T_n(f)-\vartheta_n(f))$ with $f(x) = x - \log(x)-1$}
\resizebox{\textwidth}{!}{
\begin{tabular}{lllllllllllll}
\hline
                   &                          & \multicolumn{3}{c}{$p/n=0.3$} &  & \multicolumn{3}{c}{$p/n=0.5$} &  & \multicolumn{3}{c}{$p/n=0.7$}  \\ \cline{3-5} \cline{7-9} \cline{11-13} 
$\xi_1$ & $\Sigma_n$ & mean   & sd  & 95th &&  mean   & sd  & 95th &&  mean   & sd  & 95th  \\ \hline
\specialrule{0em}{1pt}{1pt}
                   & \multirow{2}{*}{1}      &          0.18 &0.35 &0.77  &&   0.34 &0.64 &1.4        &  &         0.6 &1 &2.24       \\
                   &                          &          0.17(0.03) &0.35(0.02) &0.74(0.05)   &&    0.34(0.06) &0.63(0.03) &1.37(0.09)         &  &         0.58(0.09) &1.01(0.05) &2.25(0.14)        \\
                   \specialrule{0em}{3pt}{3pt}
\multirow{4}{*}{(i)} & \multirow{2}{*}{2}       &          0.18 &0.83 &1.55         &  &         0.36 &1.16 &2.25        &  &         0.59 &1.53 &3.1        \\ 
                   &                          &          0.19(0.22) &0.79(0.26) &1.49(0.65)         &  &         0.36(0.45) &1.11(0.37) &2.19(1.05)        &  &         0.62(0.7) &1.51(0.44) &3.1(1.42)        \\  
                   \specialrule{0em}{3pt}{3pt}
                   & \multirow{2}{*}{3}       &          0.18 &0.86 &1.62        &  &         0.35 &1.19 &2.27   &  &         0.55 &1.58 &3.12        \\
                   &                          &          0.17(0.06) &0.86(0.05) &1.59(0.14)         &  &          0.33(0.09) &1.19(0.06) &2.3(0.19)      &  &         0.58(0.12) &1.56(0.08) &3.17(0.25)         \\
                   \specialrule{0em}{3pt}{3pt}
                   & \multirow{2}{*}{4}       &            0.16 &0.91 &1.66          &  &           0.32 &1.53 &2.92        &  &           0.57 &2.21 &4.29       \\
                   &                          &            0.17(0.09) &0.91(0.07) &1.71(0.18)          &  &           0.33(0.18) &1.54(0.12) &2.91(0.31)         &  &           0.58(0.31) &2.21(0.16) &4.29(0.49)        \\
                   \specialrule{0em}{3pt}{3pt}
                   & \multirow{2}{*}{5}       &              0.18 &0.44 &0.91           &  &             0.32 &0.7 &1.46         &  &             0.56 &1.04 &2.28        \\
                   &                          &              0.17(0.04) &0.45(0.03) &0.92(0.07)           &  &             0.34(0.06) &0.69(0.03) &1.48(0.11)         &  &             0.59(0.09) &1.05(0.05) &2.31(0.15)       \\
                   \specialrule{0em}{1pt}{1pt} \hline
                   \specialrule{0em}{1pt}{1pt}
                   & \multirow{2}{*}{1}       &          1 &0.37 &1.62         &  &     1.74 &0.65 &2.79        &  &         2.56 &1.02 &4.25         \\
                   &                          &         1.04(0.11) &0.37(0.02) &1.66(0.12)       &  &                   1.79(0.17) &0.66(0.03) &2.87(0.19)        &  &         2.64(0.22) &1.04(0.05) &4.35(0.25)\\
                   \specialrule{0em}{3pt}{3pt}
\multirow{4}{*}{(ii)} & \multirow{2}{*}{2}       &          1.04 &1.52 &3.54         &  &         1.74 &2.02 &5.06        &  &         2.54 &2.49 &6.54        \\ 
                   &                          &         1.05(0.4) &1.51(0.25) &3.54(0.81)         &  &         1.77(0.7) &2.01(0.35) &5.09(1.27)        &  &         2.54(0.96) &2.46(0.41) &6.59(1.63)        \\  
                   \specialrule{0em}{3pt}{3pt}
                   & \multirow{2}{*}{3}       &         1.02 &1.63 &3.78        &  &          1.78 &2.15 &5.34         &  &    2.57 &2.63 &6.94      \\
                   &                          &        1.04(0.16) &1.6(0.12) &3.68(0.36)        &  &          1.8(0.22) &2.11(0.14) &5.31(0.45)       &  &         2.61(0.28) &2.59(0.16) &6.88(0.53) \\ 
                   \specialrule{0em}{3pt}{3pt}
                   & \multirow{2}{*}{4}       &            1.02 &0.94 &2.59          &  &           1.69 &1.57 &4.26        &  &           2.58 &2.22 &6.3       \\
                   &                          &            1.04(0.16) &0.94(0.07) &2.63(0.24)         &  &           1.8(0.31) &1.58(0.12) &4.45(0.43)          &  &           2.63(0.49) &2.25(0.17) &6.4(0.63)        \\
                   \specialrule{0em}{3pt}{3pt}
                   & \multirow{2}{*}{5}       &              1.02 &0.47 &1.8           &  &             1.78 &0.71 &2.96         &  &             2.56 &1.06 &4.31       \\
                   &                          &              1.05(0.13) &0.48(0.03) &1.84(0.16)           &  &             1.8(0.17) &0.72(0.04) &2.98(0.2)         &  &             2.64(0.22) &1.08(0.05) &4.41(0.27)        \\
                   \specialrule{0em}{1pt}{1pt} \hline
                   \specialrule{0em}{1pt}{1pt}
                   & \multirow{2}{*}{1}      &         -0.11 &0.34 &0.45        &  &                   -0.13 &0.62 &0.9         &  &         -0.14 &1 &1.49        \\
                   &                          &         -0.11(0.02) &0.34(0.02) &0.45(0.05)        &  &                   -0.14(0.04) &0.62(0.03) &0.88(0.08)         &  &         -0.09(0.07) &1.01(0.04) &1.57(0.14)        \\
                   \specialrule{0em}{3pt}{3pt}
\multirow{4}{*}{(iii)} & \multirow{2}{*}{2}       &          -0.12 &0.36 &0.49         &  &         -0.14 &0.64 &0.93        &  &         -0.08 &1.01 &1.57         \\ 
                   &                          &          -0.04(0.12) &0.47(0.18) &0.74(0.43)         &  &         0.05(0.29) &0.85(0.28) &1.45(0.74)        &  &         0.18(0.45) &1.22(0.31) &2.19(0.95)      \\ 
                   \specialrule{0em}{3pt}{3pt}
                   & \multirow{2}{*}{3}       &          -0.11 &0.38 &0.52         &  &         -0.14 &0.65 &0.93     &  &             -0.07 &1.03 &1.66   \\
                   &                          &          -0.11(0.03) &0.38(0.02) &0.51(0.05)        &  &         -0.14(0.05) &0.65(0.03) &0.93(0.09)  &  &         -0.09(0.07) &1.03(0.05) &1.6(0.14)\\
                   \specialrule{0em}{3pt}{3pt}
                   & \multirow{2}{*}{4}       &          -0.09 &0.91 &1.43            &  &           -0.12 &1.55 &2.47        &  &          -0.14 &2.21 &3.64       \\
                   &                          &           -0.09(0.07) &0.9(0.08) &1.42(0.17)           &  &           -0.1(0.13) &1.54(0.11) &2.48(0.28)         &  &           0.01(0.21) &2.21(0.16) &3.7(0.4)         \\
                   \specialrule{0em}{3pt}{3pt}
                   & \multirow{2}{*}{5}       &              -0.11 &0.43 &0.62           &  &            -0.15 &0.67 &0.94         &  &             -0.15 &1.06 &1.57        \\
                   &                          &              -0.11(0.03) &0.44(0.03) &0.62(0.07)           &  &             -0.14(0.05) &0.68(0.03) &0.99(0.09)         &  &             -0.09(0.07) &1.04(0.05) &1.64(0.15)        \\
                   \specialrule{0em}{1pt}{1pt}
                   \hline
                   \specialrule{0em}{1pt}{1pt}
                   & \multirow{2}{*}{1}       &            0.19 &0.45 &0.95          &  &           0.37 &0.86 &1.81        &  &          0.68 &1.44 &3.15         \\
                   &                          &          2.35(0.28) &0.44(0.03) &3.07(0.31)          &  &           6.45(0.79) &0.84(0.06) &7.85(0.87)        &  &          12.53(1.51) &1.41(0.1) &14.87(1.64)         \\
                   \specialrule{0em}{3pt}{3pt}
                   & \multirow{2}{*}{2}       &          0.15 &2.15 &3.75          &  &           0.26 &3.41 &5.85        &  &          0.46 &4.77 &8.24         \\
                   &                          &          2.37(0.81) &2.2(0.31) &6.01(1.31)          &  &           6.37(1.52) &3.64(0.43) &12.36(2.23)        &  &          12.6(2.65) &5.16(0.56) &21.06(3.51)         \\
                   \specialrule{0em}{3pt}{3pt}
\multirow{2}{*}{(iv)} & \multirow{2}{*}{3}       &          0.13 &2.42 &4.3          &  &           0.22 &3.96 &6.84        &  &          0.57 &5.75 &10.49         \\
                   &                          &         2.37(0.34) &2.32(0.21) &6.25(0.7)          &  &           6.45(0.86) &3.8(0.35) &12.82(1.43)        &  &          12.67(1.58) &5.31(0.5) &21.49(2.33)        \\
                   \specialrule{0em}{3pt}{3pt}
                   & \multirow{2}{*}{4}       &          0.19 &1.01 &1.89          &  &             0.4 &1.7 &3.29        &  &          0.5 &2.53 &4.75         \\
                   &                          &         2.39(0.35) &0.99(0.09) &4.06(0.44)          &  &             6.46(0.93) &1.72(0.15) &9.37(1.08)         &  &          12.44(1.78) &2.52(0.2) &16.64(1.97)         \\
                   \specialrule{0em}{3pt}{3pt}
                   & \multirow{2}{*}{5}       &          0.19 &0.54 &1.1          &  &             0.35 &0.93 &1.92        &  &          0.62 &1.49 &3.12         \\
                   &                          &         2.36(0.3) &0.54(0.04) &3.25(0.35)          &  &             6.35(0.78) &0.9(0.07) &7.86(0.87)         &  &          12.55(1.55) &1.46(0.1) &14.97(1.7)         \\
                   \specialrule{0em}{1pt}{1pt} \hline
\end{tabular}
}
\label{tab:log}
\end{table}

\noindent\textbf{Breakdown of nonparametric bootstrap.} To highlight one of the motivations for our parametric bootstrap method, we close this subsection with some numerical results exhibiting the breakdown of the standard nonparametric bootstrap, based on sampling with replacement. For the sake of brevity, we focus only on the task of estimating the standard deviation of $p(T_n(f)-\vartheta_n(f))$ in a subset of the previous parameter settings, corresponding to $p/n=1/2$, and $(\xi_n,\Sigma_n)\in \{\text{(i)}\}\times \{1,2,3,4,5\}$. The standard deviation is of particular interest, because it clarifies that the breakdown does not depend on how the statistic is centered.  

The results are presented in Table~\ref{tab:npboot}, showing that the nonparametric bootstrap tends to overestimate the true standard deviation, often by a factor of 2 or more.
For comparison, the corresponding  estimates obtained from the proposed bootstrap method are much more accurate, as can be seen from Tables~\ref{tab:x2} and~\ref{tab:log}.
In addition, the nonparametric bootstrap can have difficulties with nonlinear spectral statistics in settings like those considered here. Numerical results along these lines can be found in the paper~\citep{karoui2019}.

\begin{table}[H]
\caption{Estimation of standard deviation of $p(T_n(f)-\vartheta_n(f))$ with the nonparametric bootstrap}
\begin{tabular}{lllll}
\hline\\[-0.3cm]
$p/n$ & $\xi_1$ &$\Sigma_n$ 
                           & \multicolumn{1}{c}{$f(x)= x^2$} & \multicolumn{1}{c}{$f(x) = x - \log(x)-1$}  \\
                           \specialrule{0em}{1pt}{1pt}
                           \hline
                           \specialrule{0em}{1pt}{1pt}
\multirow{2}{*}{0.5} & \multirow{2}{*}{(i)} & \multirow{2}{*}{1}          &  3.31&  0.64\\
                           &&&  9.1(0.45)&  6.08(0.27)\\
                           \specialrule{0em}{1pt}{1pt}
                           \hline
                           \specialrule{0em}{1pt}{1pt}
\multirow{2}{*}{0.5} & \multirow{2}{*}{(i)} & \multirow{2}{*}{2}          &  0.11&  1.16\\
                           &&& 0.11(0.01) & 6.11(0.28)\\
                           \specialrule{0em}{1pt}{1pt}
                           \hline
                           \specialrule{0em}{1pt}{1pt}
\multirow{2}{*}{0.5} & \multirow{2}{*}{(i)} & \multirow{2}{*}{3}          &  12.88&  1.19\\
                           &&&  35.72(1.83)& 6.23(0.31) \\
                           \specialrule{0em}{1pt}{1pt}
                           \hline
                           \specialrule{0em}{1pt}{1pt}
\multirow{2}{*}{0.5} & \multirow{2}{*}{(i)} & \multirow{2}{*}{4}          &  63.35&  1.53\\
                           &&&  66.65(10.63)&  6.25(0.3) \\
                           \specialrule{0em}{1pt}{1pt}
                           \hline
                           \specialrule{0em}{1pt}{1pt}
\multirow{2}{*}{0.5} & \multirow{2}{*}{(i)} & \multirow{2}{*}{5}          &  5.03&  0.7\\
                           &&& 10.35(0.63) & 6.11(0.27)\\
                           \hline
\end{tabular}
\label{tab:npboot}
\end{table}

\subsection{Nonlinear spectral statistics}\label{sec:numnlss}

This subsection looks at how well the proposed bootstrap handles nonlinear spectral statistics. The statistics under consideration here are the largest sample eigenvalue $\lambda_1(\hat\Sigma_n)$, and the leading eigengap \smash{$\lambda_1(\hat{\Sigma}_n) - \lambda_2(\hat{\Sigma}_n)$}. The underlying experiments for these statistics were designed in the same manner as in Section~\ref{sec:numlss}, and Tables~\ref{tab:max} and~\ref{tab:gap} display the results in the same format as Tables~\ref{tab:x2} and~\ref{tab:log}. To a large extent, the favorable patterns that were noted in the results for linear spectral statistics are actually enhanced in the results for nonlinear spectral statistics---in the sense that the bias and standard deviations of the bootstrap estimates are generally smaller here than before.
Also, in contrast to linear spectral statistics, the results for nonlinear spectral statistics are relatively unaffected by choice 4 for $\Sigma_n$.
Lastly, under choice (iv) for $\xi_1$, the accuracy for nonlinear spectral statistics is reduced in comparison to other choices of $\xi_1$. Nevertheless, the reduction in accuracy under choice (iv) is less pronounced here than it was in the context of linear spectral statistics.
This makes sense in light of the fact that the statistics $\lambda_1(\hat\Sigma_n)$ and $\lambda_1(\hat{\Sigma}_n) - \lambda_2(\hat{\Sigma}_n)$ do not involve the centering parameter $\vartheta_n(f)$ that was used for linear spectral statistics. (Recall from Section~\ref{sec:numlss} that the reduced accuracy for linear spectral statistics under choice (iv) appeared to be related to $\vartheta_n(f)$.)

\vspace{-.2cm}

\begin{table}[]\small
\setlength\tabcolsep{3pt}
\caption{Results for $\lambda_1(\hat{\Sigma}_n)$}
\resizebox{\textwidth}{!}{
\begin{tabular}{lllllllllllll}
\hline
                   &                          & \multicolumn{3}{c}{$p/n=0.5$} &  & \multicolumn{3}{c}{$p/n=1$} &  & \multicolumn{3}{c}{$p/n=1.5$}  \\ \cline{3-5} \cline{7-9} \cline{11-13} 
$\xi_1$ & $\Sigma_n$ & mean   & sd  & 95th &&  mean   & sd  & 95th &&  mean   & sd  & 95th  \\ \hline
\specialrule{0em}{1pt}{1pt}
                   & \multirow{2}{*}{1}      &          2.9 &0.06 &3         &  &         3.96 &0.06 &4.07        &  &         4.9 &0.06 &5.01        \\
                   &                          &          2.93(0.05) &0.06(0.01) &3.04(0.07)         &  &         3.99(0.04) &0.06(0.01) &4.1(0.06)        &  &         4.93(0.04) &0.07(0.01) &5.05(0.06)        \\
                   \specialrule{0em}{3pt}{3pt}
                   & \multirow{2}{*}{2}       &          0.8 &0.05 &0.89        &  &         0.8 &0.05 &0.89        &  &         0.81 &0.05 &0.9        \\ 
                   &                          &          0.8(0.05) &0.05(0.01) &0.89(0.06)         &  &         0.8(0.05) &0.05(0.01) &0.89(0.06)        &  &         0.81(0.05) &0.05(0.01) &0.9(0.06)        \\ 
                   \specialrule{0em}{3pt}{3pt}
\multirow{2}{*}{(i)}& \multirow{2}{*}{3}       &          5.76 &0.11 &5.95         &  &         7.91 &0.12 &8.12        &  &         9.81 &0.13 &10.03        \\
                   &                          &          5.81(0.08) &0.12(0.02) &6.02(0.11)         &  &         7.98(0.09) &0.13(0.02) &8.21(0.13)        &  &         9.87(0.1) &0.14(0.03) &10.12(0.14)        \\ 
                   \specialrule{0em}{3pt}{3pt}
                   & \multirow{2}{*}{4}       &            21.35 &1.47 &23.81          &  &           41.79 &2.87 &46.62        &  &           62.34 &4.31 &69.6       \\
                   &                          &            21.21(1.5) &1.47(0.13) &23.68(1.71)          &  &           41.8(2.73) &2.9(0.23) &46.7(3.09)         &  &           62.34(4.19) &4.3(0.35) &69.62(4.73)         \\
                   \specialrule{0em}{3pt}{3pt}
                   & \multirow{2}{*}{5}       &              5.62 &0.35 &6.21           &  &             6.25 &0.34 &6.81         &  &             6.88 &0.33 &7.44        \\
                   &                          &              5.61(0.33) &0.35(0.03) &6.2(0.38)           &  &             6.22(0.33) &0.34(0.03) &6.79(0.37)         &  &             6.87(0.33) &0.33(0.03) &7.43(0.38)        \\
                   \specialrule{0em}{1pt}{1pt} \hline
                   \specialrule{0em}{1pt}{1pt}
                   & \multirow{2}{*}{1}       &          2.96 &0.07 &3.08         &  &          4.02 &0.07 &4.14        &  &         4.96 &0.07 &5.08        \\
                   &                          &          3.03(0.07) &0.07(0.01) &3.15(0.09)         &  &          4.09(0.07) &0.07(0.01) &4.22(0.09)         &  &         5.04(0.06) &0.08(0.01) &5.17(0.08)        \\
                   \specialrule{0em}{3pt}{3pt}
                   & \multirow{2}{*}{2}       &          0.8 &0.06 &0.89         &  &         0.8 &0.05 &0.89        &  &         0.8 &0.05 &0.9        \\ 
                   &                          &          0.8(0.05) &0.05(0.01) &0.89(0.06)         &  &         0.8(0.06) &0.05(0.01) &0.89(0.06)        &  &         0.81(0.06) &0.05(0.01) &0.9(0.06)        \\ 
                   \specialrule{0em}{3pt}{3pt}
\multirow{2}{*}{(ii)}& \multirow{2}{*}{3}       &          5.88 &0.13 &6.11        &  &         8.04 &0.14 &8.26        &  &         9.93 &0.14 &10.16        \\
                   &                          &          6.03(0.13) &0.14(0.02) &6.27(0.16)         &  &         8.2(0.12) &0.15(0.02) &8.47(0.15)        &  &         10.1(0.12) &0.16(0.02) &10.37(0.16)        \\
                   \specialrule{0em}{3pt}{3pt}
                   & \multirow{2}{*}{4}       &            21.34 &1.51 &23.86          &  &           41.8 &2.92 &46.72        &  &           62.32 &4.3 &69.65       \\
                   &                          &            21.44(1.48) &1.51(0.13) &23.99(1.69)         &  &           41.87(3) &2.92(0.25) &46.79(3.38)         &  &          62.37(4.37) &4.34(0.38) &69.67(4.99)         \\
                   \specialrule{0em}{3pt}{3pt}
                   & \multirow{2}{*}{5}       &              5.65 &0.35 &6.24           &  &             6.28 &0.35 &6.86         &  &             6.9 &0.34 &7.48        \\
                   &                          &              5.67(0.36) &0.36(0.03) &6.27(0.41)           &  &             6.31(0.35) &0.35(0.03) &6.9(0.4)         &  &             6.91(0.35) &0.34(0.03) &7.48(0.4)        \\
                   \specialrule{0em}{1pt}{1pt} \hline
                   \specialrule{0em}{1pt}{1pt}
                   & \multirow{2}{*}{1}      &          2.88 &0.05 &2.97        &  &          3.94 &0.06 &4.04         &  &         4.88 &0.06 &4.99        \\
                   &                          &          2.89(0.03) &0.06(0.01) &2.98(0.05)         &  &          3.95(0.03) &0.06(0.01) &4.06(0.05)        &  &         4.89(0.03) &0.06(0.01) &5(0.05)        \\
                   \specialrule{0em}{3pt}{3pt}
                   & \multirow{2}{*}{2}       &          0.8 &0.05 &0.89         &  &         0.8 &0.05 &0.89        &  &         0.8 &0.05 &0.9        \\
                   &                          &          0.8(0.05) &0.05(0.01) &0.89(0.06)         &  &         0.8(0.05) &0.05(0.01) &0.89(0.06)        &  &         0.8(0.05) &0.05(0.01) &0.89(0.06)        \\
                   \specialrule{0em}{3pt}{3pt}
\multirow{2}{*}{(iii)} & \multirow{2}{*}{3}       &          5.72 &0.1 &5.9         &  &         7.88 &0.12 &8.08        &  &         9.77 &0.12 &9.99        \\
                   &                          &          5.74(0.07) &0.11(0.02) &5.94(0.11)         &  &         7.9(0.08) &0.12(0.02) &8.12(0.13)        &  &            9.79(0.08) &0.13(0.02) &10.01(0.11)     \\ 
                   \specialrule{0em}{3pt}{3pt}
                   & \multirow{2}{*}{4}       &            21.39 &1.47 &23.84          &  &           41.83 &2.88 &46.61        &  &           62.14 &4.34 &69.51       \\
                   &                          &            21.39(1.38) &1.47(0.12) &23.87(1.56)          &  &           41.74(2.95) &2.87(0.25) &46.57(3.35)         &  &           61.71(4.39) &4.25(0.37) &68.87(4.96)        \\
                   \specialrule{0em}{3pt}{3pt}
                   & \multirow{2}{*}{5}       &              5.61 &0.34 &6.17           &  &             6.24 &0.34 &6.82         &  &             6.87 &0.34 &7.44        \\
                   &                          &              5.59(0.34) &0.34(0.03) &6.17(0.38)           &  &             6.23(0.35) &0.34(0.03) &6.8(0.39)         &  &             6.88(0.34) &0.34(0.03) &7.45(0.39)       \\
                   \specialrule{0em}{1pt}{1pt}
                   \hline
                   \specialrule{0em}{1pt}{1pt}
                   & \multirow{2}{*}{1}       &            3.24 &0.18 &3.55          &  &           4.77 &0.42 &5.48        &  &          6.27 &0.67 &7.51         \\
                   &                          &         3.39(0.22) &0.12(0.03) &3.59(0.27)          &  &           5(0.46) &0.18(0.05) &5.31(0.54)        &  &          6.59(0.64) &0.25(0.07) &7.03(0.75)         \\
                   \specialrule{0em}{3pt}{3pt}
                   & \multirow{2}{*}{2}       &          0.8 &0.06 &0.9          &  &           0.8 &0.06 &0.91        &  &          0.81 &0.06 &0.91         \\
                   &                          &          0.81(0.06) &0.06(0.01) &0.91(0.07)          &  &           0.81(0.06) &0.06(0.01) &0.91(0.07)        &  &          0.81(0.06) &0.06(0.01) &0.91(0.07)         \\
                   \specialrule{0em}{3pt}{3pt}
\multirow{2}{*}{(iv)} & \multirow{2}{*}{3}       &          6.42 &0.36 &7.03          &  &           9.51 &0.81 &10.97        &  &          12.52 &1.33 &15.13         \\
                   &                          &         6.77(0.4) &0.24(0.06) &7.18(0.49)          &  &           10.17(0.81) &0.37(0.1) &10.82(0.97)        &  &          13.55(1.35) &0.54(0.16) &14.48(1.6)         \\
                   \specialrule{0em}{3pt}{3pt}
                   & \multirow{2}{*}{4}       &          21.45 &1.62 &24.28         &  &             42.01 &3.23 &47.49        &  &          62.59 &4.68 &70.57         \\
                   &                          &         21.52(1.71) &1.62(0.16) &24.26(1.96)          &  &             42.08(3.22) &3.17(0.29) &47.44(3.66)         &  &          63.05(4.62) &4.76(0.43) &71.15(5.29)         \\
                   \specialrule{0em}{3pt}{3pt}
                   & \multirow{2}{*}{5}       &          5.71 &0.39 &6.37          &  &             6.48 &0.4 &7.18        &  &          7.33 &0.51 &8.14         \\
                   &                          &         5.79(0.39) &0.39(0.04) &6.45(0.45)          &  &             6.66(0.42) &0.39(0.04) &7.33(0.49)         &  &          7.69(0.45) &0.39(0.05) &8.37(0.53)         \\
                   \specialrule{0em}{1pt}{1pt} \hline
\end{tabular}
}
\label{tab:max}
\end{table}

\begin{table}[H]\small
\setlength\tabcolsep{3pt}
\caption{Results for $ \lambda_1(\hat{\Sigma}_n) - \lambda_2(\hat{\Sigma}_n)$}
\resizebox{\textwidth}{!}{
	\begin{tabular}{lllllllllllll}
		\hline
		&                          & \multicolumn{3}{c}{$p/n=0.5$} &  & \multicolumn{3}{c}{$p/n=1$} &  & \multicolumn{3}{c}{$p/n=1.5$}  \\ \cline{3-5} \cline{7-9} \cline{11-13}
		$\xi_1$ & $\Sigma_n$ & mean   & sd  & 95th &&  mean   & sd  & 95th &&  mean   & sd  & 95th  \\ \hline
		\specialrule{0em}{1pt}{1pt}
		& \multirow{2}{*}{1}      &          0.09 &0.05 &0.18         &  &         0.1 &0.05 &0.19        &  &         0.1 &0.06 &0.21        \\
		&                          &          0.1(0.02) &0.06(0.01) &0.2(0.05)         &  &        0.1(0.02) &0.06(0.01) &0.21(0.04)        &  &         0.11(0.02) &0.06(0.01) &0.22(0.04)        \\
		\specialrule{0em}{3pt}{3pt}
		& \multirow{2}{*}{2}       &          0.18 &0.07 &0.3         &  &         0.18 &0.07 &0.29        &  &         0.18 &0.07 &0.29        \\
		&                          &          0.19(0.06) &0.06(0.01) &0.29(0.07)         &  &         0.18(0.06) &0.06(0.01) &0.29(0.07)        &  &         0.19(0.06) &0.06(0.01) &0.3(0.07)        \\
		\specialrule{0em}{3pt}{3pt}
\multirow{2}{*}{(i)}& \multirow{2}{*}{3}       &          0.18 &0.1 &0.35         &  &         0.19 &0.1 &0.38        &  &         0.2 &0.11 &0.41        \\
		&                          &          0.2(0.04) &0.11(0.02) &0.4(0.08)         &  &         0.22(0.05) &0.12(0.03) &0.44(0.1)        &  &         0.23(0.05) &0.13(0.03) &0.47(0.11)        \\ 
       \specialrule{0em}{3pt}{3pt}
       & \multirow{2}{*}{4}       &            18.78 &1.47 &21.26          &  &           38.26 &2.87 &43.07        &  &           57.95 &4.31 &65.22       \\
       &                          &            18.62(1.5) &1.47(0.13) &21.09(1.71)          &  &           38.24(2.74) &2.9(0.23) &43.15(3.09)        &  &           57.94(4.19) &4.3(0.35) &65.22(4.73)        \\
       \specialrule{0em}{3pt}{3pt}
                   & \multirow{2}{*}{5}       &              2.78 &0.35 &3.38          &  &             2.32 &0.35 &2.9         &  &             2.01 &0.34 &2.58        \\
                   &                          &              2.75(0.33) &0.35(0.03) &3.34(0.38)           &  &             2.28(0.33) &0.34(0.03) &2.85(0.37)         &  &             1.98(0.33) &0.34(0.03) &2.55(0.38)        \\
       \specialrule{0em}{1pt}{1pt} \hline
		\specialrule{0em}{1pt}{1pt}
		& \multirow{2}{*}{1}       &          0.1 &0.05 &0.2         &  &         0.1 &0.05 &0.2        &  &         0.11 &0.06 &0.21        \\
		&                          &          0.11(0.02) &0.06(0.01) &0.22(0.05)         &  &         0.11(0.02) &0.06(0.01) &0.23(0.05)        &  &         0.12(0.02) &0.07(0.01) &0.24(0.04)        \\
		\specialrule{0em}{3pt}{3pt}
		& \multirow{2}{*}{2}       &          0.18 &0.07 &0.29         &  &         0.18 &0.07 &0.29        &  &         0.18 &0.07 &0.29        \\
		&                          &          0.18(0.06) &0.06(0.01) &0.29(0.07)         &  &         0.19(0.06) &0.06(0.01) &0.29(0.07)        &  &         0.19(0.06) &0.06(0.01) &0.3(0.07)        \\
		\specialrule{0em}{3pt}{3pt}
\multirow{2}{*}{(ii)}& \multirow{2}{*}{3}       &          0.19 &0.1 &0.38         &  &         0.2 &0.11 &0.41        &  &         0.21 &0.11 &0.42        \\
		&                          &          0.22(0.04) &0.12(0.02) &0.44(0.09)         &  &         0.23(0.05) &0.13(0.03) &0.48(0.1)        &  &         0.24(0.04) &0.13(0.03) &0.49(0.1)        \\ 
       \specialrule{0em}{3pt}{3pt}
       & \multirow{2}{*}{4}       &            18.72 &1.5 &21.24          &  &           38.21 &2.92 &43.11        &  &           57.88 &4.3 &65.19       \\
       &                          &            18.76(1.48) &1.51(0.13) &21.3(1.68)          &  &           38.22(2.99) &2.92(0.25) &43.13(3.38)        &  &           57.87(4.37) &4.34(0.38) &65.17(4.98)         \\
       \specialrule{0em}{3pt}{3pt}
                   & \multirow{2}{*}{5}       &              2.74 &0.36 &3.34           &  &             2.29 &0.35 &2.88         &  &             1.97 &0.35 &2.55        \\
                   &                          &              2.68(0.36) &0.36(0.03) &3.29(0.41)           &  &             2.24(0.35) &0.35(0.03) &2.84(0.4)         &  &             1.89(0.35) &0.34(0.03) &2.47(0.4)        \\
       \specialrule{0em}{1pt}{1pt} \hline
		\specialrule{0em}{1pt}{1pt}
		& \multirow{2}{*}{1}      &                  0.09 &0.05 &0.17  &&         0.09 &0.05 &0.19        &  &         0.1 &0.06 &0.21        \\
		&                          &          0.09(0.02) &0.05(0.01) &0.19(0.04)         &  &         0.1(0.02) &0.05(0.01) &0.2(0.03)        &  &         0.11(0.01) &0.06(0.01) &0.21(0.03)        \\
		\specialrule{0em}{3pt}{3pt}
		& \multirow{2}{*}{2}       &          0.18 &0.07 &0.29         &  &         0.18 &0.07 &0.29        &  &         0.18 &0.07 &0.29        \\
		&                          &          0.19(0.06) &0.06(0.01) &0.3(0.07)         &  &         0.18(0.06) &0.06(0.01) &0.29(0.07)        &  &    0.18(0.06) &0.06(0.01) &0.29(0.07)             \\
		\specialrule{0em}{3pt}{3pt}
\multirow{2}{*}{(iii)} & \multirow{2}{*}{3}       &          0.17 &0.09 &0.34         &  &         0.19 &0.1 &0.38        &  &         0.2 &0.11 &0.41        \\
		&                          &          0.19(0.04) &0.1(0.02) &0.38(0.08)         &  &         0.21(0.05) &0.11(0.03) &0.42(0.1)        &  &         0.18(0.06) &0.06(0.01) &0.29(0.07)        \\ 
        \specialrule{0em}{3pt}{3pt}
       & \multirow{2}{*}{4}      &            18.84 &1.48 &21.29          &  &           38.31 &2.88 &43.1        &  &           57.77 &4.34 &65.18       \\
       &                          &            18.83(1.38) &1.47(0.12) &21.32(1.57)         &  &           38.21(2.95) &2.88(0.25) &43.05(3.35)         &  &           57.32(4.39) &4.25(0.37) &64.49(4.96)        \\
       \specialrule{0em}{3pt}{3pt}
                   & \multirow{2}{*}{5}       &              2.78 &0.34 &3.35           &  &                      2.34 &0.35 &2.92    & &         2.02 &0.34 &2.6        \\
                   &                          &              2.75(0.34) &0.35(0.03) &3.34(0.38)           &  &             2.31(0.35) &0.34(0.03) &2.9(0.39)         &  &             2.03(0.34) &0.34(0.03) &2.6(0.39)       \\
                   \specialrule{0em}{1pt}{1pt}
                   \hline
                   \specialrule{0em}{1pt}{1pt}
                   & \multirow{2}{*}{1}       &            0.17 &0.14 &0.39          &  &           0.31 &0.34 &0.91        &  &          0.5 &0.54 &1.5         \\
                   &                          &          0.17(0.08) &0.09(0.03) &0.34(0.14)          &  &           0.27(0.22) &0.14(0.06) &0.53(0.32)        &  &          0.4(0.33) &0.21(0.09) &0.78(0.47)         \\
                   \specialrule{0em}{3pt}{3pt}
                   & \multirow{2}{*}{2}       &          0.19 &0.07 &0.3          &  &           0.18 &0.07 &0.3        &  &          0.19 &0.07 &0.3         \\
                   &                          &          0.19(0.06) &0.07(0.01) &0.31(0.07)          &  &           0.19(0.06) &0.07(0.01) &0.3(0.08)        &  &          0.19(0.06) &0.07(0.01) &0.3(0.07)         \\
                   \specialrule{0em}{3pt}{3pt}
\multirow{2}{*}{(iv)} & \multirow{2}{*}{3}       &          0.32 &0.28 &0.75          &  &           0.62 &0.64 &1.73        &  &          0.99 &1.06 &2.97         \\
                   &                          &         0.35(0.17) &0.19(0.07) &0.69(0.28)          &  &           0.6(0.44) &0.3(0.13) &1.15(0.64)        &  &          0.98(0.87) &0.45(0.21) &1.8(1.19)         \\
                   \specialrule{0em}{3pt}{3pt}
                   & \multirow{2}{*}{4}       &          18.59 &1.61 &21.41          &  &             37.75 &3.22 &43.23        &  &          57.01 &4.69 &65         \\
                   &                          &         18.45(1.69) &1.61(0.16) &21.18(1.93)          &  &             37.46(3.21) &3.16(0.29) &42.78(3.62)         &  &          56.94(4.6) &4.74(0.43) &64.99(5.25)         \\
                   \specialrule{0em}{3pt}{3pt}
                   & \multirow{2}{*}{5}       &          2.54 &0.41 &3.21          &  &             1.79 &0.47 &2.53       &  &          1.24 &0.49 &2.03         \\
                   &                          &         2.4(0.39) &0.39(0.04) &3.06(0.45)          &  &             1.58(0.41) &0.4(0.05) &2.27(0.47)         &  &          1.09(0.37) &0.4(0.06) &1.78(0.45)         \\
       \specialrule{0em}{1pt}{1pt} \hline
	\end{tabular}
	}
\label{tab:gap}
\end{table}

\subsection{Inference on stable rank}\label{sec:numr}

\noindent\textbf{Confidence interval.} Table~\ref{tab:CI} presents numerical results on the width and coverage probability of the bootstrap confidence interval $\hat{\mathcal{I}}$ (defined in~\eqref{eqn:Ihatdef}) for the stable rank parameter $r_n$. The results were computed using experiments based on the same design as in Section~\ref{sec:numlss}, with every interval having a nominal coverage probability of $95\%$. Due to the fact that the width of the interval scales with the value of $r_n$, we report the width as a percentage of $r_n$. To illustrate a particular example, the upper right corner of Table~\ref{tab:CI} shows that in settings 1 and (i) with $p/n=1.5$, the average width of the interval over repeated experiments is 1.97\% of $r_n$, with a standard deviation of 0.12\%. Under choices 1, 3, and 5 for $\Sigma_n$, the width is typically quite small as a percentage of $r_n$. By contrast, the percentage is larger in cases 2 and 4, which seems to occur because $r_n$ is smaller in these cases compared to 1, 3, and 5. Lastly, to consider coverage probability, the table shows good general agreement with the nominal level. Indeed, among the 60 distinct settings, there are only a few where the coverage probability differs from the nominal level by more than 2\%.

\begin{table}[H]
\small
\caption{Results for the stable rank confidence interval (95\% nominal coverage probability)}
\resizebox{\textwidth}{!}{
\setlength\tabcolsep{3pt}
\begin{tabular}{llllllllll}
\hline
                   &                          & \multicolumn{2}{c}{$p/n=0.5$} &  & \multicolumn{2}{c}{$p/n=1$} &  & \multicolumn{2}{c}{$p/n=1.5$} \\ \cline{3-4} \cline{6-7} \cline{9-10} 
$\xi_1$ & $\Sigma_n$ & width$/r_n$ $(\%)$   &  coverage $(\%)$ &&  width$/r_n$ $(\%)$   &  coverage $(\%)$  &&  width$/r_n$ $(\%)$   &  coverage $(\%)$   \\ \hline
 \specialrule{0em}{1pt}{1pt} 
\multirow{5}{*}{(i)} & 1      &            2.00(0.12) &95.20                     &  &                1.96(0.12) &95.40                &  &                     1.97(0.12) &93.80             \\
                   & 2       &            14.53(1.19) &95.40                     &  &                16.99(1.32) &91.60                &  &                     18.7(1.34) &95.40             \\
                  &  3       &            2.00(0.12) &95.00                     &  &                1.98(0.12) &97.40                &  &                     1.97(0.12) &95.40             \\
                  & 4       &            34.68(2.87) &94.00                     &  &                41.55(4.30) &94.00                &  &                     43.82(5.33) &93.80             \\
                  & 5       &            5.18(0.65) &94.80                     &  &                3.24(0.35) &95.20                &  &                     2.64(0.24) &94.60             \\
                   \specialrule{0em}{1pt}{1pt}\hline
                   \specialrule{0em}{1pt}{1pt}
\multirow{5}{*}{(ii)} & 1      &            2.07(0.12) &94.00                     &  &                2.01(0.13) &92.80                &  &                     1.98(0.12) &93.80             \\
                   & 2       &            14.68(1.17) &94.60                     &  &                17.10(1.34) &95.00                &  &                     18.75(1.42) &93.20             \\
                  &  3       &            2.08(0.13) &95.20                     &  &                2.02(0.12) &95.20                &  &                     2.00(0.13) &93.40             \\
                  & 4       &            34.86(2.76) &93.40                     &  &                41.47(4.09) &94.80                &  &                    44.99(5.34) &94.00             \\
                  & 5       &            5.12(0.63) &94.00                     &  &                3.22(0.33) &93.40                &  &                     2.63(0.24) &94.00             \\  
                   \specialrule{0em}{1pt}{1pt}
                   \hline
                    \specialrule{0em}{1pt}{1pt}
\multirow{5}{*}{(iii)} & 1      &          1.96(0.12) &94.40                                 &  &                1.96(0.11) &96.00                &  &                     1.96(0.11) &95.80             \\
                   & 2       &            14.45(1.11) &93.40                     &  &                17.01(1.27) &94.60                &  &                     18.57(1.42) &94.00             \\
                  &  3       &            1.97(0.12) &95.00                     &  &                1.96(0.12) &93.20                &  &                     1.96(0.12) &93.00             \\
                  & 4       &            34.71(2.67) &93.80                     &  &                41.49(4.38) &92.80                &  &                     44.37(5.21) &92.40             \\
                  & 5       &            5.17(0.64) &94.20                     &  &                3.22(0.33) &95.80                &  &                     2.61(0.24) &95.20             \\
                  \specialrule{0em}{1pt}{1pt}
                   \hline
                    \specialrule{0em}{1pt}{1pt}
\multirow{5}{*}{(iv)} & 1      &            2.38(0.21) &95.00                     &  &                2.32(0.19) &94.80                &  &                     2.31(0.25) &93.80             \\
                   & 2       &            15.31(1.19) &97.40                     &  &                17.86(1.42) &93.80                &  &                     19.53(1.49) &94.00             \\
                  &  3       &            2.4(0.20) &96.00                     &  &               2.37(0.24) &95.00                &  &                     2.33(0.27) &94.20             \\
                  & 4       &            35.43(2.90) &93.60                     &  &                42.01(4.33) &93.60                &  &                     44.59(5.06) &93.40             \\
                  & 5       &            5.20(0.64) &93.20                     &  &                3.26(0.33) &93.00               &  &                     2.73(0.29) &93.40             \\
                   \specialrule{0em}{1pt}{1pt}
                   \hline
\end{tabular}
}
\label{tab:CI}
\end{table}

\noindent\textbf{Stable rank test.} Here, we discuss numerical results for an instance of the testing problem~\eqref{eqn:H0} given by
\begin{equation}\label{eqn:numericalH}
\mathsf{H}_{0,n}: \ts\frac{r_n}{p} \leq 0.1  \ \ \ \text{\ \ \ \ \ \ \  vs. \ \ \ \ \ \ \   } \ \ \  \mathsf{H}_{1,n}: \ts\frac{r_n}{p} > 0.1 .
\end{equation}
As a way to unify the study of level and power, we modified the experiments from Section~\ref{sec:numlss} as follows. We rescaled the leading 15 eigenvalues in setting (1) to tune the ratio $r_n/p$ within the grid $\{0.0980, 0.0985,\dots, 0.1045,0.1050\}$. More precisely, the eigenvalues of $\Sigma_n$ were taken to be of the form $\lambda_j(\Sigma_n)=4s/3$ for $j\in\{1,\dots,5\}$, $\lambda_j(\Sigma_n)=s$ for $j\in\{6,\dots,15\}$ and $\lambda_j(\Sigma_n)=1$ for $j\in\{16,\dots,p\}$, with different values of $s$ being chosen to produce values of $r_n/p$ matching the stated gridpoints. Hence, gridpoints less than 0.1 correspond to $\mathsf{H}_{n,0}$, and gridpoints larger than $0.1$ correspond to $\mathsf{H}_{1,n}$. 

At each gridpoint, we performed experiments based on the design of those in  Section~\ref{sec:numlss}, allowing for $p/n$ to take the values $0.5, 1.0, 1.5$, and allowing the distribution of $\xi_1$ to be of the types (i), (ii), (iii), and (iv). For each such setting, we applied the relevant bootstrap test from Section~\ref{sec:testrank} at a 5\% nominal level to 500 datasets, and then we recorded the rejection rate over the 500 trials. Figures~\ref{fig:chisq},~\ref{fig:betaprime},~\ref{fig:beta}, and~\ref{fig:t} display the results by plotting the rejection rate as a function of the ratio $r_n/p$. The separate figures correspond to choices of the distribution of $\xi_1$, and within each figure, three colored curves correspond to choices of $p/n$, as indicated in the legend. In addition, all the plots include a dashed horizontal line to signify the 5\% nominal level.

An important feature that is shared by all the curves is that they stay below the nominal 5\% level for essentially every value of $r_n/p< 0.1$, which corroborates our theoretical bound~\eqref{eqn:Hlim2} in Theorem~\ref{thm:r}. Furthermore, when $r_n/p=0.1$, the curves are mostly quite close to 5\%, demonstrating that the testing procedure is well calibrated. For values of $r_n/p>0.1$, the procedure exhibits substantial power, with the curve corresponding to $p/n=0.5$ achieving approximately 100\% power at $r_n/p=0.105$ for every choice of $\xi_1$. In the cases of $p/n\in\{1.0, 1.5\}$, the procedure still retains power, but with some diminution, as might be anticipated in settings of higher dimension.

\begin{figure}[H]	
\centering
\vspace{0.1cm}
	\DeclareGraphicsExtensions{.eps}
	\begin{overpic}[width=0.45\textwidth]{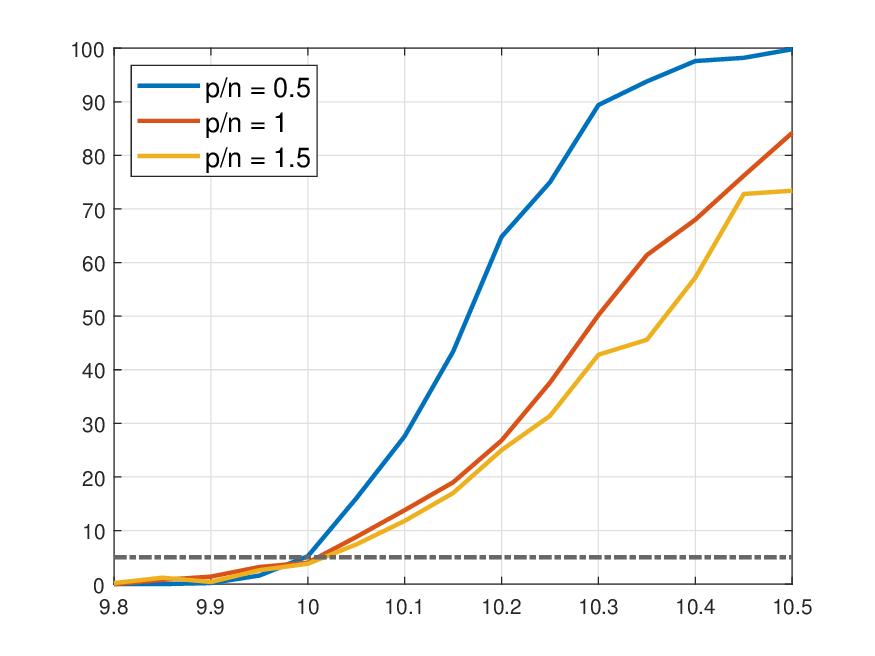} 
\put(22,75){ rejection rate vs. $r_n/p$  }
		\put(0,10){\rotatebox{90}{ {\footnotesize \ \ \ rejection rate (\%)  \ \ }}}
\put(40,-1){ $r_n/p \ (\%)$  }
\end{overpic}

\caption{ Results for the testing problem~\eqref{eqn:numericalH} when $\xi_1^2$ follows a Chi-Squared distribution with $p$ degrees of freedom. }
\label{fig:chisq}
\end{figure}

\begin{figure}[H]
\centering
	\DeclareGraphicsExtensions{.eps}
	\begin{overpic}[width=0.45\textwidth]{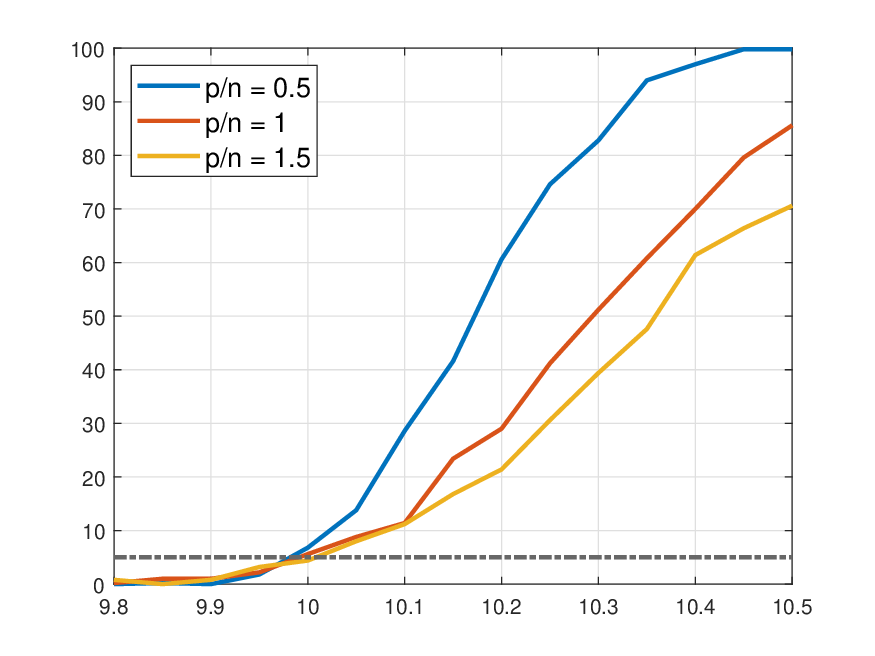} 
\put(22,75){ rejection rate vs. $r_n/p$  }
		\put(0,10){\rotatebox{90}{ {\footnotesize \ \ \ rejection rate (\%)  \ \ }}}
\put(40,-1){ $r_n/p \ (\%)$  }
	\end{overpic}
	\caption{Results for the testing problem~\eqref{eqn:numericalH} when $\xi_1^2$ follows a  Beta-Prime$\left(\frac{p(1+p+8)}{8},\frac{1+p+16}{8}\right)$ distribution.} 
	\label{fig:betaprime}
\end{figure}

\begin{figure}[H]
\centering
	\DeclareGraphicsExtensions{.eps}
	\begin{overpic}[width=0.45\textwidth]{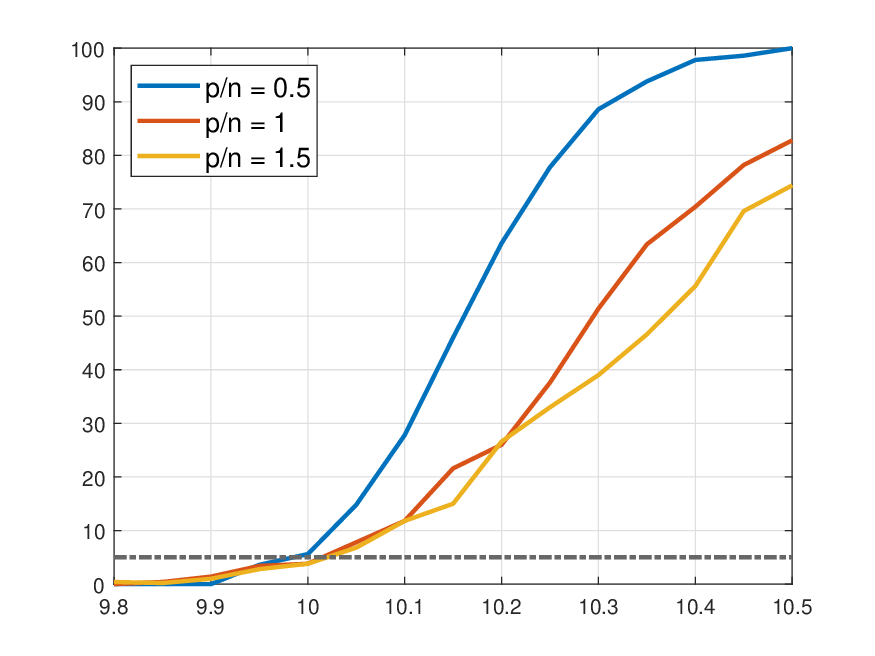} 
\put(22,75){ rejection rate vs. $r_n/p$  }
		\put(0,10){\rotatebox{90}{ {\footnotesize \ \ \ rejection rate (\%)  \ \ }}}
\put(40,-1){ $r_n/p \ (\%)$  }
	\end{overpic}	
\caption{Results for the testing problem~\eqref{eqn:numericalH} when $\xi_1^2$ follows a $(p+4)\textup{Beta}(p/2,2)$ distribution.}
\label{fig:beta}
\end{figure}

\begin{figure}[H]
\centering
	\DeclareGraphicsExtensions{.eps}
	\begin{overpic}[width=0.45\textwidth]{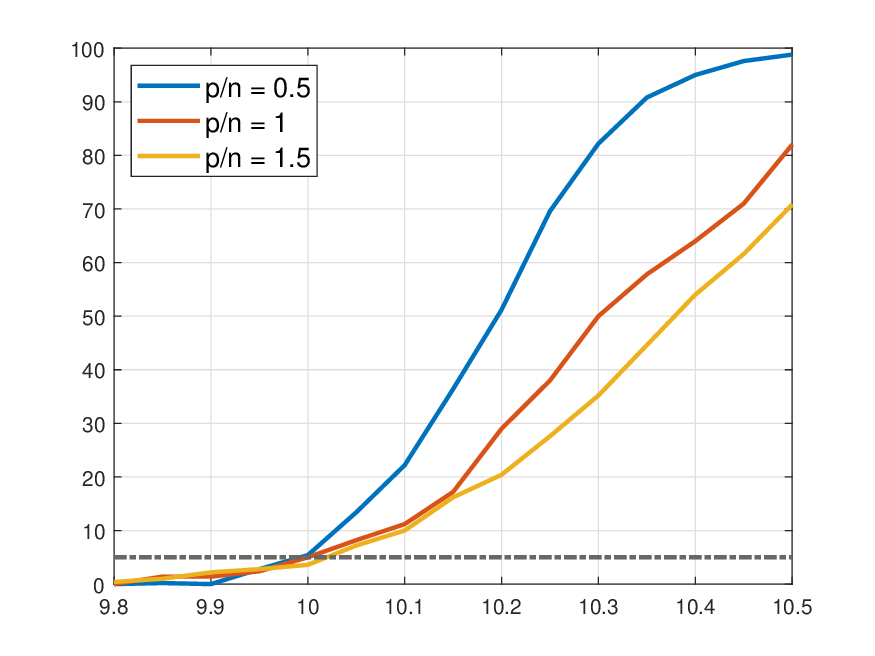} 
\put(22,75){ rejection rate vs. $r_n/p$  }
		\put(0,10){\rotatebox{90}{ {\footnotesize \ \ \ rejection rate (\%)  \ \ }}}
\put(40,-1){ $r_n/p \ (\%)$  }
	\end{overpic}	
\caption{Results for the testing problem~\eqref{eqn:numericalH} when $\xi_1^2$ follows a $\frac{9p}{10}\textup{F}(p,20)$ distribution.}
\label{fig:t}
\end{figure}

\subsection{Sphericity test}
Let $\mathsf{T}_{\text{sr}}=\hat r_n/p -1$ be a shorthand for the statistic that was introduced in Section~\ref{sec:testrank} for testing sphericity.
We now provide numerical comparisons with three other testing procedures based on linear spectral statistics. To define these other procedures, let
\begin{equation}
    \hat{\mathcal{S}}_n= \frac{p}{n}\sum_{i=1}^n \frac{\x_i\x_i\ttop}{\|\x_i\|_2^2}
\end{equation}
denote the sample covariance matrix of rescaled observations, and consider the following three statistics,
\begin{align}
    \mathsf{T}_1 &= \ts\frac{1}{p}\tr(\hat{\mathcal{S}}_n^2)-c_n-1\label{eqn:T1}\\[0.2cm]
    \mathsf{T}_2 &= \ts\frac{1}{p}\tr(\hat{\mathcal{S}}_n^4)-\ts\frac{4c_n}{p}\tr(\hat{\mathcal{S}}_n^3)-2c_n(\ts\frac{1}{p}\tr(\hat{\mathcal{S}}_n^2))^2+\ts\frac{10c_n^2}{p}\tr(\hat{\mathcal{S}}_n^2)-5c_n^3-1\label{eqn:T2}\\[0.2cm]
    \mathsf{T}_3 &= \max\Big\{\ts\frac{n\mathsf{T}_1+1}{2}\, , \, \frac{n\mathsf{T}_2+6-c_n}{\sqrt{8(18+12c_n+c_n^2)}}\Big\}.\label{eqn:Tm}
\end{align}
The testing procedures corresponding to these statistics reject the sphericity hypothesis when the statistics take large values.
The first two statistics can be attributed in part to the papers~\citep{Srivastava2005some} and~\citep{fisher2010}, and the proposal of taking the maximum was made in~\citep{tian2015robust}. However, in all of these works, an ordinary sample covariance matrix was used in place of $\hat{\mathcal{S}}_n$. Variants of the definition of $\mathsf{T}_1$ in~\eqref{eqn:T1} have been studied in~\citep{Paindeveine} and references in therein, while the definitions $\mathsf{T}_2$ and $\mathsf{T}_3$ in~\eqref{eqn:T2} and~\eqref{eqn:Tm} were proposed in~\citep{hu2019aos}. The latter paper also derived the limiting null distributions of all three statistics in high-dimensional elliptical models.

Since numerical comparisons of the statistics $\mathsf{T}_1$, $\mathsf{T}_2$, and $\mathsf{T}_3$ were given previously in~\citep{hu2019aos}, our experiments here are designed using similar settings. Under the null hypothesis, we generated data from a standard multivariate normal distribution in 15 cases, corresponding to 3 choices of $p/n\in\{0.5, 1, 2\}$ and 5 choices of $n\in\{100,200,300,400,500\}$. For the statistics $\mathsf{T}_1$, $\mathsf{T}_2$, and $\mathsf{T}_3$, we used the analytical critical values derived previously in~\cite{hu2019aos}, and for the statistic $\mathsf{T}_{\text{sr}}$, we determined its critical value using the proposed bootstrap with $B=500$.
For each setting under the null hypothesis, we generated 50000 datasets and calculated the empirical level of each test as the fraction of rejections among the 50000 trials.
The results corresponding to a nominal level of 5\% are displayed in Table~\ref{tab:null}, which shows that the empirical and nominal levels are in close agreement for all four statistics.

Regarding the alternative hypothesis, we retained all the settings described above, except that we replaced the null covariance matrix $\Sigma_n=I_p$ with a diagonal spiked covariance matrix such that $\lambda_j(\Sigma_n)=1.3$ for all $1\leq j\leq p/2$, and $\lambda_j(\Sigma_n)=1$ for all other $j$. This choice has the benefit that it creates variation in the numerical values of power, so that they are not too concentrated near 1. Similar alternatives were also used for the experiments in~\citep{hu2019aos}. The results are presented in Table~\ref{tab:power}, which is organized in the same format as Table~\ref{tab:null}. In each setting, the power of the statistic $\mathsf{T}_{\text{sr}}$ approximately matches the highest power achieved among $\mathsf{T}_1$,$\mathsf{T}_2$, and $\mathsf{T}_3$.

\begin{table}[H]
\setlength\tabcolsep{3pt}
\caption{Results on empirical level for sphericity tests (5\% nominal level)}
\begin{tabular}{llllllllllllllll}
\hline
      &  & \multicolumn{4}{c}{$p/n=0.5$} &  & \multicolumn{4}{c}{$p/n=1$} &  & \multicolumn{4}{c}{$p/n=2$} \\ \cline{3-6} \cline{8-11} \cline{13-16} 
      \specialrule{0em}{1pt}{1pt}
$n$ &  & $\mathsf{T}_1$   & $\mathsf{T}_2$   & $\mathsf{T}_3$  & $\mathsf{T}_{\text{sr}}$  &  & $\mathsf{T}_1$   & $\mathsf{T}_2$  & $\mathsf{T}_3$  & $\mathsf{T}_{\text{sr}}$  &  & $\mathsf{T}_1$   & $\mathsf{T}_2$  & $\mathsf{T}_3$  & $\mathsf{T}_{\text{sr}}$  \\
\hline
\specialrule{0em}{1pt}{1pt}
 100 &  &         0.049&0.048&0.051   & 0.050     &  &         0.049&0.049&0.051   &  0.054    &  &         0.048&0.051&0.053   &  0.051\\
 \specialrule{0em}{1pt}{1pt}
 \hline
\specialrule{0em}{1pt}{1pt}
 200 &  &         0.049&0.050&0.052   & 0.057     &  &         0.043&0.047&0.048   &  0.047    &  &         0.051&0.052&0.054   &  0.051\\
 \specialrule{0em}{1pt}{1pt}
 \hline
\specialrule{0em}{1pt}{1pt}
 300 &  &         0.046&0.048&0.051    & 0.046          &  &         0.046&0.047&0.048    &  0.048         &  &        0.051&0.054&0.059    &   0.052    \\
 \hline
\specialrule{0em}{1pt}{1pt}
 400 &  &       0.048&0.051&0.050  &  0.051          &  &         0.049&0.056&0.051    &  0.054        &  &        0.051&0.049&0.053    &  0.054\\
 \specialrule{0em}{1pt}{1pt}
 \hline
\specialrule{0em}{1pt}{1pt}
 500 &  &         0.049&0.045&0.048    & 0.052        &  &         0.050&0.046&0.048    &  0.053        &  &         0.044&0.045&0.044    &    0.046    \\
 \specialrule{0em}{1pt}{1pt}
 \hline
\end{tabular}
\label{tab:null}
\end{table}

\begin{table}[H]
\setlength\tabcolsep{3pt}
\caption{Results on power for sphericity tests}
\begin{tabular}{llllllllllllllll}
\hline
      &  & \multicolumn{4}{c}{$p/n=0.5$} &  & \multicolumn{4}{c}{$p/n=1$} &  & \multicolumn{4}{c}{$p/n=2$} \\ \cline{3-6} \cline{8-11} \cline{13-16} 
      \specialrule{0em}{1pt}{1pt}
$n$ &  & $\mathsf{T}_1$   & $\mathsf{T}_2$   & $\mathsf{T}_3$  & $\mathsf{T}_{\text{sr}}$  &  & $\mathsf{T}_1$   & $\mathsf{T}_2$  & $\mathsf{T}_3$  & $\mathsf{T}_{\text{sr}}$   &  & $\mathsf{T}_1$   & $\mathsf{T}_2$  & $\mathsf{T}_3$  & $\mathsf{T}_{\text{sr}}$  \\
\hline
\specialrule{0em}{1pt}{1pt}
 100 &  &         0.190&0.165&0.185    &   0.217          &  &         0.201&0.154&0.190     &  0.211          &  &         0.208&0.136&0.194    &   0.207       \\
 \specialrule{0em}{1pt}{1pt}
 \hline
\specialrule{0em}{1pt}{1pt}
 200 &  &         0.497&0.391&0.466    & 0.502           &  &         0.513&0.350&0.473    &   0.503       
        &  &         0.503&0.271&0.465    &     0.520       \\
 \specialrule{0em}{1pt}{1pt}
 \hline
\specialrule{0em}{1pt}{1pt}
 300 &  &         0.793&0.660&0.766    &  0.797         &  &         0.801&0.589&0.763    &  0.812         &  &        0.793&0.458&0.752    & 0.804         \\
 \hline
\specialrule{0em}{1pt}{1pt}
 400 &  &       0.952&0.863&0.941    &     0.952       &  &         0.951&0.790&0.933    &  0.954         &  &         0.958&0.670&0.941    &  0.949        \\
 \specialrule{0em}{1pt}{1pt}
 \hline
\specialrule{0em}{1pt}{1pt}
 500 &  &         0.993&0.958&0.990     &   0.994  &  &         0.993&0.922&0.989    &   0.993       &  &        0.993&0.816&0.987    &   0.995    \\ 
 \specialrule{0em}{1pt}{1pt}
 \hline
\end{tabular}
\label{tab:power}
\end{table}

\section{Conclusion}
Up to now, high-dimensional elliptical models  have generally fallen outside the scope of existing bootstrap methods for spectral statistics. In the current paper, we have addressed this problem by showing how a parametric bootstrap approach that is specialized to IC models~\citep{lopes2019bootstrapping} can be extended to elliptical models in high dimensions. In addition, we have shown that the new method is supported by two types of theoretical guarantees in the elliptical setting: First, the method consistently approximates the distributions of linear spectral statistics (Theorem~\ref{thm:main}). Second, the method can be applied to a nonlinear combination of these statistics to construct asymptotically valid confidence intervals and hypothesis tests (Theorem~\ref{thm:r}).
From a practical perspective, a valuable property of the method is its user-friendliness, since it can be applied to generic spectral statistics in an automatic way. In particular, this provides the user with the flexibility to easily explore spectral statistics whose asymptotic distributions are analytically inconvenient or unknown. With regard to empirical performance, we have presented extensive simulation results, showing that with few exceptions, the method accurately approximates the distributions of both linear and nonlinear spectral statistics across many settings and inference tasks. An interesting question for future work is to determine if a consistent parametric bootstrap method for spectral statistics can be developed within more general models that unify IC and elliptical models.

\section*{Appendices}

\appendix

In Appendices~\ref{app:estimators},~\ref{app:mainproof}, and~\ref{app:r}, we give the proofs of Theorems~\ref{thm:estimators},~\ref{thm:main}, and~\ref{thm:r} respectively. Background results are given in Appendix~\ref{app:background}, and additional details related to the examples in Section~\ref{sec:setup} are given in Appendix~\ref{app:examples}.

\setcounter{lemma}{0}
\renewcommand{\thelemma}{A.\arabic{lemma}}

\section{Proof of Theorem \ref{thm:estimators}}\label{app:estimators}
Our consistency guarantees for $\hat\varsigma_n^2$ and $\tilde H_n$ are proven separately in the next two subsections.

\subsection{Consistency of $\hat\varsigma_n^{\,2}$: Proof of~\eqref{eqn:tauconsistency} in Theorem~\ref{thm:estimators}}\label{sec:tauconsistency}

Define the parameter 
\begin{align}
    \tau_n = \frac{(p+2) (\beta_n-2\alpha_n)}{\gamma_n+2 \alpha_n} +2,
    \label{eqn:tau_tilde}
\end{align}
and the estimate
\begin{align}
    \hat{\tau}_n=\frac{(p+2)(\hat{\beta}_n - 2\hat{\alpha}_n)}{\hat{\gamma}_n + 2\hat{\alpha}_n} +2.
\end{align}

Based on the definitions of $\varsigma_n^2$ and $\hat\varsigma_n^2$, note that
$$\ts\frac{\varsigma_n^2}{p}=\tau_n \ \ \ \text{ and } \ \ \ \ \ts\frac{\hat\varsigma_n^2}{p} = \max\{\hat\tau_n,0\},$$
as well as the fact that the Assumption~\ref{Data generating model} implies $\tau_n\to\tau$ as $n\to\infty$. Since the function $x\mapsto \max\{x,0\}$ is 1-Lipschitz and $\tau_n\geq 0$, it suffices to show $\hat{\tau}_n - \tau_n \xrightarrow{\P} 0$.

In Lemmas \ref{lemma:gamma}, \ref{lemma:omega}, and \ref{lemma:nu} given later in this subsection, the following three limits  will be established,
\begin{align}
    \frac{\hat\alpha_n}{\alpha_n} \  \xrightarrow{\P} \ 1\\[0.2cm]
    \frac{\hat{\gamma}_n}{\gamma_n}\ \xrightarrow{\P} \ 1 \\[0.2cm]
    (p+2)\frac{\hat\beta_n-\beta_n}{\gamma_n + 2\alpha_n} \ \xrightarrow{\P}\  0.\label{eqn:nulim}
\end{align}
Due to the ratio-consistency of $\hat{\alpha}_n$ and $\hat{\gamma}_n$, as well as the fact that $\frac{(p+2) \alpha_n}{\gamma_n+ 2\alpha_n} \asymp 1$ holds under Assumption~\ref{Regularity of spectrum}, it is straightforward to check that 
\begin{equation}
       \frac{\gamma_n+ 2\alpha_n}{\hat{\gamma}_n + 2\hat{\alpha}_n} \xrightarrow{\P} 1  \ \ \  \ \text{ and } \ \ \ \  \frac{(p+2) \alpha_n}{\gamma_n+ 2\alpha_n} - \frac{(p+2)\hat{\alpha}_n}{\hat{\gamma}_n + 2\hat{\alpha}_n} \xrightarrow{\P} 0   .
\end{equation}
Therefore,
\footnotesize
\begin{equation}\label{eqn:taudiff}
\begin{split}
        \hat{\tau}_n - \tau_n &= \frac{(p+2)\hat\beta_n}{\gamma_n + 2\alpha_n}\left(\frac{\gamma_n+ 2\alpha_n}{\hat{\gamma}_n + 2\hat{\alpha}_n}- 1\right)+(p+2)\frac{\hat\beta_n-\beta_n}{\gamma_n + 2\alpha_n} + \bigg(2 \frac{(p+2) \alpha_n}{\gamma_n+ 2\alpha_n} -2\frac{(p+2)\hat\alpha_n}{\hat{\gamma}_n + 2\hat\alpha_n}\bigg) \\
    &= \frac{(p+2)\beta_n}{\gamma_n + 2\alpha_n}o_{\mathbb{P}}(1) + o_{\mathbb{P}}(1),
\end{split}
\end{equation}
\normalsize
where we have applied~\eqref{eqn:nulim} twice in the second step.
Under Assumption~\ref{Data generating model}, we have $\E(\xi_1^4)=p^2+\tau p +o(p)$ and so Lemma~\ref{lem:quadform} gives
\begin{align*}
    \beta_{n} & \ = \ \ts\frac{\E(\xi_1^4)}{p(p+2)}\Big(\tr(\Sigma_n)^2+2\tr(\Sigma_n^2)\Big) \ - \ \tr(\Sigma_n)^2\\[0.2cm]
    & \ = \ \ts\frac{p^2+\tau p +o(p)}{p(p+2)}\Big(\tr(\Sigma_n)^2+2\tr(\Sigma_n^2)\Big) \ - \ \tr(\Sigma_n)^2\\[0.2cm]
    & \ = \ \Big(\ts\frac{(\tau-2)p+o(p)}{p(p+2)}\Big)\tr(\Sigma_n)^2 + 2(1+o(1))\tr(\Sigma_n^2)\\[0.2cm]
    & \ \lesssim \ p.
\end{align*}
Consequently, we have $\frac{(p+2)\beta_n}{\gamma_n + 2\alpha_n} \lesssim 1$, and applying this to~\eqref{eqn:taudiff} completes the proof of~\eqref{eqn:tauconsistency} in Theorem~\ref{thm:estimators}.\qed

~\\

\begin{lemma}\label{lem:Frobenius}
If Assumptions \ref{Data generating model} and \ref{Regularity of spectrum} hold, then as $n\rightarrow \infty$
$$\frac{\hat\alpha_n}{\alpha_n} \  \xrightarrow{\P} \ 1.$$
\label{lemma:gamma} 
\end{lemma}

\proof
    Recall that $\alpha_n=\tr(\Sigma_n^2)$ and that $\hat\alpha_n=\tr(\hat\Sigma_n^2)-\frac{1}{n}\tr(\hat\Sigma_n)^2$. The two terms in the estimate can be expanded as
    \begin{equation}
        \begin{split}
            \tr(\hat\Sigma_n^2)  & \ = \ \frac{1}{n^2}\sum_{i=1}^n\sum_{j=1}^n (\mathbf{x}_i\ttop \mathbf{x}_j)^2\\[0.2cm]
            \frac{1}{n}\tr(\hat\Sigma_n)^2 & \ = \ \frac{1}{n^3}\sum_{i=1}^n\sum_{j=1}^n (\mathbf{x}_i\ttop \mathbf{x}_i)(\mathbf{x}_j\ttop \mathbf{x}_j),
        \end{split}
    \end{equation}
    which leads to the algebraic relation
    \begin{equation}
    \begin{split}
        \frac{\hat\alpha_n-\alpha_n}{\alpha_n} & \ = \ \bigg(\frac{1}{\alpha_n n^2}\sum_{i\neq j} (\mathbf{x}_i\ttop \mathbf{x}_j)^2 \ - \ 1 \bigg) \ + \ \frac{1}{\alpha_n n^3}\sum_{i\neq j} \Big((\mathbf{x}_i\ttop \mathbf{x}_i)^2 - (\mathbf{x}_i\ttop \mathbf{x}_i) (\mathbf{x}_j\ttop \mathbf{x}_j)\Big)\\[0.2cm]
        & \ =: \ A_n + B_n.
        \end{split}
    \end{equation}
    
    In the remainder of the proof, we will show that $A_n$ and $B_n$ are both $o_{\P}(1)$. We begin with the analysis of $B_n$, since it is simpler. Note that $B_n$ is always non-negative, since it can be rewritten as
    \begin{equation}
        B_n \ = \ \ts\frac{1}{\alpha_n n^3}\sum_{i>j}(\mathbf{x}_i\ttop \mathbf{x}_i - \mathbf{x}_j\ttop \mathbf{x}_j)^2,
    \end{equation}
    and so it suffices to show that $\E(B_n)=o(1)$. Furthermore, the expectation of $B_n$ can be computed directly as
    \begin{equation}\label{eqn:varnormsq}
        \begin{split}
            \E(B_n) 
            & \ = \ \ts\frac{1}{\alpha_n n^3} \sum_{i>j} 2\,\var(\mathbf{x}_i\ttop \mathbf{x}_i)\\[0.2cm]
            & \ = \ \ts\frac{n(n-1)}{\alpha_n n^3}\var(\mathbf{x}_1\ttop \mathbf{x}_1)\\[0.2cm]
            & \ = \ \frac{n(n-1)}{\alpha_n n^3}\bigg[\Big(\ts\frac{\E(\xi_1^4)}{p(p+2)}-1\Big)\tr(\Sigma_n)^2 \ + \ 2\frac{\E(\xi_1^4)}{p(p+2)}\tr(\Sigma_n^2)\bigg] \ \ \ \  \ \text{(Lemma \ref{lem:quadform})}\\[0.2cm]
            & \ = \ \frac{n(n-1)}{\alpha_n n^3}\bigg[\mathcal{O}\big(\ts\frac{1}{p}\big)\tr(\Sigma_n)^2 \ + \ 2\Big(1+\mathcal{O}\big(\ts\frac{1}{p}\big)\Big)\tr(\Sigma_n^2)\bigg] \ \ \ \ \ \ \ \text{(Assumption \ref{Data generating model})}\\[0.2cm]
            & \ \lesssim \ \ts\frac{1}{n},
        \end{split}
    \end{equation}
    where the last step uses $\frac{\tr(\Sigma_n)^2}{\alpha_n}\leq p$. Thus, $\E(B_n)=o(1)$.\\

     Now we handle the term $A_n$ by showing that $\E(A_n^2)=o(1)$. It is helpful to start by noting that if $i\neq j$, then
    \begin{equation}
        \E((\mathbf{x}_i\ttop \mathbf{x}_j)^2) = \tr(\Sigma_n^2),
    \end{equation}
    which can be checked by a direct calculation. (See the calculation in~\eqref{eqn:directcalc} for additional details.)  Consequently, if we expand out the square $A_n^2$ and then take the expectation, it follows that
    \begin{equation}\label{eqn:secondAn}
        \E(A_n^2) \ = \ \frac{1}{\alpha_n^2n^4}\E\bigg[\Big(\sum_{i\neq j}(\mathbf{x}_i\ttop \mathbf{x}_j)^2\Big)^2\bigg] \ - \  \frac{2n(n-1)}{n^2}  \ + \  1.
    \end{equation}
    Next, we compute the second moment on the right as
    \begin{equation}\label{eqn:quadsum}
        \E\bigg[\Big(\sum_{i\neq j}(\mathbf{x}_i\ttop \mathbf{x}_j)^2\Big)^2\bigg] \ = \ \displaystyle\sum_{i\neq j}\sum_{l\neq k} \E\Big[(\mathbf{x}_i\ttop \mathbf{x}_j)^2(\mathbf{x}_l\ttop \mathbf{x}_k)^2\Big].
    \end{equation}
    In Lemma~\ref{lem:mixedmoments}, it is shown that if $(i,j,k,l)$ are four distinct indices, then
     \begin{align}
    \E((\mathbf{x}_i\ttop \mathbf{x}_j)^4) \ &= \ 3\left(\frac{\E(\xi_1^4)}{p(p+2)}\right)^2 (\alpha_n^2 + 2 \tr(\Sigma_n^4))\label{eqn:quadmix}\\[0.2cm]
    \E((\mathbf{x}_i\ttop \mathbf{x}_j)^2(\mathbf{x}_i\ttop \mathbf{x}_k)^2) \ &= \ \frac{\E(\xi_1^4)}{p(p+2)} (\alpha_n^2 + 2 \tr(\Sigma_n^4))\label{eqn:triplemix}\\[0.2cm]
    \E((\mathbf{x}_i\ttop \mathbf{x}_j)^2(\mathbf{x}_l\ttop \mathbf{x}_k)^2) \ &= \ \alpha_n^2.\label{eqn:doublemix}
\end{align}
    Note that in the double sum~\eqref{eqn:quadsum}, the numbers of terms involving 2, 3, and 4 distinct indices are respectively $\mathcal{O}(n^2)$, $\mathcal{O}(n^3)$ and $n^4+\mathcal{O}(n^3)$.
    Applying these observations to~\eqref{eqn:quadsum}, we have
    \begin{equation*}
    \begin{split}
        \ts\frac{1}{\alpha_n^2n^4}\E\bigg[\Big(\sum_{i\neq j}(\mathbf{x}_i\ttop \mathbf{x}_j)^2\Big)^2\bigg]
    & \ = \ \ts\frac{\alpha_n^2}{\alpha_n^2 n^4}\Big(n^4+\mathcal{O}(n^3)\Big)\\
    & \ \ \ \ \ + \ \ts\frac{1}{\alpha_n^2 n^4}\Big(\ts\frac{\E(\xi_1^4)}{p(p+2)} (\alpha_n^2 + 2 \tr(\Sigma_n^4))\Big) \mathcal{O}(n^3)\\
    & \ \ \ \ \ + \ \ts\frac{1}{\alpha_n^2 n^4} \Big(3\left(\frac{\E(\xi_1^4)}{p(p+2)}\right)^2 (\alpha_n^2 + 2 \tr(\Sigma_n^4))\Big)\mathcal{O}(n^2)\\[0.2cm]
    & \ = \ 1 + \mathcal{O}(\ts\frac{1}{n}).
        \end{split}
    \end{equation*}
    Combining this with~\eqref{eqn:secondAn}, we reach the following bound on $\E(A_n^2)$,
    \begin{equation}
    \begin{split}
        \E(A_n^2) &\ = \ 1 +\mathcal{O}(\ts\frac{1}{n}) -  \frac{2n(n-1)}{n^2}  \ + \  1\\
        & \ \lesssim \ \ts\frac 1n,
        \end{split}
    \end{equation}
    which completes the proof.
\qed

\begin{lemma}\label{lemma:omega}
If Assumptions \ref{Data generating model} and \ref{Regularity of spectrum} hold, then as $n\to\infty$
$$\frac{\hat{\gamma}_n}{\gamma_n}\ \xrightarrow{\P} \ 1.$$
\end{lemma}

\proof
    Recall $\gamma_n = \tr(\Sigma_n)^2$ and $\hat{\gamma}_n =\tr(\hat{\Sigma}_n)^2$, and note that the algebraic identity $a^2-b^2 = 2b(a-b)+(a-b)^2$ gives
    \begin{align*}
        \E|\hat{\gamma}_n-\gamma_n|  & \ = \  \E\Big|2\tr(\Sigma_n)\big(\tr(\hat{\Sigma}_n) - \tr(\Sigma_n)\big) + \big(\tr(\hat{\Sigma}_n) - \tr(\Sigma_n)\big)^2\Big|\\[0.2cm]
        &  \ \leq \  2\tr(\Sigma_n)\big(\operatorname{var}(\tr(\hat{\Sigma}_n))\big)^{1/2} +\operatorname{var}(\tr(\hat{\Sigma}_n)).
    \end{align*}
   Next, observe that $\var(\tr(\hat\Sigma_n))=\frac 1n \var(\x_1\ttop \x_1)$, and that the calculation in~\eqref{eqn:varnormsq} shows $\var(\x_1\ttop \x_1)\lesssim p$. Combining this with the fact that $\gamma_n\asymp p^2$ under Assumption~\ref{Regularity of spectrum}, we have
   \begin{equation}
       \frac{ \E|\hat{\gamma}_n-\gamma_n|}{\gamma_n}  \ \lesssim \ \frac{\tr(\Sigma_n)+1}{p^2}
   \end{equation}
   which leads to the stated result.\qed

\begin{lemma}
    \label{lemma:nu}
If Assumptions \ref{Data generating model} and \ref{Regularity of spectrum} hold, then as $n\to\infty$,
 \begin{equation}
     (p+2)\frac{\hat\beta_n-\beta_n}{\gamma_n + 2\alpha_n} \ \xrightarrow{\P}\  0.
 \end{equation}
\end{lemma}

\proof
Recall that $\beta_{n}=\var(\|\mathbf{x}_1\|_2^2)$ and $\hat\beta_n=\frac{1}{n-1}\sum_{i=1}^n \Big(\|\mathbf{x}_{i}\|_2^2 - \ts\frac{1}{n}\sum_{i'=1}^n \|\mathbf{x}_{i'}\|_2^2\Big)^2$. It is clear that $\E(\hat{\beta}_{n}) = \beta_n$, and so it suffices to show  
$$\ts\frac{(p+2)^2}{(\gamma_n + 2\alpha_n)^2}\var(\hat\beta_n) = o(1).$$
Since Assumption~\ref{Regularity of spectrum} implies $\gamma_n + 2\alpha_n\asymp p^2$, it remains to show $\var(\hat{\beta}_{n}) = o(p^2)$. \\

Making use of a standard bound for the variance of a sample variance, we have
\begin{equation}\label{eqn:rosenthal}
\operatorname{var}(\hat{\beta}_{n})  \ \lesssim  \ \ts\frac{1}{n}\, \E\Big|\left\|\mathbf{x}_{1} \right\|_{2}^{2}-\tr(\Sigma_n)\Big|^4.
\end{equation}
To bound the right side of~\eqref{eqn:rosenthal}, observe that
\small
\begin{align*}
\E\Big|\!\left\|\mathbf{x}_{1} \right\|_{2}^{2}-\tr\left(\Sigma_n\right)\!\Big|^4
& \ = \ \E\left|\xi_1^{2}\left(\mathbf{u}_1\ttop \Sigma_n \mathbf{u}_1-\ts\frac{1}{p} \tr (\Sigma_n)\right)+\left(\ts\frac{1}{p} \xi_1^{2}-1\right) \tr(\Sigma_n)\right|^{4} \\[0.2cm]
& \ \lesssim \  \E(\xi_1^{8}) \cdot \E\left|\mathbf{u}_1\ttop \Sigma_n \mathbf{u}_1-\ts\frac{1}{p} \tr (\Sigma_n)\right|^{4}  + \tr(\Sigma_n)^4 \cdot \E\left|\ts\frac{1}{p} \xi_1^{2}-1\right|^{4}.\label{eqn:var4thmoment}
\end{align*}
\normalsize
Since $\lambda_1(\Sigma_n)\lesssim 1$ holds under Assumption~\ref{Regularity of spectrum}, it follows from Lemma~\ref{lem:4thmomentbound} that
$$
\E\left|\mathbf{u}_1\ttop \Sigma_n \mathbf{u}_1-\ts\frac{1}{p} \tr(\Sigma_n) \right|^{4} \ \lesssim \ \ts\frac{1}{ p^2}.
$$
Also, Assumption~\ref{Data generating model} implies
\begin{equation}
    \E\left|\ts\frac{1}{p} \xi_1^{2}-1\right|^{4} \ \lesssim \ \ts\frac{1}{p^2} \ \ \ \ \ \text{ and } \ \ \ \ \ \E(\xi_1^8) \, \lesssim \, p^4.
\end{equation}
Applying the last several observations to~\eqref{eqn:rosenthal} implies $\var(\hat\beta_n) \lesssim p$,
which yields the stated result.
\qed

\subsection{Consistency of $\tilde H_n$: Proof of~\eqref{eqn:Hconsistency} in Theorem~\ref{thm:estimators}}\label{sec:Hconsistency}

For each $t\in\R$, denote the empirical distribution function of the QuEST eigenvalue estimates as
$$\hat{H}_{\textup{Q},n}(t)=\frac 1p\sum_{j=1}^p 1\{\hat\lambda_{\textup{Q},j}\leq t\}.$$
It follows from Lemma~\ref{lem:LW} below that the limit
\begin{equation}\label{eqn:prelimLW}
    \hat H_{\textup{Q},n}\Rightarrow H
\end{equation}
holds almost surely as $n\to\infty$. So, to prove~\eqref{eqn:Hconsistency}, it is sufficient to show 
    \smash{$ \sup_{t\in\R}  | \hat{H}_{\textup{Q},n}(t)  -  \tilde{H}_{n}(t) | \xrightarrow{\P}0$.}

Since $\hat\lambda_{\textup{Q},j}$ and $\tilde\lambda_j$ can only disagree when $\hat\lambda_{\textup{Q},j}>\hat b_n$, we have
\begin{equation}\label{eqn:dKbound}
\begin{split}
 \sup_{t\in\R}  | \hat{H}_{\textup{Q},n}(t)  -  \tilde{H}_{n}(t) | & \ \leq  \ \frac{1}{p}\sum_{j=1}^p 1\{\hat\lambda_{\textup{Q},j}\neq \tilde\lambda_{j}\}\\[0.2cm]
    & \ = \ \frac{1}{p}\sum_{j=1}^p 1\{\hat\lambda_{\textup{Q},j}> \hat b_n\}\\[0.2cm]
    & \ = \ 1-\hat{H}_{\textup{Q},n}(\hat b_n).
    \end{split}
\end{equation}
Let $u$ denote the upper endpoint of the support of the distribution associated with $H$, and fix any $\e\in (0,\frac{1}{4})$. Since $\hat H_{\textup{Q},n}$ is a non-decreasing function, we have
\begin{equation}\label{eqn:LWalmost}
\begin{split}
     \sup_{t\in\R}  | \hat{H}_{\textup{Q},n}  -  \tilde{H}_{n}(t) | & \ \leq \
     1-\hat{H}_{\textup{Q},n}(u+\e)1\{\hat b_n\geq u+\e\}.
     \end{split}
\end{equation}
By the definition of $u$, the value $u+\e$ is a continuity point of $H$ with $H(u+\e)=1$, and so the limit~\eqref{eqn:prelimLW} implies $\hat H_{\textup{Q},n}(u+\e)=1+o_{\P}(1)$. Hence, in order to show that the right side of~\eqref{eqn:LWalmost} converges to 0 in probability, it is enough to show that $\P(\hat b_n\geq u+\e)\to 1$.

We will handle this remaining task by showing instead that $\P(\hat b_n<u+\e)\to 0$. Recall that $\hat b_n=\lambda_1(\hat\Sigma_n)+1$ and note that the limit $H_n\Rightarrow H$ implies that $\lambda_1(\Sigma_n)\geq u-\e$ must hold for all large $n$. Therefore,  we have
\begin{equation}\label{eqn:probrightside}
\begin{split}
    \P(\hat b_n<u+\e) & \ \leq  \  \P\Big(\lambda_1(\hat\Sigma_n)\leq \lambda_1(\Sigma_n)+2\e -1 \Big)
\end{split}
\end{equation}
for all large $n$.
To derive an upper bound on the last probability, we will replace $\lambda_1(\hat\Sigma_n)$ with a smaller random variable, and then rearrange the event. Let $\mathbf{v}_1$ denote an eigenvector of $\Sigma_n$ corresponding to $\lambda_1(\Sigma_n)$ with $\|\mathbf{v}_1\|_2=1$. Defining the random matrix $W_n=\frac{1}{n}\sum_{i=1}^n \xi_i^2 \u_i\u_i\ttop$, the variational representation of $\lambda_1(\hat\Sigma_n)$ gives
$\lambda_1(\hat\Sigma_n) \geq \lambda_1(\Sigma_n) \mathbf{v}_1' W_n \mathbf{v}_1$.
Also note that our choice of $\e$ ensures $2\e-1\leq -1/2$. This yields the following bounds for all large $n$,
\begin{equation}
\begin{split}
 \P(\hat b_n<u+\e) 
 & \ \leq  \ \P\Big( \lambda_1(\Sigma_n)|\mathbf{v}_1\ttop W_n\mathbf{v}_1-1|\geq 1/2\Big)\\[0.2cm]
 & \ \leq \ 4\lambda_1(\Sigma_n)^2\var\big(\mathbf{v}_1\ttop W_n\mathbf{v}_1\big)\\[0.2cm]
 & \ = \ \ts\frac{4\lambda_1(\Sigma_n)^2}{n}\,\var\big(\xi_1^2 \u_1\ttop (\mathbf{v}_1\mathbf{v}_1\ttop) \u_1\big)\\[0.2cm]
 & \ = \ \ts\frac{4\lambda_1(\Sigma_n)^2}{n}\,\left(3\frac{\E(\xi_1^4)}{p(p+2)}-1\right) \ \ \ \ \ \ \ \ \ \ \ \text{(Lemma~\ref{lem:quadform})}\\[0.0cm] 
 & \ \lesssim \ \ts\frac{1}{n},
 \end{split}
\end{equation}
where the last step uses Assumptions~\ref{Data generating model}~and~\ref{Regularity of spectrum}.\qed

\begin{lemma}\label{lem:LW}
    Suppose that Assumptions~\ref{Data generating model} and~\ref{Regularity of spectrum} hold. Then, the following limit holds almost surely as $n\to\infty$
    \begin{equation}\label{eqn:LWlimit}
        \hat H_{\textup{Q},n} \Rightarrow H.
    \end{equation}
\end{lemma}
\proof In the paper~\citep[][pp.381-382]{ledoit2015spectrum}, the limit~\eqref{eqn:LWlimit} is established in the context of an IC model where $p/n\to c\in (0,\infty)\setminus \{1\}$, and the eigenvalues of the population covariance matrix $\Sigma_n$ satisfy Assumption~\ref{Regularity of spectrum}.
To adapt the proof from~\citep{ledoit2015spectrum} to our current setting, it is sufficient to show that the following two facts hold under Assumptions~\ref{Data generating model} and~\ref{Regularity of spectrum}. First, there is a constant $C>0$ such that the bound $\limsup_{n\to\infty}\lambda_1(\hat\Sigma_n)\leq C$ holds almost surely, which we prove later using a truncation argument in Lemmas~\ref{lem:largest eigenvalues} and~\ref{lem:extra}. Second, the random distribution function $\hat H_n(t)=\frac{1}{p}\sum_{j=1}^p 1\{\lambda_j(\hat{\Sigma}_n)\leq t\}$ satisfies $\hat H_n\Rightarrow \Psi(H,c)$ almost surely, where we recall that the distribution $\Psi(H,c)$ is defined near equation~\eqref{eqn:MPdef}. The validity of this second fact under our current assumptions is a consequence of Theorem~1.1 in~\citep{Bai:Zhou:2008}.\qed

\subsection{Boundedness of sample eigenvalues}

\begin{lemma}
    \label{lem:largest eigenvalues}
   Under Assumptions \ref{Data generating model} and \ref{Regularity of spectrum}, there is a constant $C>0$ such that the bounds
    \begin{align}\label{eqn:2limsups}
        \limsup_{n\to\infty}\lambda_{1}({\hat{\Sigma}_n}) &\leq C \text{ \ \ \ \ \  \ and \ \  \ \ \ \ }     \limsup_{n\to\infty}\lambda_{1}({\tilde\Sigma_n}) \leq C
    \end{align}
    hold almost surely.
\end{lemma}
\proof Since the estimates $\tilde\lambda_j=\lambda_j(\tilde\Sigma_n)$ are bounded above by $\hat b_n=\lambda_1(\hat\Sigma_n)+1$ for all $j\in\{1,\dots,p\}$, it is enough to focus on the first inequality in~\eqref{eqn:2limsups}. Define a sequence of truncated random variables $\breve{\xi}_i=\xi_i\, 1\{|\xi_i^2-p|< \sqrt{pn}\}$ for $i=1,\dots,n$, as well as the following truncated version of $\hat\Sigma_n$, 
\begin{equation}\label{eqn:truncsigmadef}
    \breve \Sigma_n = \frac 1n \sum_{i=1}^n \breve \xi_i^2 \,\Sigma_n^{1/2}\u_i\u_i\ttop\Sigma_n^{1/2}.
\end{equation}
Lemma~\ref{lem:extra} below shows that $\P(\hat\Sigma_n\neq \breve \Sigma_n \  \text{i.o.})=0$.
Consequently, it suffices to show there is a constant $C>0$ such that $\limsup_{n\to\infty}\lambda_1(\breve\Sigma_n)\leq C$ holds almost surely.

Since the vectors $\u_1,\dots,\u_n$ are uniformly distributed on the unit sphere of $\R^p$, we may express them as $\u_i=\mathbf{z}_i/\|\mathbf{z}_i\|_2$ for a sequence of i.i.d.~standard Gaussian vectors $\mathbf{z}_1,\dots,\mathbf{z}_n$ in $\R^p$. This yields
\begin{equation}
    \breve \Sigma_n = \frac{1}{n}\sum_{i=1}^n \frac{\breve \xi_i^2/p}{\|\mathbf{z}_i\|_2^2/p} \Sigma_n^{1/2}\mathbf{z}_i\mathbf{z}_i\ttop\Sigma_n^{1/2}.
\end{equation}
By construction, we have $\max_{1\leq i\leq n} \breve\xi_i^2/p\leq 1+1/\sqrt{c_n}$ for all $n\geq 1$. Also, using standard tail bounds for the $\chi_p^2$ distribution and the Borel-Cantelli lemma, it is straightforward to show that $\liminf_{n\to\infty} \min_{1\leq i\leq n}\|\mathbf{z}_i\|_2^2/p$ is at least $1/2$ almost surely.  Taken together, these observations imply there is a constant $C>0$ such that the bound
\begin{equation}
    \lambda_1(\breve\Sigma_n) \ \leq \ C\,\lambda_1(\Sigma_n)\,\lambda_1\Big(\ts\frac 1n \sum_{i=1}^n \mathbf{z}_i\mathbf{z}_i\ttop\Big)
\end{equation}
holds almost surely for all large $n$. In addition, note that $\sup_{n\geq 1}\lambda_1(\Sigma_n)\lesssim 1$ holds by Assumption~\ref{Regularity of spectrum}. Lastly, it is known from~\citep[][Theorem 3.1]{Yin} that the limit
$$\lim_{n\to\infty}\lambda_1\Big(\ts\frac 1n \sum_{i=1}^n \mathbf{z}_i\mathbf{z}_i\ttop\Big) \ = \ (1+\sqrt c)^2$$
holds almost surely, which completes the proof.\qed

\begin{lemma}\label{lem:extra}
Suppose that Assumption~\ref{Data generating model} holds, and let $\breve\Sigma_n$ be as defined in~\eqref{eqn:truncsigmadef}. Then,
\begin{equation}
    \P(\hat\Sigma_n\neq \breve \Sigma_n \  \textup{i.o.})=0.
\end{equation}
\end{lemma}
\proof We adapt a classical argument from~\citep{Yin}. For a fixed number $n$, the matrices $\hat\Sigma_n$ and $\breve\Sigma_n$ can only disagree if $|\xi_i^2-p|\geq \sqrt{pn}$ for at least one $i=1,\dots,n$, and so
\begin{equation}
     \P(\hat\Sigma_n\neq \breve \Sigma_n \  \textup{i.o.}) \ \leq \ \lim_{j\to\infty}\P\bigg(\bigcup_{n=j}^{\infty}\bigcup_{i=1}^n \{|\xi_i^2-p|\geq \sqrt{pn}\}\bigg).
\end{equation}
Next, we partition the values of $n$ into the intervals $[2^{m-1},2^m),[2^{m},2^{m+1}),\dots$ and take a union bound across the intervals, yielding
\begin{equation}\label{eqn:bigmax}
    \begin{split}
             \P(\hat\Sigma_n\neq \breve \Sigma_n \  \textup{i.o.}) 
             &\ \leq \ \lim_{j\to\infty}\sum_{m=j}^{\infty}\P\Bigg(\max_{\underset{2^{m-1}\leq n< 2^m}{1\leq i\leq n}}|\ts\frac{\xi_i^2-p}{\sqrt p}| \, \geq \, \sqrt{2^{m-1}}\Bigg).
    \end{split}
\end{equation}
For a generic sequence of random variables $Y_1,\dots,Y_N$ and number $q\geq 1$, recall the standard maximal inequality
\begin{equation}\label{eqn:maximal}
\E\Big|\!\max_{1\leq i\leq N} Y_i\Big|^{q} \ \leq \ N\max_{1\leq i\leq N}\E|Y_i|^q.
\end{equation}
Since the number of pairs $(i,n)$ in the maximum in~\eqref{eqn:bigmax} is at most $2^{2m}$, if we apply Chebyshev's inequality to~\eqref{eqn:bigmax} and use the condition $ \E|\ts\frac{\xi_1^2-p}{\sqrt p}|^{4+\varepsilon}\lesssim 1$ from Assumption~\ref{Data generating model}, then for each $m\geq 1$ we have
\begin{equation}\label{eqn:firstYin}
\begin{split}
    \P\Bigg(\max_{\underset{2^{m-1}\leq n\leq 2^m}{1\leq i\leq n}}|\ts\frac{\xi_i^2-p}{\sqrt p}|\, \geq \, \sqrt{2^{m-1}}\Bigg)%
     & \ \lesssim \  \frac{2^{2m}}{\big(\sqrt{2^{m-1}}\big)^{4+\varepsilon}}\\[0.2cm]
     & \ \lesssim \ 2^{-m\varepsilon/2}.
     \end{split}
\end{equation}
 Hence, we may insert this bound into~\eqref{eqn:bigmax} to conclude that 
$$\P(\hat\Sigma_n\neq \breve \Sigma_n \  \textup{i.o.})  \ \lesssim \ \lim_{j\to\infty}\sum_{m=j}^{\infty} 2^{-m\varepsilon/2} \ = \ 0,$$
which completes the proof.\qed

\section{Proof of Theorem \ref{thm:main}}\label{app:mainproof}
    Let $\mathbf{z}\in\R^k$ denote a Gaussian random vector to be described in a moment, and consider the triangle inequality
    \begin{equation}
        d_{\mathrm{LP}}\Big(\mathcal{L}(p\big\{T_{n}(\mathbf{f})-\vartheta_{n}(\mathbf{f})\big\}\big) \, , \, \mathcal{L}\big(p\big\{T_{n, 1}^{*}(\mathbf{f})-\tilde{\vartheta}_{n}(\mathbf{f})\big\} \big| X\big)\Big) 
        \ \leq \ \textup{I}_n \ + \ \textup{II}_n,
    \end{equation} 
    where we define
  \begin{align}
              \textup{I}_n & \ = \  d_{\mathrm{LP}}\Big(\mathcal{L}\big(p\big\{T_{n}(\mathbf{f})-\vartheta_{n}(\mathbf{f})\big\}\big), \mathcal{L}(\mathbf{z})\Big) \\[0.2cm]
        \textup{II}_n & \ = \  d_{\mathrm{LP}}\Big(\mathcal{L}(\mathbf{z})\, ,\, \mathcal{L}\big(p\big\{T_{n, 1}^{*}(\mathbf{f})-\tilde{\vartheta}_{n}(\mathbf{f})\} \big| X\big)\Big).
  \end{align}
To handle the terms $\textup{I}_n$ and $\textup{II}_n$, we will apply a central limit theorem for linear spectral statistics established in \citep[][]{hu2019aos}, which relies on the following two conditions when $c_n\to c>0$ as $n\to\infty$:
\begin{itemize}
    \item[(a).] The elliptical model in Assumption~\ref{Data generating model} holds with the conditions on $\xi_1^2$ being replaced by
    \begin{equation}\label{eqn:xireqs}
      \E(\xi_1^2)=p,   \ \  \ \ \var\Big(\ts\frac{\xi_1^2-p}{\sqrt p}\Big)= \tau+o(1) \ \ \ \ \ \text{ and } \ \  \ \  \E\Big|\ts\frac{\xi_1^2-p}{\sqrt p}\Big|^{2+\varepsilon}\lesssim 1,
    \end{equation}
    for some constants constants $\tau\geq 0$ and $\varepsilon>0$ that do not depend on $n$.\\[-.2cm]
    \item[(b).] There is a distribution $H$ such that $H_n\Rightarrow H$ as $n\to \infty$, and $\lambda_1(\Sigma_n)\lesssim 1$.
\end{itemize}
Under these conditions, Theorem 2.2 in~\citep{hu2019aos} ensures there exists a Gaussian distribution $\mathcal{L}(\mathbf{z})$ depending only on $(\mathbf{f},H,c,\tau)$ such that $\textup{I}_n\to  0$ as $n\to\infty$.

To finish the proof, we must show $\textup{II}_n\xrightarrow{\P}0$. This can be done using the mentioned central limit theorem, but some extra considerations are involved, due to the fact that $\textup{II}_n$ is random. It is sufficient to show that for any subsequence $J\subset\{1,2,\dots\}$, the limit $\textup{II}_n\to 0$ holds almost surely along a further subsequence of $J$. Since the bootstrap data are generated from an elliptical model that is parameterized in terms of $\mathcal{L}((\xi_1^*)^2|X)$ and $\tilde \Sigma_n$, this amounts to verifying that $\mathcal{L}((\xi_1^*)^2|X)$ and $\tilde H_n$ satisfy analogues of (a) and (b) almost surely along a subsequence of $J$.

To proceed, recall that Algorithm~\ref{alg:boot} is designed so that $\E((\xi_1^*)^2|X)=p$ and $\var((\xi_1^*)^2|X)=\hat\varsigma_n^2$. Also, note that under Assumption~\ref{Data generating model}, the parameter $\varsigma_n^2=\var(\xi_1^2)$ satisfies $\varsigma_n^2/p\to \tau$. Consequently, Theorem~\ref{thm:estimators} implies $\hat\varsigma_n^2/p=\tau+o_{\P}(1)$, and so there is a subsequence $J'\subset J$ along which the limit $\hat\varsigma_n^2/p\to \tau$ holds almost surely. In other words, the limit
\begin{equation}
    \var\!\Big(\ts\frac{(\xi_1^*)^2-p}{\sqrt p}\Big| X\Big) \ \to \ \tau
\end{equation}
holds almost surely along $J'$.
Moreover, since $\mathcal{L}((\xi_1^*)^2|X)$ is a Gamma distribution with mean $p$ and variance $\hat\varsigma_n^2$, Lemma~\ref{lem:gamma} implies there is a constant $C>0$ not depending on $n$ such that the bound
\begin{equation}\label{eqn:conditionalmomentbound}
\Big(\E\Big(\big|\ts\frac{(\xi_1^*)^2-p}{\sqrt p}\big|^6\Big|X\Big)\Big)^{1/6}\leq
C\big(\ts\frac{\hat\varsigma_n}{\sqrt p}+\ts\frac{\hat\varsigma_n^2}{p^{3/2}}\big)
\end{equation}
holds almost surely. Therefore, the left side of~\eqref{eqn:conditionalmomentbound} is bounded almost surely along $J'$.

With regard to the empirical spectral distribution $\tilde H_n$ associated with the matrix $\tilde\Sigma_n$, it satisfies the limit $\tilde H_n\Rightarrow H$ almost surely along a further subsequence $J''\subset J'$, due to Theorem~\ref{thm:estimators}. In addition, Lemma~\ref{lem:largest eigenvalues} ensures there is a constant $C>0$ such that $\limsup_{n\to\infty}\lambda_1(\tilde\Sigma_n)\leq C$ almost surely. Altogether, it follows that $\mathcal{L}((\xi_1^*)^2|X)$ and $\tilde H_n$ simultaneously satisfy analogues of (a) and (b) almost surely along $J''$, which implies $\textup{II}_n\xrightarrow{\P}0$.\qed

\section{Proof of Theorem \ref{thm:r}}\label{app:r}
To begin the proof, we need to introduce three auxiliary statistics defined by
\begin{align}\label{eqn:rbrevedef}
    \breve{r}_n &= \frac{\tr(\hat\Sigma_n)^2}{\tr(\hat\Sigma_n^2) - \frac{1}{n}\tr(\hat\Sigma_n)^2}  \\[0.2cm]
    \breve{r}_n^* & = \frac{\tr(\hat{\Sigma}^*_n)^2}{\tr((\hat{\Sigma}^*_n)^2) - \frac{1}{n}\tr(\hat{\Sigma}_n^*)^2}\\[0.2cm]
\tilde r_n &=\frac{\tr(\tilde\Sigma_n)^2}{\tr(\tilde\Sigma_n^2)}.
\end{align}
Also, we define $\hat r_n^*$ as the statistic obtained by applying the formula~\eqref{eqn:rhatdef} for $\hat r_n$ to the bootstrap data $\x_1^*,\dots,\x_n^*$. The primary task is to establish the following four limits, which are established later in this appendix in Propositions~\ref{prop:delta} and~\ref{prop:correction}. Specifically, these results show that there exists a non-degenerate Gaussian random variable $\zeta$ and a constant $a$, such that as $n\to\infty$,
\begin{align}
    \mathcal{L}(\breve r_n - r_n) & \ \Rightarrow \ \mathcal{L}(\zeta),\label{eqn:zetalim1}\\[0.2cm]
\mathcal{L}\big( \breve r_n^* - \tilde r_n\big|X\big) & \ \xRightarrow{\P} \ \mathcal{L}(\zeta),\label{eqn:zetalim2}\\[0.2cm]
    \mathcal{L}(\hat r_n -\breve r_n) &  \ \Rightarrow  \  \mathcal{L}(a),\label{eqn:alim1}\\[0.2cm]
     \mathcal{L}(\hat r_n^*-\breve r_n^*\big| X\big) & \ \xRightarrow{\P} \  \mathcal{L}(a),\label{eqn:alim2}
\end{align}
where $\mathcal{L}(a)$ denotes the point mass distribution at $a$. (In the current section, we sometimes use the notation for weak convergence in limits where convergence in probability holds, because it will help to clarify how $\hat r_n^*-\breve r_n^*$ can be analyzed in an analogous manner to $\hat r_n-\breve r_n$.) Using Slutsky's lemma, it follows from the limits~\eqref{eqn:zetalim1} and~\eqref{eqn:alim1} that
\begin{align}
    \mathcal{L}(\hat r_n - r_n) &  \ \Rightarrow  \  \mathcal{L}(\zeta+a).
    \end{align}
Analogously, Slutsky's lemma can be applied in a conditional manner to the limits~\eqref{eqn:zetalim2} and~\eqref{eqn:alim2}, yielding
    \begin{align}
    \mathcal{L}\big( \hat r_n^* - \tilde r_n\big|X\big) & \ \xRightarrow{\P} \ \mathcal{L}(\zeta+a).
\end{align} 
Since the limiting distribution  $\mathcal{L}(\zeta+a)$ is continuous, P\'olya's theorem implies 
\begin{align}
    \sup_{t\in\R}\Big|\P(\hat r_n-r_n\leq t) - \P(\hat r_n^*-\tilde r_n \leq t|X)\Big| \ \xrightarrow{\P} \  0.
\end{align}
Due to this uniform limit and the continuity of the distribution $\mathcal{L}(\zeta+a)$, standard arguments can be used to show that the quantiles of  $\mathcal{L}(\hat r_n^*-\tilde r_n|X)$  are asymptotically equivalent to those of $\mathcal{L}(\hat r_n-r_n)$. (For example, see the proof of Lemma 10.4 in~\citep{lopes2022}.) More precisely, if we note that the quantile estimate $\hat q_{1-\alpha}$ defined in Section~\ref{sec:rank} is the same as the $(1-\alpha)$-quantile of  $ \mathcal{L}\big(\frac 1p (\hat r_n^*-\tilde r_n)|X\big) $, then as $n\to\infty$ we have the limit
\begin{equation}\label{eqn:lasthyplimit}
    \P\big( \ts\frac 1p(\hat r_n-r_n) > \hat q_{1-\alpha}\big) \to \alpha.
\end{equation}
This limit directly implies the first two statements~\eqref{eqn:Hlim1} and~\eqref{eqn:Hlim2} in Theorem~\ref{thm:r}. Regarding the third statement~\eqref{eqn:Hlim3}, if $\mathsf{H}_{0,n}'$ holds for all large $n$, then replacing $\tilde\Sigma_n=\textup{diag}(\tilde\lambda_1,\dots,\tilde\lambda_p)$ with $\tilde \Sigma_n=I_p$ does not affect the reasoning leading up to~\eqref{eqn:lasthyplimit}, because this replacement does not affect the proofs of Propositions~\ref{prop:delta} and~\ref{prop:correction} given later.
Consequently, if $\mathsf{H}_{0,n}'$ holds for all large $n$, then we have $\P\big(\frac 1p(\hat r_n-1)\leq \hat q_{\alpha}'\big)\to \alpha$ as $n\to\infty$, which completes the proof.\qed

~\\

\begin{proposition}\label{prop:delta}
Under Assumptions~\ref{Data generating model} and~\ref{Regularity of spectrum}, the limits~\eqref{eqn:zetalim1} and~\eqref{eqn:zetalim2} hold as $n\to\infty$.
\end{proposition}

\proof The proof is based on viewing $\breve r_n$ as a nonlinear function of $T_n(\mathbf{f})$, where the components of $\mathbf f =(f_1,f_2)$ are taken to be $f_1(x)=x$ and $f_2(x)=x^2$. In this case, the centering parameter $\vartheta_n(\mathbf{f})$ defined by~\eqref{eqn:varthetadef}  reduces to
\begin{equation}\label{eqn:vartheta2}
    \vartheta_n(\mathbf f)  
    \ = \  \left(\ts\frac{1}{p}\tr(\Sigma_n)\,,\, \ts\frac{1}{p}\tr(\Sigma_n^2) + \frac{1}{np}\tr(\Sigma_n)^2 \right),
\end{equation}
as recorded in Lemma 2.16 of~\citep{yao2015sample}. Consequently, by considering the function 
\begin{equation}\label{eqn:gndef}
g_n(x_1,x_2) = \ts\frac{x_1^2}{x_2 - c_n x_1^2},
\end{equation}
we have the key relation
\begin{equation}\label{eqn:firstdiff}
    \breve{r}_n - r_n \ = \  p\big\{g_n(T_n(\mathbf{f})) - g_n(\vartheta_{n}(\mathbf{f})) \big\}.
\end{equation}
We can also develop a corresponding relation for the bootstrap statistic $\breve r_n^*-\tilde r$. 
Due to Lemma 2.16 in~\citep{yao2015sample} and our choice of $\mathbf f$, the definition of $\tilde\vartheta_n(\mathbf f)$ below~\eqref{eqn:varthetadef} implies 
\begin{equation}\label{eqn:tildevartheta2}
        \tilde \vartheta_n(\mathbf f)  
        \ = \  \left(\ts\frac{1}{p}\tr(\tilde \Sigma_n) \, , \, \ts\frac{1}{p}\tr(\tilde\Sigma_n^2) + \frac{1}{np}\tr(\tilde \Sigma_n)^2 \right).
\end{equation}
Likewise, the bootstrap version of~\eqref{eqn:firstdiff} is
\begin{equation}\label{eqn:seconddiff}
    \breve{r}_n^* - \tilde r_n \ = \  p\big\{g_n(T_{n}^{*}(\mathbf{f})) - g_n(\tilde{\vartheta}_{n}(\mathbf{f}))\big \}.
\end{equation}

We will proceed by applying the delta method to the relations~\eqref{eqn:firstdiff} and~\eqref{eqn:seconddiff}. For this purpose, note that the proof of Theorem~\ref{thm:main} shows there exists a Gaussian random vector $\mathbf z\in\R^2$ such that as $n\to\infty$,
\begin{align}
    \mathcal{L}\big(p\big\{T_{n}(\mathbf{f})-\vartheta_{n}(\mathbf{f})\big\}\big) &  \ \Rightarrow  \  \mathcal{L}(\mathbf{z})\label{eqn:CLT2first}\\[0.2cm]
       \mathcal{L}\big(p\big\{T_{n, 1}^{*}(\mathbf{f})-\tilde{\vartheta}_{n}(\mathbf{f})\} \big| X\big) & \ \xRightarrow{\P} \ \mathcal{L}(\mathbf z).\label{eqn:CLT2second}
\end{align}
Also, to introduce some further notation, we refer to the $j$th moment of the distribution $H$ as
\begin{equation}
    \phi_j = \ts\int t^j dH(t),
\end{equation}
and we use these moments to define the parameter
\begin{equation}\label{eqn:thetaformula}
    \vartheta(\mathbf f) = (\phi_1,\phi_2+c\phi_1^2).
\end{equation}
This parameter arises in the following limits as $n\to\infty$,
\begin{equation}\label{eqn:twolimits}
    \vartheta_n(\mathbf f) \to \vartheta(\mathbf f) \ \ \ \ \  \ \text{ and } \ \ \ \  \ \ \tilde\vartheta_n(\mathbf f) \xrightarrow{\P} \vartheta(\mathbf f).
\end{equation}
To see why these limits hold, first note that by Assumption~\ref{Regularity of spectrum}, there is a compact interval containing support of $H_n$ for all large $n$, and by Lemma~\ref{lem:largest eigenvalues}, the same statement holds almost surely for $\tilde H_n$. Moreover, Assumption~\ref{Regularity of spectrum} and Theorem~\ref{thm:estimators} ensure the limits $H_n\Rightarrow H$ and $\tilde H_n\xRightarrow{\P} H$, and so it follows that the moments of $H_n$ converge to those of $H$, and the moments of $\tilde H_n$ converge in probability to those of $H$. Combining these facts with the formulas~\eqref{eqn:vartheta2} and~\eqref{eqn:tildevartheta2} yields the limits~\eqref{eqn:twolimits}.

In light of the relations~\eqref{eqn:firstdiff} and~\eqref{eqn:seconddiff}, and the limits~\eqref{eqn:twolimits}, we will expand the function $g_n$ around $\vartheta(\mathbf f)$ and apply the delta method. This is justified because the gradient $\nabla g_n$ has the following continuity property. Namely, if we let $g(x_1,x_2)=\frac{x_1^2}{x_2-cx_1^2}$, let $\mathcal{U}$ be a sufficiently small open neighborhood of $\vartheta(\mathbf f)$ in $\R^2$, and let $\{(x_{1,n},x_{2,n})\}$ be any sequence of points within $\mathcal{U}$ that converges to $\vartheta(\mathbf f)$, then we have the limit
\begin{equation}\label{eqn:gradlimit}
    \nabla g_n(x_{1,n},x_{2,n}) \ \to \ \nabla g(\vartheta(\mathbf f)) \ = \  \left(\ts\frac{2\phi_1}{\phi_2}+\frac{2c\phi_1^3}{\phi_2^2},- \frac{\phi_1^2}{\phi_2^2}\right).
\end{equation}
So, applying the delta method to the relations~\eqref{eqn:firstdiff} and~\eqref{eqn:seconddiff} and the weak limits~\eqref{eqn:CLT2first} and~\eqref{eqn:CLT2second} gives
\begin{align}
    \mathcal{L}(\breve r_n -r_n) &  \ \Rightarrow  \  \mathcal{L}(\nabla g(\vartheta(\mathbf f))\ttop \mathbf{z})\\[0.2cm]
       \mathcal{L}(\breve r_n^*-\tilde r_n\big| X\big) & \ \xRightarrow{\P} \  \mathcal{L}(\nabla g(\vartheta(\mathbf f))\ttop \mathbf{z}).
\end{align}

Now, it remains to show that the variance of the Gaussian random variable $\nabla g(\vartheta(\mathbf f))\ttop \mathbf{z}$ is positive. Letting $K$ denote the $2\times 2$ covariance matrix of $\mathbf z$, it is shown below equation 2.10 in~\citep{hu2019aos} that the entries of $K$ are given by
\begin{align*}
    K_{11} &= 2c\phi_2+c(\tau-2)\phi_1^2,\\
   K_{12} & = 4c\phi_3 + 4 c^2\phi_1\phi_2 +2c(\tau-2)\phi_1(c\phi_1^2+\phi_2)\\
    K_{22} &= 8c\phi_4 +4c^2\phi_2^2 + 16 c^2\phi_1\phi_3 +8 c^3\phi_1^2\phi_2+4c(\tau-2)(c\phi_1^2+\phi_2)^2.
\end{align*}
Combining this with the formula for $\nabla g(\vartheta(\mathbf f))$ in~\eqref{eqn:gradlimit},
a direct but lengthy computation of the quadratic form $\nabla g(\vartheta(\mathbf f))\ttop K \,\nabla g(\vartheta(\mathbf f))$ yields
\begin{equation}\label{eqn:varformula}
\begin{split}
    \var\!\Big(\nabla g(\vartheta(\mathbf f))\ttop \mathbf{z}\Big)
    & \ = \ \ts\frac{4c\phi_1^2}{\phi_2^4}\Big(c\phi_1^2\phi_2^2 +2\phi_1^2\phi_4 +2\phi_2^3-4\phi_1\phi_2\phi_3\Big).
    \end{split}
\end{equation}
To see that the variance is positive, it  suffices to check that $2\phi_1^2\phi_4 +2\phi_2^3-4\phi_1\phi_2\phi_3$ is non-negative. Rewriting this as $2(\phi_1\phi_4^{1/2}-\phi_2^{3/2})^2+ 4\phi_1\phi_2(\phi_4^{1/2}\phi_2^{1/2}-  \phi_3)$, its non-negativity follows from the observation that $\phi_3\leq \phi_4^{1/2}\phi_2^{1/2}$ is an instance of the Cauchy-Schwarz inequality. \qed

~\\

\begin{proposition}\label{prop:correction}
    Under Assumptions~\ref{Data generating model} and~\ref{Regularity of spectrum}, the limits~\eqref{eqn:alim1} and~\eqref{eqn:alim2} hold as $n\to\infty$.
\end{proposition}
\proof Here, we retain the definitions of $f_1$ and $f_2$ used in the proof of the previous proposition. The difference $\hat r_n-\breve r_n$ can be written explicitly in terms of $T_n(f_1)$, $T_n(f_2)$, and $\hat\Delta_n$ (defined below~\eqref{eqn:rhatdef}) as
\begin{align}\label{eqn:rhatdiff}
    \hat r_n - \breve r_n = \frac{T_{n}(f_1)^2\Big(\hat \Delta_n - pc_nT_{n}(f_1)^2\Big)}{(T_{n}(f_2)-c_nT_{n}(f_1)^2)(T_{n}(f_2)-\hat \Delta_n /p)}.
\end{align}
Furthermore, $\hat\Delta_n$ can be decomposed as
\begin{equation}
    \hat\Delta_n =  \hat\kappa_{1} T_n(f_1)^2 \ + \ \hat\kappa_{2} T_n(f_2),
\end{equation}
where the random coefficients $\hat\kappa_{1}$ and $\hat\kappa_{2}$ are defined by
\begin{align}
    \hat\kappa_{1} \ &= \ c_n p \bigg[ \frac{n+1}{n} +\frac{n-1}{n} \frac{{\hat\varsigma}_n^2 - 2p}{p(p+2)} - \frac{2(n-1)}{n^2} \frac{p^2+\hat\varsigma_n^2}{p(p+2)} \bigg]\\[0.2cm]
\hat\kappa_{2} \ &= \ c_n\bigg[\frac{2(n-1)}{n}\frac{p^2+\hat\varsigma_n^2}{p(p+2)} - 1\bigg].\label{eqn:kappa2}
\end{align}
The last few displays show that in order to determine the limit of $\hat r_n-\breve r_n$, it suffices to determine the limit of the triple $(T_{n}(f_1),T_{n}(f_2), \hat\varsigma_n^2/p)$. For the random variables $T_n(f_1)$ and $T_n(f_2)$, we can apply the limits~\eqref{eqn:CLT2first} and~\eqref{eqn:twolimits} as well the formula~\eqref{eqn:thetaformula} to obtain
\begin{align}\label{eqn:T12const}
    \mathcal{L}(T_{n}(f_1),T_{n}(f_2)) \Rightarrow \mathcal{L}(\phi_1,\phi_2+c\phi_1^2).
\end{align}
In addition, the proof of the limit~\eqref{eqn:tauconsistency} in Theorem~\ref{thm:estimators} shows that 
\begin{equation}\label{eqn:sig2const}
    \mathcal{L}(\ts\frac{1}{p}\hat\varsigma_n^2) \ \Rightarrow \ \mathcal{L}(\tau).
\end{equation}
Combining the previous two displays with the formulas~\eqref{eqn:rhatdiff}-\eqref{eqn:kappa2}, a direct calculation leads to
\begin{align}\label{eqn:hatrlimtemp}
    \mathcal{L}(\hat r_n - \breve r_n) \ \Rightarrow \ \mathcal{L}(a),
\end{align}
where
\begin{equation*}
    a=\ts\frac{c\phi_1^2(\phi_2 + (\tau-2)\phi_1^2)}{\phi_2^2},
\end{equation*}
which proves~\eqref{eqn:alim1}. 

Now we turn to the proof of~\eqref{eqn:alim2}.
By analogy with the previous argument that led to~\eqref{eqn:hatrlimtemp}, it is enough to show the following two limits hold as $n\to\infty$,
\begin{align}
        \mathcal{L}(T_{n}^*(f_1),T_{n}^*(f_2)|X) &\xRightarrow{\P} \mathcal{L}(\phi_1,\phi_2+c\phi_1^2)\label{eqn:firstrstarlim}\\[0.2cm]
        \mathcal{L}(\ts\frac{1}{p}(\hat\varsigma_n^2)^*|X) &\xRightarrow{\P} \mathcal{L}(\tau),\label{eqn:secondrstarlim}
\end{align}
where $(\hat\varsigma_n^2)^*$ is obtained by applying the formula~\eqref{eqn:varest} to the bootstrap data $\x_1^*,\dots,\x_n^*$. The first limit~\eqref{eqn:firstrstarlim} is a consequence of two limits that were established in the proof of Theorem~\ref{thm:main}, which are that $\textup{II}_n\xrightarrow{\P}0$ and $\tilde\vartheta_n(\mathbf{f})\xrightarrow{\P}(\phi_1,\phi_2+c\phi_1^2)$.

Regarding the second limit~\eqref{eqn:secondrstarlim}, note that it is equivalent to showing that for any subsequence $J\subset\{1,2,\dots\}$, there is a further subsequence $J'\subset J$ such that $\mathcal{L}(\ts\frac{1}{p}(\hat\varsigma_n^2)^*|X)\Rightarrow \mathcal{L}(\tau)$ holds almost surely along $J'$. The latter statement can be proven by analogy with the limit $\mathcal{L}(\frac{1}{p}\hat{\varsigma}_n^2)\Rightarrow \mathcal{L}(\tau)$, which follows from the proof of~\eqref{eqn:tauconsistency} in Theorem~\ref{thm:estimators}. To be more precise, this analogy can be justified as follows: The proof of~\eqref{eqn:tauconsistency} only relies on Assumption~\ref{Data generating model} and two other conditions, which are
\begin{equation}\label{eqn:pairlim}
\lambda_1(\Sigma_n)\lesssim 1, \ \ \  \ \text{ and } \ \ \  \ (\ts\frac{1}{p}\tr(\Sigma_n),\ts\frac 1p \tr(\Sigma_n^2)) \to (\phi_1,\phi_2).
\end{equation}
Consequently, it is enough to check that bootstrap counterparts of these conditions hold almost surely along $J'$. First, the bootstrap counterpart of Assumption~\ref{Data generating model} was shown to hold almost surely along subsequences in the proof of Theorem~\ref{thm:main}. Second, the bootstrap counterpart of~$\lambda_1(\Sigma_n)\lesssim 1$ is implied by Lemma~\ref{lem:largest eigenvalues}, which guarantees that there is a constant $C>0$ such that $\limsup_{n\to\infty}\lambda_1(\tilde\Sigma_n)\leq C$ holds almost surely. Lastly, the bootstrap counterpart of the limit $(\ts\frac{1}{p}\tr(\Sigma_n),\ts\frac 1p \tr(\Sigma_n^2)) \to (\phi_1,\phi_2)$ is handled by the fact that the moments of $\tilde H_n$ converge in probability to the moments of $H$, which was shown in the proof of Proposition~\ref{prop:delta}. This completes the proof.
\qed

\setcounter{lemma}{0}
\renewcommand{\thelemma}{D.\arabic{lemma}}

\section{Background results}\label{app:background}
\begin{lemma}[\cite{hu2019aos}, Lemma A.1]\label{lem:quadform}
     Let $\xi_1\in\R$ and $\mathbf{u}_1\in\R^p$ satisfy the conditions in Assumption \ref{Data generating model}, and fix any symmetric matrix $M\in \R^{p\times p}$. Then,
    \begin{align}
        \var(\xi_1^2\mathbf{u}_1\ttop M \mathbf{u}_1) = \frac{\E(\xi_1^4)}{p(p+2)}(\tr(M)^2+2\tr(M^2)) - \tr(M)^2.
    \end{align}
\end{lemma}

\begin{lemma}\label{lem:mixedmoments}
Let $\mathbf{x}_{1}, \ldots, \mathbf{x}_{n}$ satisfy the conditions in Assumption \ref{Data generating model}, and let $i,j,k,l$ be four distinct indices in $\{1,\dots,n\}$. Then, 
 \begin{align}
    \E((\mathbf{x}_i\ttop \mathbf{x}_j)^4) \ &= \ 3\left(\frac{\E(\xi_1^4)}{p(p+2)}\right)^2 (\tr(\Sigma_n^2)^2 + 2 \tr(\Sigma_n^4))\label{eqn:quadmix2}\\[0.2cm]
    \E((\mathbf{x}_i\ttop \mathbf{x}_j)^2(\mathbf{x}_i\ttop \mathbf{x}_k)^2) \ &= \ \frac{\E(\xi_1^4)}{p(p+2)} (\tr(\Sigma_n^2)^2 + 2 \tr(\Sigma_n^4))\label{eqn:triplemix2}\\[0.2cm]
    \E((\mathbf{x}_i\ttop \mathbf{x}_j)^2(\mathbf{x}_l\ttop \mathbf{x}_k)^2) \ &= \ \tr(\Sigma^2_n)^2.\label{eqn:doublemix2}
\end{align}

\end{lemma}
\begin{proof}
    Since the observations are i.i.d.~and centered, we have
    \begin{equation}\label{eqn:directcalc}
        \begin{split}
               \E((\mathbf{x}_i\ttop \mathbf{x}_j)^2(\mathbf{x}_l\ttop \mathbf{x}_k)^2) &= (\var(\mathbf{x}_i\ttop \mathbf{x}_j))^2\\[0.2cm]
               &= (\cov(\mathbf{x}_i\ttop \mathbf{x}_j\, ,\,\mathbf{x}_i\ttop \mathbf{x}_j))^2\\
               &=\Big(\sum_{r=1}^p\sum_{s=1}^p \E( \x_{ir}\x_{jr}\x_{is}\x_{js})\Big)^2\\[0.2cm]
               &= \Big(\sum_{r=1}^p\sum_{s=1}^p((\Sigma_n)_{rs})^2\Big)^2,
        \end{split}
    \end{equation}
    which establishes~\eqref{eqn:doublemix2}. The two other assertions can be shown using conditional expectation. For the statement~\eqref{eqn:quadmix2}, we use Lemma~\ref{lem:quadform} to obtain
    \begin{align*}
    \small
        \E((\mathbf{x}_i\ttop \mathbf{x}_j)^4) & \ = \  \E(\E((\mathbf{x}_i\ttop \mathbf{x}_j\mathbf{x}_j\ttop\mathbf{x}_i)^2|\mathbf{x}_j))\\[0.2cm]
        & \ = \ \frac{\E(\xi_1^4)}{p(p+2)} \E\bigg(\ts\Big(\tr(\Sigma_n^{1/2}\mathbf{x}_j\mathbf{x}_j\ttop\Sigma_n^{1/2})\Big)^2+2\tr\Big(\big(\Sigma_n^{1/2}\mathbf{x}_j\mathbf{x}_j\ttop\Sigma_n^{1/2}\big)^2\Big)\bigg) \\[0.2cm]
        & \ = \ 3\ts\frac{\E(\xi_1^4)}{p(p+2)} \E((\mathbf{x}_j\ttop \Sigma_n\mathbf{x}_j)^2)\\[0.2cm]
        &\ = \ 3\left(\frac{\E(\xi_1^4)}{p(p+2)}\right)^2 (\tr(\Sigma_n^2)^2 + 2 \tr(\Sigma_n^4)). 
    \end{align*}
   The argument for~\eqref{eqn:triplemix2} is similar.
\end{proof}

\begin{lemma}\label{lem:4thmomentbound}
Let $\u_1\in\R^p$ be a random vector that is uniformly distributed on the unit sphere, and let $M\in\R^{p\times p}$ be a non-random positive semidefinite matrix with $\lambda_1(M) \lesssim 1$. Then, 
\begin{equation}\label{eqn:4thmomentbound}
    \E\Big| \u_1\ttop M \u_1 -\ts\frac 1p \tr(M)\Big|^4 \lesssim  \ \frac{1}{p^2}.
\end{equation}
\end{lemma}

\proof Due to the orthogonal invariance of $\u_1$, we may work under the assumption that $M$ is diagonal, i.e. $M=\text{diag}(\lambda_1(M),\dots,\lambda_p(M))$. 
Therefore, the quantity $\E\big| \u_1\ttop M \u_1 -\ts\frac 1p \tr(M)\big|^4$ is the same as
\begin{equation}\label{eqn:bigsum}
   \displaystyle\sum_{j_1,j_2,j_3,j_4} 
    \textstyle \lambda_{j_1}(M)\lambda_{j_2}(M)\lambda_{j_3}(M)\lambda_{j_4}(M)
   \,\E\Big(\prod_{l=1}^4 (\u_{1j_l}^2-\ts\frac 1p)\Big).
\end{equation}
Depending on the number of distinct indices among $(j_1,j_2,j_3,j_4)$, the following bounds can be obtained via direct calculation
\begin{equation}
    \bigg|\E\Big(\textstyle\prod_{l=1}^4 (\u_{1j_l}^2-\ts\frac 1p)\Big)\bigg|  \ \lesssim \ \begin{cases} \ts\frac{1}{p^6} \ \ \ \ \text{(4 distinct indices)}\\[0.2cm]
    \ts\frac{1}{p^5} \ \ \ \ \text{(3 distinct indices)}\\[0.2cm]
    \ts\frac{1}{p^4} \ \ \ \ \text{(1 or 2 distinct indices).}
    \end{cases}
\end{equation}
Since the eigenvalues of $M$ are non-negative with $\lambda_1(M)\lesssim 1$, and since  the numbers of terms in~\eqref{eqn:bigsum} involving $k$ distinct indices is $\mathcal{O}(p^k)$, the stated result~\eqref{eqn:4thmomentbound} is proved.\qed

\setcounter{lemma}{0}
\renewcommand{\thelemma}{E.\arabic{lemma}}

\section{Discussion of examples in Section~\ref{sec:setup}}\label{app:examples}
This section provides detailed information related to examples of the random variable $\xi_1^2$ stated in Section~\ref{sec:setup}. We give explicit parameterizations, and we check that the distributions satisfy the conditions in Assumption~\ref{Data generating model}. The only three examples we do not individually cover are the Chi-Squared, Poisson, and Negative-Binomial distributions, because they can be decomposed into sums of independent random variables, and consequently, such examples are covered by the following lemma.

     \begin{lemma}\label{lem:sum}
       Suppose $\xi_1^2= \sum_{j=1}^p z_{1j}^2$ for some independent random variables $z_{11},\dots,z_{1p}$ satisfying
 \small
 \begin{equation}
         \ts\frac{1}{p}\sum_{j=1}^p \E(z_{1j}^2)=1, \quad  \ts\frac{1}{p}\sum_{j=1}^p\var(z_{1j}^2) = \tau+o(1) \text{\quad and  \quad } \displaystyle\max_{1\leq j\leq p}\E|z_{1j}|^{8+2\varepsilon}\,\lesssim\, p^{1+\frac{\varepsilon}{4}}
 \end{equation}
 \normalsize
 as $n\to\infty$, for some fixed constants $\tau\geq 0$ and $\varepsilon >0$ not depending on $n$.
Then, $\xi_1^2$ satisfies the conditions in Assumption~\ref{Data generating model}. 
\end{lemma}
\proof
     It is only necessary to show  $\E\big|(\xi_1^2-\E(\xi_1^2))/\sqrt p\big|^{4+\varepsilon}\lesssim 1$.
   Using Rosenthal's inequality~\citep{Rosenthal} to bound the $L^{4+\varepsilon}$ norm of a sum of independent centered random variables, we have
   \footnotesize
     \begin{equation}
        \begin{split}
         \Big(\E\big|\xi_1^2-\E(\xi_1^2)\big|^{4+\varepsilon}\Big)^{\frac{1}{4+\varepsilon}} & \ = \ \Big(\E\big|\sum_{j=1}^p z_{1j}^2-\E(z_{1j}^2)\big|^{4+\varepsilon}\Big)^{\frac{1}{4+\varepsilon}}\\[0.2cm]
         & \ \lesssim \  \sqrt{\ts\sum_{j=1}^p\var(z_{1j}^2)}+p^{\frac{1}{4+\varepsilon}}\max_{1\leq j\leq p} (\E|z_{1j}^2-\E(z_{1j}^2)|^{4+\varepsilon})^{\frac{1}{4+\varepsilon}}\\[0.2cm]
         & \ \lesssim \sqrt{p} + p^{1/4}\max_{1\leq j\leq p} (\E |z_{1j}|^{8+2\varepsilon})^{\frac{1}{4+\varepsilon}}\\[0.2cm]
         & \ \lesssim \  \sqrt{p},
         \end{split}
    \end{equation}
    \normalsize
          which completes the proof.
     \qed

~\\
        \noindent\textbf{Beta distribution} The Beta$(a,b)$ distribution with parameters $a,b>0$ has a density function that is proportional to $x^{a-1}(1-x)^{b-1}$ for $x\in(0,1)$.

        \begin{lemma}
          If $\beta>0$ is fixed with respect to $n$ and $\xi_1^2 \sim (p+2\beta)$\textup{Beta}$(p/2,\beta)$, then $\xi_1^2$ satisfies the conditions in Assumption \ref{Data generating model}.
        \end{lemma}
        \proof
       The mean and variance are given by $\E(\xi_1^2) = p$ and $\var(\xi_1^2) = \frac{4\beta p}{p+2\beta+2}$, and in particular we have $\var(\xi_1^2)/p \to 0$ as $p\to\infty$. Based on Equation (25.14) in \cite{Kotz:vol2}, the higher moments of $\xi_1^2$ are
        \begin{align*}
            \E(\xi_1^{2k}) = \prod_{j = 0}^{k-1}\Big(p + \ts\frac{4j\beta}{p+2\beta+2j}\Big).
        \end{align*}
        Using the general relationship between central moments to ordinary moments
        \begin{equation}\label{eqn:centralmoments}
        \E\big|\ts\frac{\xi_1^2-p}{\sqrt p}\big|^6 
        \ = \ \frac{1}{p^3} \displaystyle\sum_{j=0}^6 \ts\binom{6}{j}(-1)^{6-j}\E(\xi_1^2)^{6-j}\E(\xi_1^{2j}),
        \end{equation}
        it can be checked that that $\E|\frac{\xi_1^2-p}{\sqrt p}|^6$ is a rational function of $\beta$ that converges pointwise to 0 as $p\to\infty$. This implies the $(4+\varepsilon)$-moment condition in~\eqref{eqn:moment assumption}.\qed
        ~\\
        
    \noindent\textbf{Beta-Prime distribution.} A random variable $W$ is said to follow a Beta-Prime$(a,b)$ distribution with parameters $a,b>0$ if it can be expressed as $W=\frac{U}{1-U}$ with $U\sim$\,Beta$(a,b)$.

         \begin{lemma}
            If $\tau>0$ and $\xi_1^2 \sim$ \textup{Beta-Prime}$\left(\frac{p(1+p+\tau)}{\tau},\frac{1+p+2\tau}{\tau}\right)$,  then $\xi_1^2$ satisfies the conditions in Assumption \ref{Data generating model}.
         \end{lemma}
         \proof
  For any positive integer $k<(1+p+2\tau)/\tau$, the random variable $\xi_1^2$ satisfies the following moment formula~\citep[][\textsection 5]{siegrist2017probability},
   \begin{align}
       \E(\xi_1^{2k}) = \prod_{j=1}^k \frac{p(1+p+\tau)+(j-1)\tau}{1+p+(2-j)\tau}.
   \end{align}
Since $2<(1+p+2\tau)/\tau$, this gives $\E(\xi_1^2)  = p$ and $\E(\xi_1^4) = p(p + \tau)$, which yield $\var\Big(\frac{\xi_1^2-p}{\sqrt p}\Big)=\tau$. Also, using the formula~\eqref{eqn:centralmoments}, it can be checked that $\E|\frac{\xi_1^2-p}{\sqrt p}|^{6}$, viewed as a function of $\tau$, converges pointwise to a polynomial function of $\tau$ as $p\to\infty$. This implies the $(4+\varepsilon)$-moment condition in~\eqref{eqn:moment assumption}.\qed
        ~\\
        ~\\
        \noindent\textbf{Gamma distribution.} For $\alpha, \beta >0$, we parameterize the Gamma$(\alpha, \beta)$ distribution so that its density function is proportional to $x^{\alpha-1}e^{-\beta x}$ when $x\in(0,\infty)$.
      
      \begin{lemma}\label{lem:gamma}
        If $\tau >0$ and $\xi_1^2 \sim \textup{Gamma}(p/\tau, 1/\tau)$, then $\xi_1^2$ satisfies the conditions in Assumption \ref{Data generating model}.
        \end{lemma}
        \proof   The conditions $\E(\xi_1^2)=p$ and $\var\Big(\frac{\xi_1^2-p}{\sqrt p}\Big)=\tau$ follow directly from the stated parameterization. Also, Theorem 2.3 in \cite{boucheron2013concentration} gives 
        \begin{align}\label{eqn:Gammabound}
            \big(\E|\xi_1^2 - p|^6\big)^{1/6} \leq C(\sqrt{\tau p} +\tau)
        \end{align}
        for an absolute constant $C>0$, which implies the $(4+\varepsilon)$-moment condition in~\eqref{eqn:moment assumption}.\qed
         ~\\
         ~\\
\noindent\textbf{Log-Normal distribution.} A positive random variable $Y$ is said to follow a Log-Normal$(\mu,\sigma^2)$ distribution if $\log(Y)\sim N(\mu,\sigma^2)$.
\begin{lemma}
If $\tau > 0$ and $\xi_1^2 \sim$ \textup{Log-Normal}$\left(\!\log(p)-\frac{1}{2}\log\big(1+\frac{\tau}{p}\big), \log\big(1+\frac{\tau}{p}\big)\!\right)$,  then $\xi_1^2$ satisfies the conditions in Assumption \ref{Data generating model}. \  
\end{lemma}
\proof
    Equations (14.8a)-(14.8b) in \cite{Kotz:vol1} show that if $Y \sim \textup{Log-Normal}(\mu,\sigma^2)$, then $\E(Y) = e^{\mu +\sigma^2/2}$ and $\var(Y) = e^{2\mu +\sigma^2}\big(e^{\sigma^2}-1\big)$. Consequently, the stated choice of $\xi_1^2$ satisfies $\E(\xi_1^2)=p$ and $\var(\xi_1^2)/p=\tau$.  Equation (14.8c) in \cite{Kotz:vol1} shows that the central 6th moment of $Y$ is
    \begin{align*}
        \E|Y-\E(Y)|^6 &= e^{6\mu+ 3\sigma^2} \sum_{j=0}^6 (-1)^j \tbinom{6}{j}e^{\sigma^2(6-j)(5-j)/2},
    \end{align*}
    and a direct calculation reveals that $\E|\frac{\xi_1^2-p}{\sqrt p}|^6$ is a polynomial function of $\tau$ whose coefficients are bounded as $p\to\infty$. This ensures that $\xi_1$ satisfies the $(4+\varepsilon)$-moment condition in~\eqref{eqn:moment assumption}.

\section*{Acknowledgements}
The authors thank the reviewers and Associate Editor for their helpful feedback, which significantly improved the paper.

\bibliography{citation}
\bibliographystyle{plainnat_reversed}

\end{document}